\pdfoutput=1
\documentclass[12pt]{amsart}

\usepackage{color}
\usepackage{graphicx}
\usepackage[mathcal, mathscr]{euscript}
\usepackage{bbm}
\usepackage{float}
\usepackage{amsmath,amsthm,amssymb,amsfonts,amscd,epsfig,latexsym,graphicx,textcomp}
\usepackage{tikz-cd}
\usepackage{psfrag}
\usepackage[all]{xy}
\xyoption{pdf}
\usepackage{stackengine}
\usepackage{comment}
\usepackage{hhline}
\usepackage{diagbox}
\usepackage{float}
\usepackage{caption}
\usepackage{eufrak}
\usepackage{slashed}

\usepackage{tikz}

\restylefloat{figure}

\usepackage{thmtools}
\usepackage{thm-restate}

\usepackage{hyperref}

\usepackage{cleveref}

\newtheorem{Theorem}{Theorem}

\newtheorem{theorem}{Theorem}[section]
\newtheorem{lemma}[theorem]{Lemma}
\newtheorem{corollary}[theorem]{Corollary}
\newtheorem{proposition}[theorem]{Proposition}

\newtheorem{conjecture}[theorem]{Conjecture}
\newtheorem{scholium}[theorem]{Scholium}
\newtheorem{question}[theorem]{Question}

\theoremstyle{definition}
\newtheorem{definition}[theorem]{Definition} 
\newtheorem{example}[theorem]{Example} 

\theoremstyle{remark}
\newtheorem{remark}[theorem]{Remark}
\numberwithin{equation}{section}
\newcommand{\ot}{\otimes}
\newcommand{\ra}{\rightarrow}

\newcommand{\BC}{\mathbb{C}}
\newcommand{\BR}{\mathbb{R}}

\DeclareMathOperator{\Ima}{Im}

\def \sign{{\text{sign}}}

\newcommand{\IIDiag}{\raisebox{-0.33\height}{\includegraphics[scale=0.25]{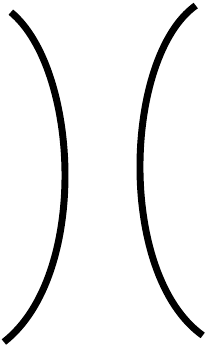}}}

\newcommand{\XDiag}{\raisebox{-0.33\height}{\includegraphics[scale=0.25]{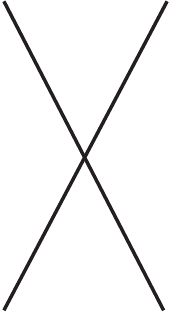}}}
\newcommand{\PMEdgeDiag}{\raisebox{-0.33\height}{\includegraphics[scale=0.25]{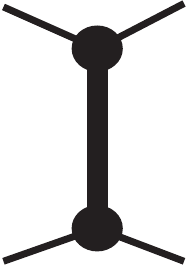}}}
\newcommand{\NegPMEdgeDiag}{\raisebox{-0.33\height}{\includegraphics[scale=0.25]{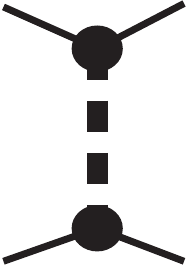}}}

\newcommand{\qdim}{q\!\dim}

\newcommand{\Kthreethree}{\raisebox{-0.33\height}{\includegraphics[scale=0.5]{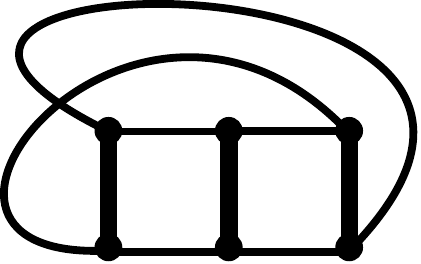}}}
\newcommand{\Kthreethreezero}{\raisebox{-0.33\height}{\includegraphics[scale=0.5]{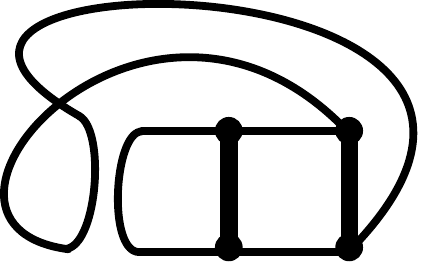}}}
\newcommand{\Kthreethreeone}{\raisebox{-0.33\height}{\includegraphics[scale=0.5]{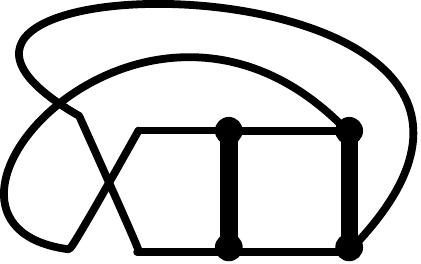}}}
\newcommand{\thetatwohor}{\raisebox{-0.33\height}{\includegraphics[scale=0.5]{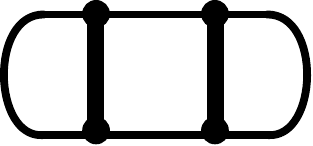}}}
\newcommand{\thetatwovert}{\raisebox{-0.45\height}{\includegraphics[scale=0.5]{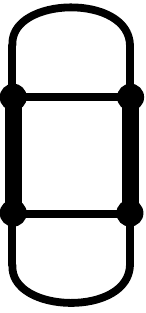}}}

\newcommand{\BZ}{\mathbb{Z}}
\newcommand{\BQ}{\mathbb{Q}}

\newcommand{\BN}{\mathbb{N}}

\newcommand{\ZZ}{{\mathbb Z}}

\newcommand{\NN}{{\mathbb N}}

\def\del{\partial}

\def\del{\partial}

\usepackage{xpatch}
\makeatletter   
\xpatchcmd{\@tocline}
{\hfil\hbox to\@pnumwidth{\@tocpagenum{#7}}\par}
{\ifnum#1<0\hfill\else\dotfill\fi\hbox to\@pnumwidth{\@tocpagenum{#7}}\par}
{}{}
\makeatother    

\makeatletter
 \def\l@subsection{\@tocline{2}{0pt}{4pc}{6pc}{}}
\def\l@subsubsection{\@tocline{3}{0pt}{8pc}{8pc}{}}
 \makeatother

\begin{document}

\title{A topological quantum field theory approach to graph coloring}

\thanks{}

\author{Scott Baldridge}
\address{Department of Mathematics, Louisiana State University\\
Baton Rouge, LA}
\email{baldridge@math.lsu.edu}

\author{Ben McCarty}
\address{Department of Mathematical Sciences, University of Memphis\\
Memphis, TN}
\email{ben.mccarty@memphis.edu}

\date{}

\begin{abstract} In this paper, we use a topological quantum field theory (TQFT) to define families of new homology theories of a $2$-dimensional CW complex of a smooth closed surface.  The dimensions of these homology groups can be used to count the number of ways that each face of the CW complex can be colored with one of $n$ colors so that no two adjacent faces have the same color. We use these homologies to define new invariants of graphs, give new characterizations of well-known polynomial invariants of graphs, and rephrase and offer new approaches to famous conjectures about graph coloring. In particular, we show that the TQFT has the potential to generate $4$-face colorings of a bridgeless planar graph, leading to a constructive approach to the four color theorem.  The TQFT has ramifications for the study of smooth surfaces and provides examples of new types of Frobenius algebras.
\end{abstract}

\maketitle

\section{Introduction}
This paper introduces a topological quantum field theory (TQFT) approach to coloring faces of embedded graphs on surfaces with $n$ colors for any $n\in \BN$.  A $(d+1)$-dimensional TQFT assigns a projective $R$-module of finite type $V_M$ to any closed $d$-dimensional manifold $M$, an $R$-isomorphism between $V_M$ and $V_{M'}$ to any homeomorphism from $M$ to $M'$, and an $R$-homomorphism to any $(d+1)$-dimensional cobordism between $d$-dimensional manifolds \cite{Atiyah}. The category we use has a circle object with morphisms that are cobordisms between circles.  Unlike most $(d+1)$-dimensional TQFTs studied in topology, these cobordisms can be unoriented. The $(1+1)$-dimensional TQFT is then a functor from this category to the category of $R$-modules. For the applications presented here, $R$ will often be taken to be a field, therefore the values of the TQFT live in the category of vector spaces over $R$.

Given a TQFT and certain relations, there is a procedure for constructing homology theories from the TQFT by extending it to the category of chain complexes (cf. \cite{BN3}).  In this paper, we produce two related homology theories using this procedure: the {\em bigraded $n$-color homology} and the {\em filtered $n$-color homology}.  These homology theories are defined for a graph that is the $1$-skeleton of a $2$-dimensional CW complex of a closed smooth surface and encode structural information about the colorings (or states) of the $2$-cells of the surface with $n$ colors (see {\em $n$-face colorings} in \Cref{definition:n-face-coloring}).  Because the homologies are expressing graphical information about the surfaces, we will refer to the CW complexes  as ribbon graphs throughout the paper (cf. \Cref{Def:ribbongraph}).   

\newpage

Our homologies allow us to prove numerous new results as well as reframe several important problems in topology, geometry, graph theory and physics. In particular, we 
\begin{enumerate}
\item introduce two new polynomial invariants of a ribbon graph with a perfect matching, the {\em $n$-color polynomial} and {\em total face color polynomial}, and prove several known and new graph theoretic results using them,
\item show that the two homology theories are related via a spectral sequence, and in doing so prove several new results that parallel recent theorems proven using gauge theory and instanton homology of webs,
\item identify face colorings of a $2$-dimensional CW complex of a surface with harmonic classes of a Laplace operator on a finite dimensional Hilbert space, and show that the space of harmonic classes is isomorphic to the filtered $n$-color homology,
\item explain how the bigraded $n$-color homology categorifies the $n$-color polynomial, which connects computations of the filtered $n$-color homology to the Penrose polynomial,
\item give the first simple topological characterization of the fifty-year-old Penrose polynomial for nonplanar graphs, 
\item express the four-color theorem, cycle double cover conjecture, and Tutte's flow conjectures as questions about $n$-color homology theories, and in particular, show how the spectral sequence can be used to generate $4$-face colorings of a planar graph, i.e. a possible {\em construction-based} approach to  the four color theorem, and 
\item explicitly describe the category of cobordisms, local relations, and functors of the topological quantum field theory used to generate these homologies and polynomials. 
\end{enumerate}

These results constitute the first steps toward the new direction in topology/geometry and graph theory proposed in \cite{BaldCohomology}. We discuss each of these results in the next four subsections, summarized  by Theorems A--G. Each of these theorems build out one of the homologies/polynomials in \Cref{fig:InvariantTable} or connect them together into a coherent whole. In fact, taken together, these theorems tell a story. In the last subsection of the introduction, we briefly describe the proposal in \cite{BaldCohomology} that lead to the homologies in this paper.

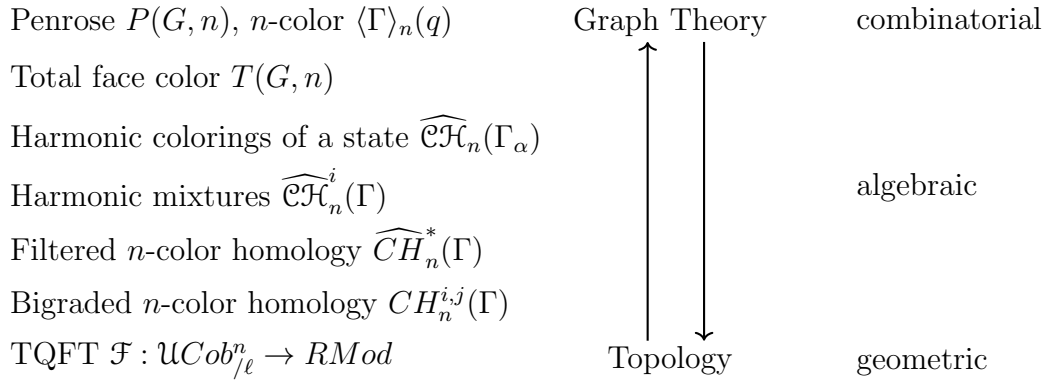
\begin{figure}[h]
$$\begin{tikzpicture}[thick, scale = .75]

\begin{scope}[xshift = -3 cm, yshift = 0 cm]
\draw (-8,2) node[anchor=west] {Penrose $P(G,n)$, $n$-color $\langle \Gamma \rangle_n(q)$};
\draw (-8,1) node[anchor=west]  {Total face color $T(G,n)$};
\draw (-8,0) node[anchor=west]  {Harmonic colorings of a state $\widehat{\mathcal{CH}}_n(\Gamma_\alpha)$};
\draw (-8,-1) node[anchor=west]  {Harmonic mixtures $\widehat{\mathcal{CH}}^i_n(\Gamma)$};
\draw (-8,-2) node[anchor=west] {Filtered $n$-color homology $\widehat{CH}_n^*(\Gamma)$};
\draw (-8,-3) node[anchor=west] {Bigraded $n$-color homology $CH_n^{i,j}(\Gamma)$};
\draw (-8,-4) node[anchor=west] {TQFT $\mathcal{F}: \mathcal{U}Cob_{/\ell}^n \rightarrow RMod$};
\end{scope}

\begin{scope}[xshift = 0.5 cm, yshift = 0 cm]
\draw (0.4,2) node{Graph Theory};
\draw[->] (0,-3.65) --(0,1.65);
\draw[->] (1,1.65) --(1,-3.65);
\draw (0.4,-4) node{Topology};
\end{scope}

\begin{scope}[xshift = -4 cm, yshift = 0.08 cm]
\draw (8,2) node[anchor=west] {combinatorial};
\draw (8,-1) node[anchor=west]  {algebraic};
\draw (8,-4.1) node[anchor=west]  {geometric};

\end{scope}

\end{tikzpicture}$$
\caption{The main structures, homologies, and polynomials defined in this paper and  how they relate to each other and to topology and graph theory.}
\label{fig:InvariantTable}
\end{figure}

\newpage

\subsection{Polynomials related to the Penrose polynomial} A {\em ribbon graph} $\Gamma$ is the closure of a small neighborhood of the $1$-skeleton $G$ of a CW complex of a closed surface $\overline{\Gamma}$ together with the $1$-skeleton  (cf. \Cref{Sec:perfect-matching-graphs} for details), which is equivalent to the CW complex.  A {\em perfect matching graph}, $\Gamma_M$, is a ribbon graph $\Gamma$ of $G(V,E)$ together with a perfect matching $M\subset E$ (see \Cref{Subsec:pm-graphs}). As first described in \cite{BaldCohomology}, choosing a perfect matching allows one to upgrade number-only invariants (e.g. the Penrose Formula or $n$-color number in \Cref{defn:n-color-number}) to a well-defined polynomial called the {\em $n$-color polynomial of a perfect matching graph}. This polynomial is the graded Euler characteristic of the bigraded $n$-color homology.

Briefly, here is the setup to describe the $n$-color polynomial for $n=3$: Given a  perfect matching graph $\Gamma_M$, resolve the perfect matching edges in two different ways inductively using the Kauffman-like bracket, $\langle \PMEdgeDiag \rangle_{\! 3} = \langle \IIDiag  \rangle_{\! 3}  - q\langle \XDiag \rangle_{\! 3}$ and replace immersed circles  when they appear with a Laurent polynomial expression, $\langle \bigcirc \rangle_{\! 3} = q+1+q^{-1}$. For example, the polynomial for the theta graph $\theta$ with standard perfect matching is:

\begin{eqnarray*}
\Big\langle \raisebox{-0.43\height}{\includegraphics[scale=0.30]{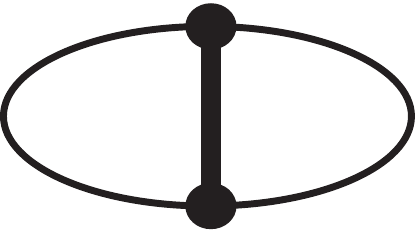}} \Big{\rangle}_{\!\!3} &=& \Big \langle \raisebox{-0.43\height}{\includegraphics[scale=0.20]{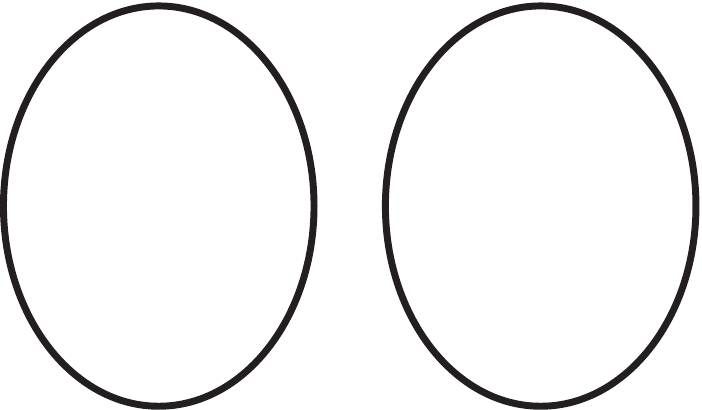}} \Big\rangle_{\!\!3} -q \Big\langle \raisebox{-0.43\height}{\includegraphics[scale=0.20]{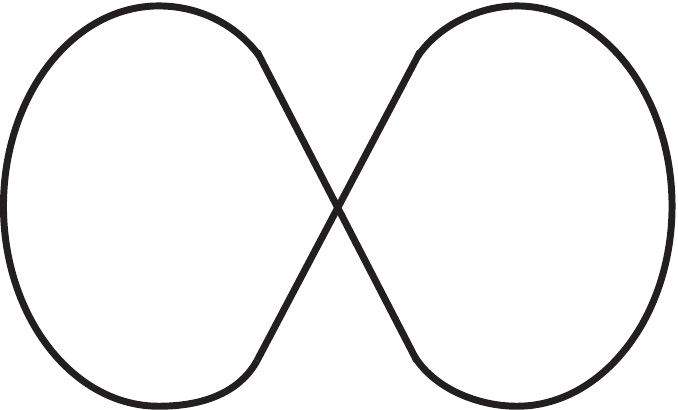}} \Big\rangle_{\!\!3}\\
&=& (q+1+q^{-1})^2 - q(q+1+q^{-1})\\
&=& q+2+2q^{-1}+q^{-2}.
\end{eqnarray*}

This is an invariant of the perfect matching graph and can be used to distinguish two perfect matching graphs in much the same way that knot theorists use invariants to distinguish links.  This is more than just an analogy: it is exactly what was done in \cite{BKR}, where (directed) trivalent perfect matching graphs were identified with link diagrams (for a graph theoretic-only example, see James Oxley's graphs in \cite{BaldCohomology}). 

We do not need to restrict to just invariants of ribbon graph/perfect matching pairs.  There are two ways to get an invariant of a ribbon graph from perfect matching graphs: either sum over all perfect matchings of a given ribbon graph or work with the canonical perfect matching graph associated to a ribbon graph. We work with the latter in this paper.  Given a ribbon graph $\Gamma$ for a graph $G(V,E)$, one obtains a canonical perfect matching graph $\Gamma^\flat_E$, called the blowup (cf. \Cref{def:blow-up-of-a-graph}).  This allows us to define the $n$-color polynomial of the graph itself, i.e., $\langle \Gamma\rangle_{\! n} :=\langle \Gamma^\flat_E \rangle_{\! n}$ (cf. \Cref{defn:n-color-poly-of-a-graph}).  Comparing \Cref{defn:n-color-poly} and \Cref{defn:penrose_poly} one sees that substituting $q=1$ in the $n$-color polynomial  $\langle \Gamma\rangle_n$ yields the value of the Penrose polynomial, $P(\Gamma,n)$, evaluated at $n$.  Thus, invariants of perfect matching graphs lead naturally to invariants, like the Penrose polynomial, of the original ribbon graph.

It is in the context of exploring invariants of perfect matching graphs $\Gamma_M$, or through the blowup $\Gamma_E^\flat$, exploring invariants of the ribbon graph $\Gamma$ itself, that one should understand \Cref{MainTheorem:n-color-polynomial} below. The first four statements are generalizations of well known results about the Penrose polynomial (see \Cref{thm:penrose-polynomial-values}). To setup this theorem, let $K_n$ be the Klein group $K_n:=\BZ_2\times \BZ_2 \times \dots\times \BZ_2$ where $n=2^k$ for some $k\in \BN$, i.e., $|K_n|=n$.  The last two statements are derived from the perfect matching version of \Cref{MainTheorem:Colorings-of-State-Graphs}  and connects nowhere zero $K_n$-flows to $(n-1)$-face and $n$-face colorings in later theorems (see \Cref{MainTheorem:Bigraded-n-color-homology}--\Cref{Theorem:mainthmPenrose}).

\newpage

\begin{Theorem} \label{MainTheorem:n-color-polynomial}
Let $\Gamma_M$ be a perfect matching graph of a connected trivalent graph $G(V,E)$ with perfect matching $M\subset E$. For $n\in\BN$, the $n$-color polynomial, $ \langle \Gamma_M\rangle_{\! n}(q)$, is an invariant of $\Gamma_M$, i.e., it depends only on the ribbon graph $\Gamma$ and perfect matching $M$. Furthermore, for a plane perfect matching graph:
\begin{enumerate}
\item $ \langle \Gamma_M\rangle_{\! 2}(1)= 2^k$ where $k$ is the number of cycles of $G\setminus M$ if the perfect matching is even, and $\langle \Gamma_M\rangle_{\! 2}(1)=0$ otherwise,
\item $\langle \Gamma_M\rangle_{\! 3}(1) = \#\{\mbox{$3$-edge colorings of $G$}\}$,
\item $ \langle \Gamma_M\rangle_{\! n}(1)=0$ if $G$ has a bridge, 
\item for $n=2^k$ with $k\in\BN$, $\langle \Gamma_M\rangle_{\! n-1}(1)>0$ if and only if $\Gamma$ is $n$-face colorable, and
\item for $n=2^k$ with $k\in \BN$, $$ \langle \Gamma_M\rangle_{\! n-1}(1)  \leq \#\{\mbox{nowhere zero $K_n$-flows of $G$}\} \leq \frac{1}{n} \langle \Gamma\rangle_{\! n}(1).$$
\end{enumerate}
\end{Theorem}
\bigskip

\noindent {\em Remark.} Many of the statements in \Cref{MainTheorem:n-color-polynomial} also hold for nonplanar ribbon graphs with perfect matchings (see their proofs where we prove them in greatest generality). However, some do not---see \Cref{fig:K33-is-zero} for an example where Statement (2) fails. In the cases where they do not hold, there is a stronger polynomial invariant introduced in this paper called the {\em total face color polynomial}, $T(G,n)$, described in \Cref{section:totalfacecolorpoly}, that can be used to get similar statements. 
\mbox{}\\

When $n=4$, the number of nonzero $K_4$-flows is equal to the number of $3$-edge colorings of $G$. A later, stronger version of Statement (5) using filtered $3$- and $4$-color homology can then be used to give a new proof of Statement (2). In fact, we state a theorem and make a conjecture (see \Cref{Theorem:mainthm-four-color} and  \Cref{conjecture:total-color-poly-is-an-upper-bound}) about nowhere zero $K_n$-flows using filtered $n$-color homology. In particular, Tutte's $4$-flow conjecture and the cycle double cover conjecture can be rephrased as statements using the total face color polynomial.

One of the  corollaries that follows from \Cref{thm:map-R-theorem} is that many of the polynomial invariants defined in this paper for ribbon graphs can be turned into {\em abstract} graph invariants. Thus, they are theorems about graphs themselves and not just CW complexes of a surface.  For example, we give the first well-defined definition of the Penrose polynomial for an abstract connected trivalent graph $G$ in the literature (see \Cref{def:penrose-poly-abstract-graph}).

\subsection{New homology theories that count $n$-face colorings} At the heart of this paper is the interplay between  the bigraded  and filtered $n$-color homologies through a spectral sequence.   The bigraded $n$-color homology relates the TQFT to the Penrose polynomial.   The filtered $n$-color homology theory is the theory that relates the TQFT to important ideas in graph theory like the four color theorem and the cycle double cover conjecture. The filtered $n$-color homology theory is the one that gives counts of $n$-face colorings of the CW complex of the surface.  One of the strengths of both homology theories is that they are easily computable from a ribbon diagram (cf. \Cref{Def:ribbondiagram}). In fact, we have computer programs for computing them that are available upon request.

\Cref{MainTheorem:Bigraded-n-color-homology} states that the graded Euler characteristic of the bigraded $n$-color homology is the $n$-color polynomial.  Briefly, to define this homology for a perfect matching diagram $\Gamma_M$, we first start with a ribbon diagram with a perfect matching $M$. Replace the perfect matching edges in the diagram  $\PMEdgeDiag$ by either a $0$-smoothing $\IIDiag$ or a $1$-smoothing $\XDiag$ to form a ``hypercube of states,'' which is itself an $|M|$-regular graph with $2^{|M|}$ vertices. To each vertex we associate a {\em state} that corresponds to an element $\alpha=(\alpha_1,\ldots,\alpha_{|M|}))\in \{0,1\}^{|M|}$ and is a set of immersed circles in the plane given by the choice of a $0$- or $1$-smoothing at each perfect matching edge (cf. \Cref{fig:theta_3}).

The hypercube is arranged in columns from the ``all zero smoothings'' state $(0,0,\ldots,0)\in\{0,1\}^{|M|}$ to the  ``all one smoothings'' state $(1,1,\ldots,1)\in\{0,1\}^{|M|}$, where columns are the states that have the same number of $1$-smoothings---let $|\alpha|$ be that number for each state. Next, replace circles in each state by a tensor product of copies of the algebra $V=\mathbbm{k}[x]/(x^n)$ and replace edges of the hypercube with maps between these algebras that depend upon what happens to the circles.  These maps gives rise to a differential $\partial$ between columns that turn the hypercube of states into a bigraded chain complex $(C^{i,j}(\Gamma_M;\mathbbm{k}),\partial)$. This differential preserves the quantum grading (the $j$-grading) and increases the homological grading $i$ by one.  The {\em bigraded $n$-color homology of $\Gamma_M$}, $CH^{i,j}_n(\Gamma_M;\mathbbm{k})$, is then the homology of this complex.

We can now state the second main theorem of this paper:

\bigskip
\begin{Theorem} \label{MainTheorem:Bigraded-n-color-homology}
Let $\Gamma_M$ be a perfect matching graph of a connected trivalent ribbon graph $\Gamma$ with a perfect matching $M$. Let $n\in \BN$ and $\mathbbm{k}$ be a ring in which $\sqrt{n}$ is defined.  Then the bigraded $n$-color homology is an invariant of $\Gamma_M$,  i.e., it depends only on the ribbon graph $\Gamma$ and perfect matching $M$.  Furthermore, the graded Euler characteristic of it is the $n$-color polynomial: 
$$\langle \Gamma_M \rangle_n = \chi_q(CH^{*,*}_n(\Gamma_M;\mathbbm{k})).$$
\end{Theorem}
\bigskip

The bigraded $n$-color homology is a far stronger invariant than the $n$-color polynomial in the same way that Khovanov homology is stronger than the Jones polynomial in knot theory. At the same time, the $n$-color polynomial has more information in it than the Penrose polynomial, which only reports the value of the $n$-color polynomial when evaluated at one (cf. \Cref{prop:EulercharPenrose}). For example, the $n$-color polynomial can be nontrivial for each $n\in\BN$ even for examples where the Penrose polynomial is  zero for all $n$ (see \Cref{example:K33-computation-of-Penrose-poly}).

In \Cref{section:filtered-n-color-homology}, another differential, $\widetilde{\partial}$, is introduced on the same complex. This differential  no longer preserves the quantum grading but does preserve a filtration of the complex.  The differential $\widetilde{\partial}$  anti-commutes with $\partial$ and can be used to form a spectral sequence $(E_r(\Gamma_M),d_r)$.  Here $(E_0,d_0)$ is the original bigraded chain complex $C^{i,j}(\Gamma_M)$ with differential $d_0:=\partial$. The first page is $E_1(\Gamma_M)=CH^{i,j}_n(\Gamma_M)$ with differential $d_1:=\widetilde{\partial}$.  

In \Cref{subsection:colorbasis}, we combine the two differentials into $\widehat{\partial}:=\partial+\widetilde{\partial}$ on the same complex constructed from the hypercube, except we choose $\widehat{V}=\mathbbm{k}[x]/(x^n-1)$ as the algebra.  This is no longer a bigraded theory but a filtered one.  The {\em filtered $n$-color homology of $\Gamma_M$},  $\widehat{CH}^*_n(\Gamma_M;\mathbbm{k})$, is the homology of this complex. Then the $E_\infty$-page of the spectral sequence $(E_r(\Gamma_M),d_r)$ and $\widehat{CH}^*_n(\Gamma_M;\mathbbm{k})$ are the same:

\bigskip
\begin{Theorem}\label{MainTheorem:spectralsequence}
Let $\Gamma_M$ be a perfect matching graph of a connected trivalent ribbon graph $\Gamma$ with perfect matching $M$. Let $n\in \BN$ and $\mathbbm{k}$ be a ring in which $\sqrt{n}$ is defined. Then there exist a spectral sequence $(E_r(\Gamma_M),d_r)$ such that the $E_1$-page is isomorphic to $CH^{*,*}_n(\Gamma_M;\mathbbm{k})$ and the $E_\infty$-page is isomorphic to $\widehat{CH}_n(\Gamma_M;\mathbbm{k})$.
\end{Theorem}
\bigskip

\Cref{MainTheorem:spectralsequence} links the Penrose polynomial and bigraded $n$-color homology to the filtered $n$-color homology. It is at the filtered $n$-color homology level (see \Cref{fig:InvariantTable} again) that  colorings of faces of the CW complex begin to emerge from the theory explicitly.  To see the colors, we need the space of harmonic colorings of a state.  First, we make two simplifying assumptions:  (1) we work with the blowup $\Gamma^\flat_E$ of the ribbon graph $\Gamma$ of the graph $G(E,V)$ and (2)  take $\mathbbm{k}=\BC$. The first assumption ensures that the circles in each state can be identified with faces of a ribbon graph and the second choice ensures that the $n$th roots of unity exist, which are used to define a ``color basis.''

A Hermitian metric for the entire chain complex can be defined after changing the basis to the color basis (\Cref{def:colorbasis}). With this metric and basis, there is an adjoint operator $\widehat{\partial}^*:\widehat{C}^i(\Gamma) \ra \widehat{C}^{i-1}(\Gamma)$ and a Laplacian $\slashed{\Delta}:\widehat{C}^i(\Gamma) \ra \widehat{C}^{i}(\Gamma)$ defined by $\slashed{\Delta} = (\widehat{\partial}+\widehat{\partial}^*)^2$. Hence, one can define the {\em space of harmonic mixtures}, $\widehat{\mathcal{CH}}^i_n(\Gamma)$, and prove a Hodge decomposition theorem for the space $\widehat{C}^i(\Gamma)$ (see \Cref{lem:hodge-decomposition-theorem}). 

We use the term ``harmonic mixture'' for the elements of $\widehat{\mathcal{CH}}^i_n(\Gamma)$ because, in the color basis, elements are linear combinations of colorings on the same state {\em and across different states}  $\Gamma_\alpha$ where $|\alpha|=i$. In other words, it may be possible for a solution to $\slashed{\Delta}(\omega) = 0$ to be represented, say, by $\omega=\omega_\alpha+\omega_{\alpha'}\in \widehat{V}_\alpha \oplus \widehat{V}_{\alpha'} \subset \widehat{C}^i(\Gamma)$ for two  states $\Gamma_\alpha$ and $\Gamma_{\alpha'}$ in the hypercube such that $\partial(\omega_\alpha) \not=0$ and $\partial(\omega_{\alpha'})\not=0$, but $\partial(\omega)=0$. The solution does not live in either state independently, but through a ``mixture of colorings'' on each state.  \Cref{MainTheorem:Colorings-of-State-Graphs} below, which is the key link between topology and graph theory in \Cref{fig:InvariantTable}, rules this possibility out.

To isolate the colorings to each state $\Gamma_\alpha$, we define the {\em space of harmonic colorings of a state}, $\widehat{\mathcal{CH}}_n(\Gamma_\alpha)$, see \Cref{defn:harmonic-coloring-of-a-state}. Note the difference in notation between harmonic mixtures and colorings: the harmonic colorings reside {\em only} on the state.  Hence, for a state $\Gamma_\alpha$ for $|\alpha|=i$, this is the space of elements $\omega\in \widehat{V}_\alpha \subset \widehat{C}^i(\Gamma)$ such that $\widehat{\partial}(\omega)=0$ and $\widehat{\partial}^*(\omega)=0$. The main proof of this paper shows that the subspace
$$\bigoplus_{|\alpha|=i} \widehat{\mathcal{CH}}_n(\Gamma_\alpha) \subset \widehat{\mathcal{CH}}_n^i(\Gamma)$$
is actually the entire space. This is done by proving a ``Poincar\'{e} lemma'' for the homology theory that we suspect will be important in further research (cf. \Cref{prop:Poincare-Lemma}).

The space of harmonic colorings of a state $\Gamma_\alpha$ is generated by $n$-face colors on a ribbon graph of $G$ associated to that state. In \Cref{subsec:face-colorings-of-ribbon-graphs}, we describe a map $\mathcal{R}$ from the set of  states in the hypercube of $\Gamma_E^\flat$ to the set of ribbon graphs of $G$. The ribbon graphs in the image of $\mathcal{R}$ are called {\em state graphs}, which we continue to call $\Gamma_\alpha$. Each state graph determines a CW complex for the closed associated surface $\overline{\Gamma}_\alpha$. \Cref{MainTheorem:Colorings-of-State-Graphs} establishes that the number of $n$-face colorings of the $2$-cells of $\overline{\Gamma}_\alpha$ is  equal to the dimension of the space of harmonic colorings of the state:

\begin{Theorem}\label{MainTheorem:Colorings-of-State-Graphs}
Let $\Gamma$ be a connected ribbon graph of an abstract graph $G(V,E)$. Then the filtered $n$-color homology counts $n$-face colorings on the ribbon graphs associated to the blowup of $\Gamma$:
\begin{enumerate}
\item The $i^{th}$ homology group of the filtered $n$-color homology is isomorphic to the direct sum of the spaces of harmonic colorings of state ribbon graphs $\Gamma_\alpha$ where $|\alpha|=i$:
$$\widehat{CH}^i_n(\Gamma;\BC) \cong \oplus_{|\alpha|=i} \widehat{\mathcal{CH}}_n(\Gamma_\alpha).$$
In particular, if $\Gamma$ is planar, then the space of harmonic colorings of a state ribbon graph $\Gamma_\alpha$ of $\Gamma$ can be nonzero only when $|\alpha|$ is even.
\item The dimension of the space of harmonic colorings of a state graph $\Gamma_\alpha$ is equal to the number of $n$-face colorings of its closed associated surface $\overline{\Gamma}_\alpha$:
$$\dim \widehat{\mathcal{CH}}_n(\Gamma_\alpha)=\#\{\mbox{$n$-face colorings of $\overline{\Gamma}_\alpha$}\}.$$
\item The all-zero state $\Gamma_{\vec{0}}$, i.e., $\alpha=(0,0,\ldots,0)\in \{0,1\}^{|E|}$ of the blowup $\Gamma^\flat_E$ of $\Gamma$, is equivalent to the original ribbon graph $\Gamma$, and therefore
$$\dim \widehat{CH}^0_n(\Gamma;\BC) = \dim  \widehat{\mathcal{CH}}_n(\Gamma_{\vec{0}})=\#\{\mbox{$n$-face colorings of $\overline{\Gamma}$}\}.$$
In particular, if $\Gamma$ is a plane graph of $G$ and $n=4$, then $\dim \widehat{CH}^0_4(\Gamma;\BC)$ counts the number of $4$-face colorings of $\Gamma$. 
\end{enumerate}
\end{Theorem}

At the beginning of this subsection, we started with the bigraded $n$-color homology and ended with the fact that the filtered $n$-color homology ``counts'' face colors of ribbon graphs associated to $G$, including the face colors of the original ribbon graph using Statement (3) of  \Cref{MainTheorem:Colorings-of-State-Graphs}.  Importantly, we will see that the spectral sequence plays the role for which it is designed: many facts about the filtered $n$-color homology are hard to obtain (like showing that it is nontrivial). However, these properties are easily computable in the first approximation to it, i.e., the $E_1$-page or the bigraded $n$-color homology. Whether or not these properties survive to the $E_\infty$-page is then a {\em calculation} in homological algebra.  In the next subsection we show how the spectral sequence given by the TQFT can be applied to the Penrose polynomial to give it a simple description for both planar and nonplanar graphs. In the subsection after we investigate  approaches to  the four color theorem and other conjectures using the spectral sequence.   

\subsection{A complete characterization of the Penrose polynomial when $n>0$} Penrose defined what is now called the ``Penrose polynomial'' $P(G,n)$ of a planar trivalent graph $G$ over fifty years ago in \cite{Penrose}. He was investigating pictorial representations of abstract tensor systems. Penrose's paper, in many ways, is an ``origin story'' for many modern ideas in mathematics and physics, from the Kauffman bracket of the Jones polynomial to spin networks and quantum gravity.  Penrose showed that $P(G,n)$, when evaluated at $n=3$, was equal to the number of $3$-edge colorings of the plane graph. The number of $4$-face colorings is then equal to four times this number by a result that goes back to Tait \cite{Tait}. Hence the Penrose polynomial is reporting important information related to the four color theorem. 

The polynomial was heavily studied and generalized in the years that followed (see references within) and much was discovered about it: \Cref{thm:penrose-polynomial-values} gives a brief summary. The definitions of the polynomial were generally algorithmic, eg. ``create left-right paths based upon a procedure and count up all such paths.''  This helped generalize the polynomial to any-valence graphs but made the polynomial difficult to study and compute. Researchers eventually settled on working with medial graphs and admissible $n$-valuations to describe the Penrose polynomial \cite{Moffat2013}.  Jaeger \cite{Jaeger}, see also Proposition 4 of \cite{Aigner}, showed that $P(G,n)$ equals the admissible $n$-valuations of the medial graph $G'$ when $G$ is planar. Later, Ellis-Monaghan and Moffatt in \cite{EMM} showed that $n$-valuations can be associated to ``partial Petrials'' of the ribbon graph. In this paper, their results follow immediately as a corollary of \Cref{MainTheorem:Colorings-of-State-Graphs}, \Cref{proposition:even-degree-non-zero-n-face-colorings}, and the fact that the graded Euler characteristic of the bigraded $n$-color homology evaluated at $q=1$ is equal to the Penrose polynomial evaluated at $n$.  However, we go further and give a simple characterization of the polynomial: \Cref{Theorem:mainthmPenrose} describes the Penrose polynomial as sums of $n$-face colors of {\em all} ribbon graphs of $G$, where ribbon graphs are considered unique up to ribbon graph equivalence (\Cref{Subsec:ribbongraphs}). This statement is clean in that it avoids  double counting and/or  under counting $n$-face colors of ribbon graphs present in earlier formulations of the polynomial.

First, we set up the theorem.  Let $\Gamma$ be a ribbon graph of a connected trivalent graph $G$. Each ribbon graph of $G$ is either even or odd with respect to $\Gamma$ based upon the number of half-twists that are required to be inserted into its bands to make it equivalent to $\Gamma$ (cf. \Cref{def:even-odd-ribbon-graph}). Also a new result of this paper, we show how to define the Penrose polynomial for any connected abstract trivalent graph $G$ using the notion of a nonnegative ribbon diagram (cf. \Cref{def:penrose-poly-abstract-graph}).  Thus, the Penrose polynomial is a graph invariant, not just a ribbon graph invariant.

\begin{Theorem} \label{Theorem:mainthmPenrose}
Let $\Gamma$ be a nonnegative ribbon graph of a  connected trivalent graph $G$ with trivial automorphism group.  After bifurcating all ribbon graphs of $G$ into even or odd with respect to $\Gamma$,  the Penrose polynomial of the graph $G$, evaluated at $n\in\BN$, is 
$$P(G,n) =  \# \left\{\mbox{\parbox{1.55in}{$n$-face colorings of \\ all  even ribbon graphs with respect to $\Gamma$}}\right\} \ - \ \# \left\{\mbox{\parbox{1.45in}{$n$-face colorings of \\ all odd ribbon graphs with respect to $\Gamma$}}\right\}.$$
Furthermore, if $G$ is planar, then $P(G,n)$ is the total number of $n$-face colorings on all ribbon graphs of $G$.
\end{Theorem}

The conditions that  $G$ has trivial automorphism group or that $G$ is trivalent are not major restrictions when studying the existence of $n$-face colorings of ribbon graph surfaces.  The reason for this is that a ribbon graph can always be blown up to get a trivalent ribbon graph, and one can further blowup the graph at different individual vertices until the resulting graph satisfies both conditions. If the new, blown up, ribbon graph has an $n$-face coloring, then the original ribbon graph did.

Due to the signs in the definition of the Penrose polynomial, it is possible that the polynomial is identically zero when the graph is nonplanar. For example, the Penrose polynomial of the $K_{3,3}$ graph evaluated at $n=3$ is zero even though $K_{3,3}$ has twelve $3$-edge colorings (compare to Statement (2) of \Cref{MainTheorem:n-color-polynomial}).  Mathematicians have tried to modify the definition of the Penrose polynomial to capture these counts again, but these modifications turn out not to be generalizations of the polynomial (cf. \cite{Kauffman}). We found such a generalization, which we discuss next.

There are are two natural ways  to sum the dimensions of the homology groups of the filtered $n$-color homology.  The Penrose polynomial evaluated at $n$ is equal to the Euler characteristic of the filtered $n$-color homology.  One could also take the Poincar\'{e} polynomial of the homology. This is an invariant of the ribbon graph $\Gamma$, and when the graph is trivalent, the evaluation of this Poincar\'{e} polynomial at one is an invariant of the abstract graph $G$ again. Thus, our homology theory allows us to define a new, stronger invariant than the Penrose polynomial: for a connected trivalent graph $G$, the {\em total face color polynomial}, $T(G,n)$, is the sum of the dimensions of the filtered $n$-color homology (see \Cref{definition:totalfacecolorpolynomial}). When $Aut(G)=1$, this polynomial counts the total number of $n$-face colors on {\em all} ribbon graphs of $G$, hence the reason for its name.  It is also the generalization of the Penrose polynomial: when $G$ is a connected trivalent planar graph, $T(G,n)= P(G,n)$. 

While the Penrose polynomial already exists at the level of the bigraded $n$-color homology and is easily computed from the hypercube of states, one needs to compute the filtered $n$-color homologies for several $n$ before the total face color polynomial can be completely determined (cf. \Cref{thm:bound-filtered-n-color-homology-calc}). Once established, however, it replaces the Penrose polynomial in many theorems that take advantage of the properties of the Penrose polynomial.  For example, $T(K_{3,3},3)=12$, which counts the number of $3$-edge colorings again as in Statement (2) of \Cref{MainTheorem:n-color-polynomial}. In particular, we show how the total face color polynomial can be used to rephrase other famous conjectures in graph theory.

\subsection{TQFT approaches to  the four color theorem and other graph coloring conjectures} \label{subsection:intro:famousgraphconjectures}Most proof-attempts (and all successful proofs) of the four color theorem are negative in nature: start with a counterexample and derive a contradiction. What if instead one could  do an algebra computation that generates a set of $4$-face colorings of a planar graph? This is the motivation of \Cref{Theorem:mainthm-four-color}. It offers enticing new ways to give non-computer-aided proofs of the four color theorem:

\begin{Theorem} \label{Theorem:mainthm-four-color}
Let $\Gamma$ be an oriented ribbon graph of a connected graph $G$. 
\begin{enumerate}
\item  The total space $CH^{*,*}_n(\Gamma;\BC)$ is nonzero (even when $G$ has a bridge) and, via the spectral sequence of \Cref{MainTheorem:spectralsequence},
$$\dim CH^{*,*}_n(\Gamma;\BC) \geq \dim \widehat{CH}^*_n(\Gamma;\BC).$$
In particular, there is a distinguished class $\psi(\Gamma) = [x^{n-1}\ot x^{n-1} \ot \cdots \ot x^{n-1}] \in CH^{0,k}(\Gamma;\BC)$ for $k$ a function of the number of faces of $\Gamma$ that is always nonzero.  
\item If $\Gamma$ is a trivalent plane graph and  $n=2^k$ for some $k\in \BN$, then $\dim \widehat{CH}^*_n(\Gamma;\BC)>0$ if and  only if $\Gamma$ is $n$-face colorable.
\item If $G$ is a trivalent graph and $n=2^k$ for some $k\in\BN$, then
$$\frac{1}{n}\cdot T(G,n-1) \leq \# \{\mbox{nowhere zero $K_n$-flows of $G$}\},$$
and, for $n=4$,  $T(G,3) = \#\{\mbox{$3$-edge colorings of $G$}\}$. In particular, for a trivalent graph $G$ with trivial automorphism group, the number of $3$-edge colorings of $G$ is equal to the sum of the counts of the $3$-face colorings on all ribbon graphs of $G$. 
\end{enumerate}
\end{Theorem}

There are two key ideas to a TQFT approach to the four color theorem  in the theorem above.  The first, derived from Statement (1), is whether or not the nonzero distinguished class $\psi(\Gamma)$, thought of as an element of $E_1$-page in the spectral sequence in \Cref{MainTheorem:spectralsequence}, lives to the $E_\infty$-page as a nonzero class, i.e., $[\psi(\Gamma)]_\infty \in\widehat{CH}^*_n(\Gamma;\BC)$. If it lives, then Statement (2)  implies the plane graph $\Gamma$ is $4$-face colorable. From all calculations we have done so far, $\psi(\Gamma)$ for a bridgeless plane graph $\Gamma$ lives to a nonzero linear combination of a basis where each basis element represents an $n$-face coloring of $\Gamma$ (cf. \Cref{subsection:colorbasis} for a discussion of the coloring basis). Furthermore, the class it lives to is predictable in the sense that each coefficient of the linear combination is some function of $n$ and the specific coloring basis element.  Also, this works in general in our computed examples, not just when $\Gamma$ is a plane graph and $n=4$, but for any oriented ribbon graph and for any $n\in \BN$. Thus, we conjecture: a ribbon graph $\Gamma$ supports $n$-face colorings if and only if $\psi(\Gamma)$ lives to a nonzero class on the $E_\infty$-page (cf. \Cref{conjecture:Psi-lives-to-E-infinity-page}). If this conjecture is true, and $\Gamma$ supports $n$-face colorings, then computing $d_1[\psi(\Gamma)]_1 = 0$, $d_2[\psi(\Gamma)]_2=0$, $d_3[\psi(\Gamma)]_3=0$, etc.  generates a set of $n$-face colorings of $\Gamma$ by the $E_\infty$-page, i.e., once you know the maps that make up the differentials and you have the class $\psi(\Gamma)$, you can ignore that  the  pages are related to a ribbon graph or that the final result is related to colorings---it is simply a calculation in homological algebra to produce colorings on a ribbon graph. Hence, if the conjecture is true,  this would give a  constructible proof of the four color theorem rather than one based upon negating it.

The second TQFT approach to the four color theorem, derived from Statement (2), uses the entire homology of $CH^{*,*}_4(\Gamma;\BC)$. Statement (2)  implies that if any nonzero class in any homological grading (not just $\psi(\Gamma)$ in degree zero) lives to the $E_\infty$-page, then there exists $4$-face colorings on all plane ribbon graphs of $G$.  Hence, a ``counterexample'' to the four color theorem would be a bridgeless trivalent plane graph $\Gamma$ that, for some $r>1$, has the property that the $E_{r-1}$-page is nontrival but the $E_r$-page and higher pages vanish completely.  Then \Cref{MainTheorem:Colorings-of-State-Graphs} together with \Cref{thm:map-R-theorem}  imply that there are no $4$-face colorings on {\em any} of the ribbon graphs of $G$, not just the planar ribbon graphs.  However, unlike a statement about trying to color ribbon graphs of a planar graph $G$, which is a ribbon-graph-by-ribbon-graph proposition, the vanishing of the entire homology of the $E_r$-page requires that all nontrivial classes in $E_{r-1}$ be exact with respect to $d_{r-1}$. This, in turn,  becomes a statement about the interplay of homology classes associated to different {\em sets} of ribbon graphs and how they interact in the hypercube of states, i.e., the TQFT comes with far more structure than just counting  $n$-face colorings of individual ribbon graphs of a graph $G$. Therefore, one must look for ways this rich structure can force the spectral sequence to vanish at the $E_r$-page.  One obvious way is if all ribbon graphs of $G$ shared the same property, like a bridge (cf. \Cref{conj:G-has-a-bridge-implies-E2-vanishes}).

Note that, even for $k=2$, Statement (2) does not follow from other theorems in this paper. In fact, Statement (2) gives a proof based upon TQFT to Aigner's \Cref{thm:4-color-number-positive}. Statement (2) is particular to {\em only} $n=2^k$ because it uses the Klein group $K_n=\BZ_2\times \BZ_2\times\cdots\times \BZ_2$ group explicitly. There are no corresponding statements for any other $n$ without assuming the existence of, say, an oriented cycle double cover.

Statement (3) of \Cref{Theorem:mainthm-four-color} describes a relationship between the filtered $n$-color homology and flows on a graph in a way that allows Tutte's $4$-flow conjecture (see \Cref{prop:Tutte-4-flow-equivalence}) and the cycle double cover conjecture (see \Cref{conjecture:total-color-poly-is-an-upper-bound}) to be restated and understood in a TQFT context. It also generalizes many of the statements of \Cref{MainTheorem:n-color-polynomial} to nonplanar graphs. Integral to the proof of this statement is again the  group structure of the Klein group.  In terms of only graph theory, the last part of Statement (3) hints at possible ways to generalize this count of $3$-edge colorings when the automorphism group is nontrivial to actions on the set of state graphs that takes state graphs to equivalent state graphs.

Finally, Statement (1) was inspired by our attempt to understand papers \cite{KM3, KM2, KM1} in the context of \cite{BaldCohomology}.  In those papers, Kronheimer and Mrowka defined instanton homology $J^\sharp(K)$ and $J^\sharp(K;\mathcal{L})$ for a web $K$ using gauge theory (a web is a trivalent graph embedded in $\BR^3$).  We briefly discuss their results in \Cref{section:comparisionofhomologies}, see   \Cref{eqn:JsharpleqJsharpL} and the three facts after it. The inequality in \Cref{eqn:JsharpleqJsharpL} and three facts are their intriguing attempt at a gauge theory proof of the four color theorem.  We have a similar inequality and three facts: Statement (1) together with \Cref{Theorem:mainthmPenrose} when $n=3$ is our version of their results. (The bigraded $3$-color homology is like their $J^\sharp(K)$ and the filtered $3$-color homology is like their  $J^\sharp(K;\mathcal{L})$.) However, these parallel facts turn out to be the only similarities between the two theories so far. There are many important differences between their theory and ours, which we highlight in \Cref{section:comparisionofhomologies}.  

\subsection{The first application of an unoriented TQFT to colorings of graphs} Unoriented (1+1)-dimensional TQFTs over an $R$-algebra $V$ were defined and classified by Turaev and Turner in \cite{TT}. In that paper, they modified the definition of a topological quantum field theory of Atiyah \cite{Atiyah} for oriented cobordisms to the unoriented case and showed that all unoriented (1+1)-dimensional TQFT over $R$ are in bijective correspondence with isomorphism classes of extended Frobenius algebras over $R$ (cf. Proposition 2.9 in \cite{TT}). They then applied unoriented TQFTs to knots and links.

In this paper we show that unoriented TQFTs provide the correct setting for exploring TQFTs of graphs and colorings of graphs. In particular, while knots and links are limited to Frobenius algebras of dimension two (via the ``$4Tu$'' relation), we show that there is a $(1+1)$-dimensional unoriented TQFT for certain extended Frobenius algebras of any positive dimension. 

To setup up \Cref{MainTheorem:TQFT}, we assume familiarity with the ideas presented in Bar-Natan \cite{BN3}.  Here are the basic notions: Given a ribbon diagram $\Gamma$ of a graph $G$, build the hypercube of resolutions out of the $0$- and $1$-smoothings of edges in $G$ using the blowup of $\Gamma$.   The directed edges of the hypercube generate certain cobordisms (with appropriate signs attached to each cobordism) from circles in the initial state to circles in the target state. The entire cube of sets of cobordisms is formally summed into a complex, called a {\em geometric complex}. The geometric complex will be denoted by $[[\Gamma]]$ in this paper. A well-written description of this type of process can be found in Sections 2 and 3 of \cite{BN3}. 

For each $n\in\BN$, the geometric category from which the bigraded and filtered $n$-color homology of ribbon graphs in this paper can be derived  is $\mbox{Kom}_{/h}(\mbox{Mat}(\mathcal{UC}ob^n_{/l}))$. This stands for the category of complexes ($\mbox{Kom})$, up to homotopy ($\mbox{}_{/h}$), built from columns and matrices of objects and morphisms respectively ($\mbox{Mat}$) taken from $\mathcal{UC}ob^n_{/l}$. The space $\mathcal{UC}ob^n_{/l}$ is the category of $2$-dimensional non-orientable cobordisms (morphisms) between one dimensional circles (objects), formally summed over a ground ring where $\sqrt{n}$ exists, modulo the following local ($\mbox{}_{/l})$ relations:

\begin{equation}
\label{eq:neckcutting}
\mbox{The Neck Cutting (NC) relation: \ } \begin{minipage}[c]{3in} \includegraphics[scale=0.65]{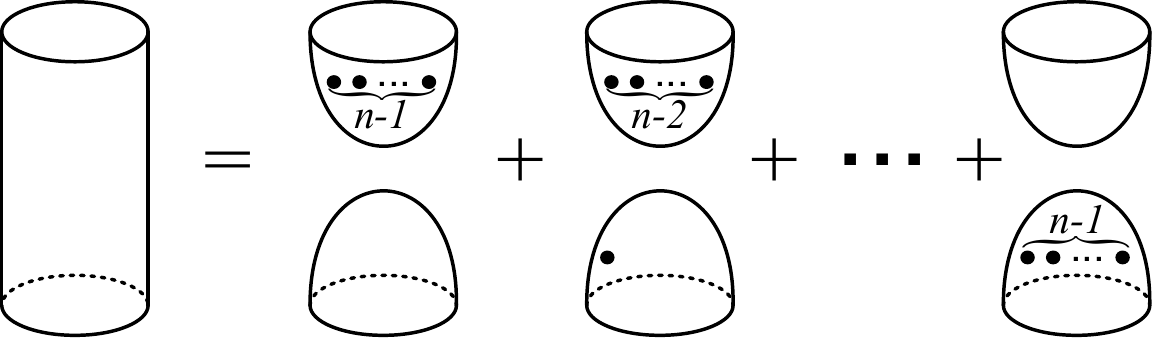}\end{minipage}
\end{equation}

\begin{equation}
\label{eq:S-relation}
\mbox{The S relation: \ } \begin{minipage}[c]{.75in} \includegraphics[scale=0.60]{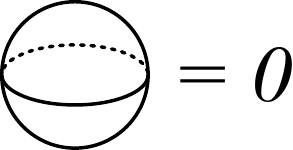} \end{minipage} \mbox{when $n>1$ \hspace{2.64in}}
\end{equation}

\begin{equation}
\label{eq:T-relation}
\mbox{The T relation: \ } \begin{minipage}[c]{.75in} \includegraphics[scale=0.55]{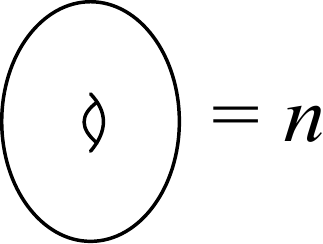} \end{minipage}, \mbox{ \ which implies \ }\begin{minipage}[c]{2.2in} \includegraphics[scale=0.25]{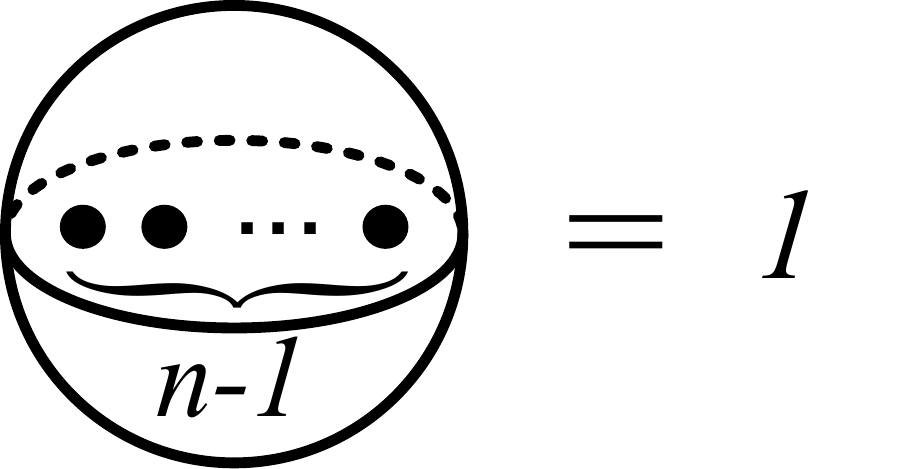} \end{minipage}
\end{equation}

\noindent In the pictures above, each dot ($\bullet$) represents a connect sum with one copy of $\BR P^2$. The names of these relations are based upon the $NC$, $S$, and $T$ relations of Section 4 of \cite{BN3}. For example, it is well known that a similarly-defined neck cutting relation in knot theory is equivalent to the $4Tu$ relation when $2$ is invertible in the ground ring and $n=2$.  Also, by applying the neck cutting relation once to the torus, one gets $n$ copies of $\sharp (n-1)\BR P^2$, which shows $\sharp (n-1)\BR P^2=1$. In our theory,  $n$ is the dimension of the  algebra, hence it generalizes the $T=2$ relation in knot theory, which uses a two dimensional algebra. 

There is one more relation that can be imposed. It is not part of the local relations, but we are free to impose it when necessary---it can be thought of as a free parameter:

\begin{equation}
\label{eq:U-relation}
\mbox{The U relation: \ } \begin{minipage}[c]{1in} \includegraphics[scale=0.30]{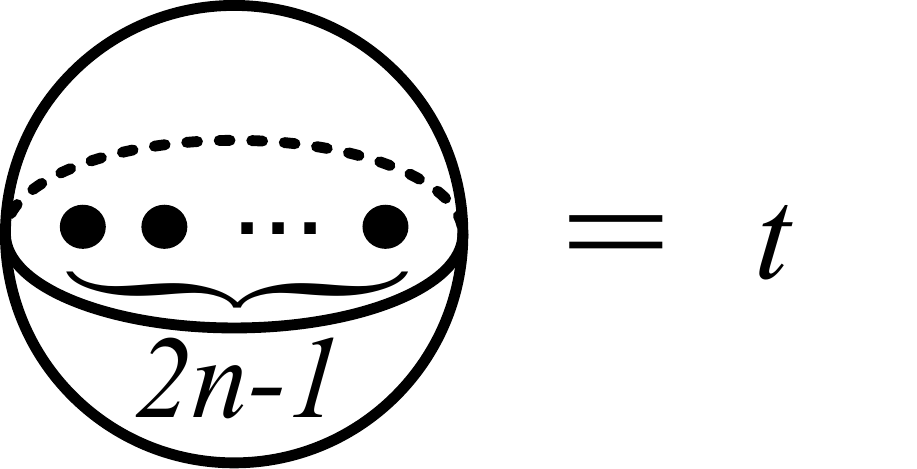} \end{minipage} \mbox{ for some $t\in \BC$ (usually $0\leq t \leq 1$) \ \ \ \ \ \ } 
\end{equation}

\noindent The $U$ relation plays a similar role as the genus $3$ surface does in \cite{BN3}.

This unoriented $(1+1)$-dimensional TQFT corresponds to an algebra over $R$ by Proposition 2.9 of \cite{TT}.  The ground ring will be taken to be $R=\BC$ in what follows. For $0  \leq t \leq 1$,  let $$V_t=\BC[x]/(x^n-t\cdot 1_{V_t})$$ be the (universal) Frobenius algebra with multiplication $m:V_t \ot V_t \ra V_t$ given by multiplication in $V_t$, $\iota:\BC \ra V_t$  given by $1_\BC \mapsto 1_{V_t}$, and $\epsilon:V_t \ra \BC$ given by $\epsilon(x^{n-1}) = 1_\BC$ and zero elsewhere. Following  Definition 2.5 of \cite{TT}, when $n=3$, take the involution $\phi:V_t \ra V_t$ to be the identity map and let the element $\theta=\sqrt{n}x$. This turns $(V_t, m,\iota,\epsilon, \phi, \theta)$ into an extended Frobenius algebra (cf. Example 2.6 of \cite{TT}).  When $n\not=3$, we define a new algebra, a {\em hyperextended  Frobenius algebra}, which is a generalization of an extended Frobenius algebra that retains its features. See \Cref{subsection:facts-about-UCOB} for details. 

The parameter $t$ determines the homology: the bigraded $n$-color homology of this paper corresponds to setting $t=0$ in the $U$ relation and the filtered $n$-color homology corresponds to setting $t=1$. 
\Cref{subsection:differential,subsection:defintion-of-tilde-del,subsection:colorbasis} describe implicitly how to take a ribbon diagram $\Gamma$ of an abstract graph $G$ to a geometric complex $[[\Gamma]] \in \mbox{Kom}_{/h}(\mbox{Mat}(\mathcal{UC}ob^n_{/l}))$.   In \Cref{subsection:geometricchaincomplex} we show how to do this explicitly. 
The category of complexes can also be turned into a graded category by defining gradings on $\mathcal{UC}ob^n$. The differentials of the complex preserve this grading when setting $t=0$ in the $U$ relation.   

In \Cref{subsection:functorFandproof}, we define a TQFT $\mathcal{F}:\mathcal{UC}ob^n \ra \BC\mbox{Mod}$ in the sense of Atiyah's definition described above, where $\BC\mbox{Mod}$ is the category of $\BC$-modules,  and show that it satisfies the local relations and preserves degrees. Therefore it extends to a functor from graded geometric complexes to graded complexes of $\BC\mbox{Mod}$. It is in this last category that one can compute homology: $H(\mathcal{F}([[\Gamma]]))$.

We are now ready to state the last of the main theorems. This is the theorem that  explicitly links the $n$-color homologies defined in this paper to TQFTs.


\begin{Theorem} \label{MainTheorem:TQFT} Let $\Gamma$ be a ribbon graph of $G(E,V)$ and $n>1$. The isomorphism class of the complex $[[\Gamma]]$ regarded in $\mbox{Kom}_{/h}(\mbox{Mat}(\mathcal{UC}ob^n_{/l}))$ is an invariant of the ribbon graph $\Gamma$ up to chain homotopy of geometric chain complexes. Furthermore, the functor $\mathcal{F}$ takes $[[\Gamma]]$ to a complex of modules whose homology is equal to the $n$-color homologies, that is:
$$H(\mathcal{F}([[\Gamma]])) = CH^{*,*}_n(\Gamma;\BC)$$
when $t$ is set to $0$ in the $U$ relation, and
$$H(\mathcal{F}([[\Gamma]])) = \widehat{CH}^*_n(\Gamma;\BC)$$
when $t$ is set to $1$ in the $U$ relation.
\end{Theorem}

We develop \Cref{MainTheorem:TQFT} at the end of this paper, working instead with the hyperextended Frobenius algebras (first $V_0$, then $V_1$)  explicitly throughout to develop the homology theories.  A category theorist may wonder why we did not take Bar-Natan's approach from the start and simply apply his machinery together with the unoriented $(1+1)$-dimensional TQFT developed by Turaev-Turner to come up with the homology invariant. The most important reason for not doing so is that our homologies do not fit perfectly into their machinery; one cannot simply apply already-defined functors to off-the-shelf theories, especially when $n$ is even. Therefore, in a real sense, the calculations leading up to \Cref{MainTheorem:TQFT} are making the case that hyperextended Frobenius algebras with counital and shifted comultiplications are interesting TQFT functors to investigate.  The second reason is that one must still carefully define the complexes, gradings, differentials, etc.  to define and compute the homologies. It is straightforward to express these concepts using algebras and vector spaces. Third, we still needed to rely heavily on the algebra to show that the filtered $n$-color homology is counting the number of $n$-face colorings of a ribbon graph and its state graphs.  This cannot be seen directly from the TQFTs as of yet. Fourth, we do not have to prove Reidemeister invariance as one does in knot theory. Those proofs turn out to be pictorial in \cite{BN3} using the TQFT, which is a {\em major} enticement for a category theoretic approach in the case of knot theory.

There is also a stylistic reason for approaching our homology theories from a computational-versus-categorical perspective: mathematicians who are not steeped in working with categories may not appreciate the beauty of working with such theories. We did not want to create any inadvertent discouragements to working with our new homology theories. These homologies are fun to calculate and our paper takes this perspective from the onset! The category perspective is shown at the end as a summative section with the hope that it may entice the uninitiated in TQFTs  into thinking about graphs in these terms. On the other hand, we encourage  researchers interested in ``categorification'' to read \Cref{section:UnorientedTQFTTheory} first to get a feel for the TQFT perspective of earlier theorems and proofs.

There are good reasons for topologists to investigate the TQFT of ribbon graphs from a purely topological perspective. The homology theories presented in this paper are induced from unoriented $(1+1)$-dimensional TQFT---circles mapping to circles. However, upon closer inspection of the hypercube of states, the states in the hypercube are not just circles but are actually ribbon graph surfaces with boundary (see state graphs in \Cref{subsec:face-colorings-of-ribbon-graphs}). Thus, the ``true'' cobordisms (morphisms) between two states in the hypercube of states are cobordisms/$3$-dimensional foams between 2-dimensional surfaces with boundaries (the objects), of which we are taking the ``boundary components'' parts of the objects and morphisms to get the $(1+1)$-dimensional theory.  The actual theory appears to be some form of a $(2+1)$-dimensional TQFT.  

Graph theorists have other reasons for studying TQFTs of graphs.  One of the reasons that topologists investigate TQFTs is that they have excellent functorial and composition properties. For example, the main point of \cite{BN3} was that TQFTs can be used to construct invariants of tangles.  In graph theory,  configurations play the analogous role of tangles in knot theory. Hence, our TQFT appears to be an important environment for study reducible configurations and unavoidable sets of the four color theorem. Future research will explore this domain.

\subsection{Background to the homologies in this paper}

In \cite{BaldCohomology}, the first author showed how to turn several number invariants of trivalent ribbon graphs into polynomial invariants.  These number invariants were  based upon abstract tensor systems  of Penrose \cite{Penrose} and were different ways to count the number of 3-edge colorings of a planar graph. The polynomial invariants of \cite{BaldCohomology} are stronger invariants than the number invariants in the sense that each number invariant is recovered from its associated polynomial invariant by evaluating the polynomial at one. 

That paper then went on to show that one of these polynomials, the {\em $2$-factor polynomial}, could be ``categorified'' into a bigraded homology theory whose graded Euler characteristic is the $2$-factor polynomial. This polynomial invariant and homology theory fits nicely into a suite of recent gauge theoretic invariants for trivalent graphs (cf. \cite{KM3, KM2, KM1, KR, RW}), but has the added benefit of being defined using a Kauffman-like bracket.  In fact, one can think of the $2$-factor polynomial and homology theory as the ``ribbon graph equivalent'' of the Jones polynomial and Khovanov homology of knot theory (cf. \cite{BKR}). Because of this direct link to knot theory, one can port theorems in knot theory to trivalent ribbon graphs that cannot be realized (as of yet) in the suite of other gauge theoretic theories. 

One of the main themes of \cite{BaldCohomology} was that there should be TQFTs for ribbon graphs for the other polynomial invariants described in that paper, that is, the homology theory for the $2$-factor polynomial was posited to only be an important {\em first example} of such theories. In this paper, we further fulfill this new direction in topology and graph theory by generalizing the $3$-color polynomial described in that paper to the  $n$-color polynomial and show how to categorify it to the bigraded and filtered $n$-color homologies. Our next paper will categorify another polynomial from that paper, called the {\em vertex polynomial}, using the TQFT of this paper.  The vertex polynomial is interesting because it counts the number of perfect matchings of a graph. Thus, all of these new homology theories should be viewed as having their origin in \cite{BaldCohomology}.

\bigskip

\subsection{Outline of this paper} The paper follows a straight path through Theorems A--G described above. After a preliminary section on definitions and notations used in this paper, each subsequent section addresses one of the main theorems above in the same order it was presented in the introduction. Since the audience for this paper ranges over many fields, i.e., topologists, graph theorists, representational theorists, etc., with distinctly different backgrounds, we have written the paper to be reasonably self-contained. We do not expect topologists to know specialized facts from graph theory or have a detailed knowledge of its literature, or vice versa.

\tableofcontents

\section{Preliminaries}\label{Sec:perfect-matching-graphs}
In this section we introduce the notion of a perfect matching graph, an equivalence class of decorated trivalent ribbon graphs. A {\em plane graph} $\Gamma$ is a specific embedding, $i:G\ra S^2$, of a connected planar graph $G$ into the sphere. Ribbon graphs are the natural generalizations of plane graphs in that they allow for non-planar embeddings of graphs into surfaces of any genus while  retaining a key aspect of plane graphs: $S^2 \setminus i(G)$ is a set of disks.  A ribbon graph decorated with a perfect matching is called a perfect matching graph (cf. \cite{BaldCohomology}).  This section describes how perfect matching graphs correspond to an equivalence class of immersed graphs in the plane  with thickened edges for the perfect matching edges. We will need these diagrams to define the homology theories.  Full details of the relationship between oriented ribbon graphs and diagrams can be found in \cite{BKR}.

In this paper, an abstract graph $G(V,E)$ is often thought of as a $1$-dimensional CW complex by identifying vertices of $V$ with points and edges with segments that are glued to their coincident vertices.  Also, all graphs are {\em multigraphs}, which are allowed to have loops (edges with a single incident vertex) and multiple edges incident to the same two distinct vertices. Finally, ``vertex-free'' edges are allowed, i.e., circles.

\subsection{Ribbon graphs}\label{Subsec:ribbongraphs}
A perfect matching graph is an equivalence class of trivalent graphs with extra structure. One of these structures is that of a ribbon graph. For a detailed introduction to ribbon graphs see \cite[Section 1.1.4]{Moffat2013}. 
\begin{definition}\label{Def:ribbongraph}
A {\em ribbon graph of a graph $G$} is an embedding $i:G\ra \Gamma$ where $G$ is thought of as a $1$-dimensional CW complex and $\Gamma$ is a surface with boundary where $\Gamma$ deformation retracts onto $i(G)$.  We say that \( G \) is the \emph{underlying graph} of \( \Gamma \), and that \( \Gamma \) is \emph{the surface associated to} the ribbon graph. 
\end{definition}

We will often refer to the ribbon graph simply by $\Gamma$ and think of $\Gamma$ as a surface with an embedded graph $G$. An orientation of a ribbon graph, if one exists, is an orientation of the surface. Let \( \overline{\Gamma} \) denote the closed smooth surface obtained by attaching discs to the boundary of $\Gamma$. The embedding of $G$ into the surface \( \overline{\Gamma} \) is known as a {\em 2-cell embedding}.\footnote{Such an embedding is also known as a {\em cellular embedding} or {\em cellular map}.}

Let \( \Gamma_1 \) and \( \Gamma_2  \) be ribbon graphs.  We say that \( \Gamma_1\) and \( \Gamma_2 \) are \emph{equivalent} ribbon graphs if there is a homeomorphism \( f : \overline{\Gamma}_1 \rightarrow \overline{\Gamma}_2 \)  that induces an isomorphism from  $G_1$  to $G_2$. Thus, one can define the {\em genus of a ribbon graph $\Gamma$} to be the genus of the associated closed smooth surface $\overline{\Gamma}$.  

\begin{definition}\label{definition:n-face-coloring}
An {\em $n$-face coloring of a ribbon graph $\Gamma$ (or $\overline{\Gamma}$)} is a choice of one of $n$ different colors (or more generally, labels) for each attaching disk of $\overline{\Gamma}$ such that no two disks adjacent to the same edge have the same color.
\end{definition}

Henceforth, unless otherwise noted, all ribbon graphs are assumed to have connected (often trivalent) underlying graphs. All trivalent graphs are assumed to possess at least one perfect matching. 

\begin{remark}
Initially, all ribbon graphs in this paper will be oriented, that is, the closed associated surface $\overline{\Gamma}$ to a ribbon graph $\Gamma$ is orientable with an orientation. The reader may have noticed that our main theorems, except for \Cref{Theorem:mainthm-four-color}, are stated for and are true for both orientable and non-orientable ribbon graphs.  In \Cref{subsection:non-orientable-surfaces}, we will show that all theorems in this paper for oriented ribbon graphs are also true for non-orientable ribbon graphs after making a simple modification to the definitions, see  \Cref{theorem:all-ribbon-graphs-theorem}. We take this approach since we loose nothing by initially restricting to oriented ribbon graphs while at the same time avoiding unnecessarily convoluted proofs that would arise from having to keep track of which formulas to apply in which cases. It is much easier to explain how to modify all the proofs at once and at the end than prove the most general form of each theorem from the beginning.
\end{remark}

Ribbon graphs get their name from the topological construction of attaching bands (ribbons) to disks.  Given a graph \( G \), a cyclic ordering of the edges at every vertex determines a ribbon graph \( \Gamma \). This ribbon graph is obtained by taking a disk for every vertex of \( G \), and attaching bands as prescribed by the edges and their cyclic ordering. Half twists may be added to the bands, provided that the resulting surface is equivalent to the original surface. Thus, the vertices (edges) of \( G \) are in bijection with the discs (bands) of \( \Gamma \), and we shall not distinguish between them, referring to \emph{vertices} and \emph{edges} of \( \Gamma \).

\Cref{Fig:ribbons} shows that two distinct ribbon graphs may have the same underlying abstract graph. These ribbon graphs are distinguished by the number of boundary components of their associated surfaces. The ribbon graph on the left of \Cref{Fig:ribbons} is planar, that is, the associated closed surface $\overline{\Gamma}_1$ to $\Gamma_1$ is a $2$-sphere. We will continue to call genus zero ribbon graphs  {\em plane graphs} or {\em plane ribbon graphs} when the context is clear.

\begin{figure}
\includegraphics[scale=0.75]{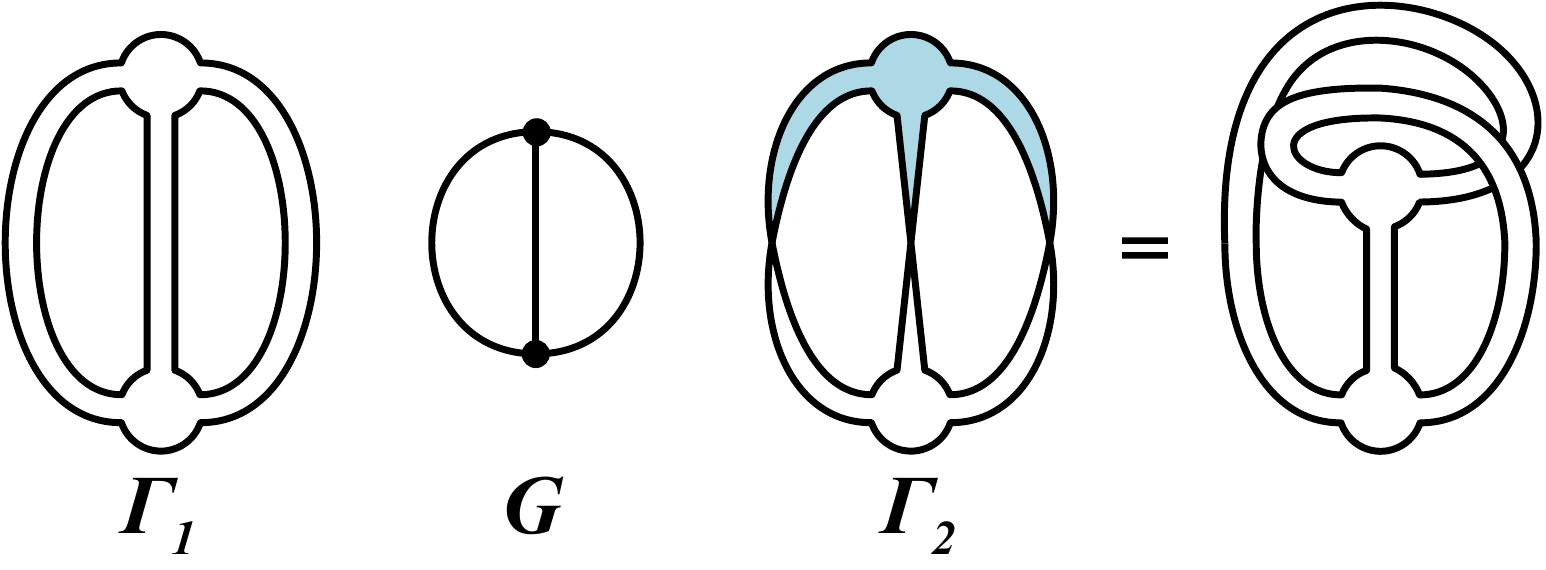}
\caption{Distinct ribbon graphs \( \Gamma_1 \) and \( \Gamma_2 \) with the same underlying graph \( G \).}
\label{Fig:ribbons}
\end{figure}

In this paper,  ribbon graphs are represented by the following diagrams (see \Cref{Fig:ribbondiagram}).

\begin{definition}[Ribbon diagram]\label{Def:ribbondiagram}
A \emph{ribbon diagram} is a graph drawn in the plane (with possible intersections between its edges), with vertices decorated by circular regions, \raisebox{-6pt}{\includegraphics{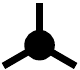}}, to distinguish between vertices and edge intersections. A cyclic ordering of the edges at a vertex is given implicitly by such a diagram, i.e.,\ it is given by their  ordering in the plane.
\end{definition}

\begin{figure}
\includegraphics{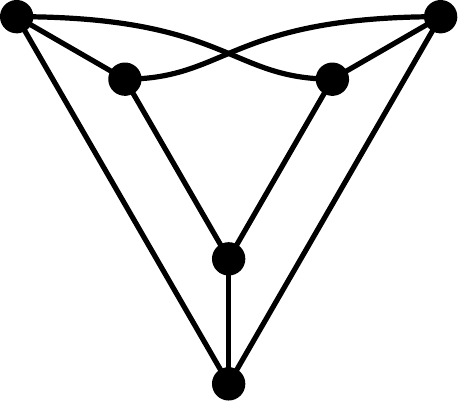}
\caption{A ribbon diagram of a $K_{3,3}$ ribbon graph. Note that the closed oriented surface associated to the ribbon graph is a torus.}
\label{Fig:ribbondiagram}
\end{figure}

A ribbon diagram can be used to construct an oriented ribbon graph using disks for vertices and bands for edges.  If \( \Gamma \) is obtained from a ribbon diagram, \( D \), in this manner we say that \( D \) \emph{represents} \( \Gamma \). 

\begin{proposition}[Baldridge, Kauffman, Rushworth, \cite{BKR}]\label{Prop:ribbonrep}
Every oriented ribbon graph is represented  by a ribbon diagram.
\end{proposition}

In order for ribbon diagrams to faithfully represent ribbon graphs up to equivalence, the following diagrammatic moves can be used to move between two distinct ribbon diagrams of the same ribbon graph.

\begin{definition}\label{Def:ribbonmoves}
The following moves on ribbon diagrams are known as the \emph{ribbon  moves}:
\begin{center}
\includegraphics[scale=0.9]{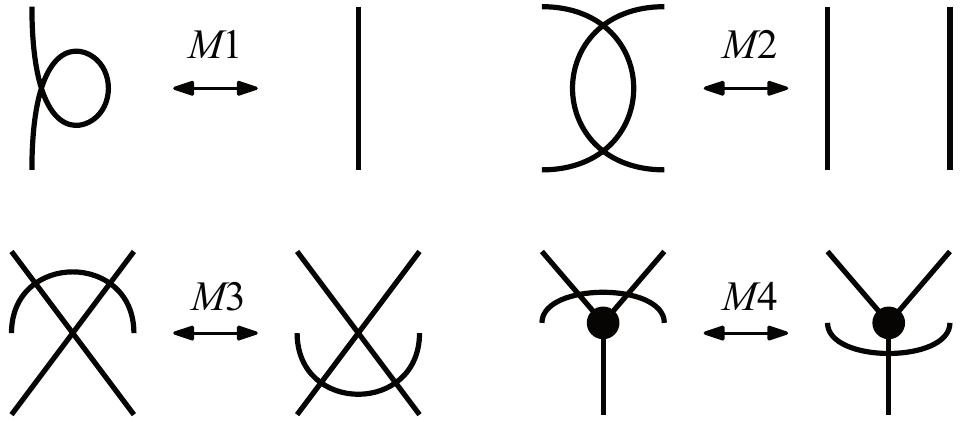}
\end{center}

\end{definition}

\begin{remark} In \cite{BKR}, a fifth move was included to simplify diagrams in order to keep them from becoming too unwieldy. We will not work with overly complicated examples and so it is safe to ignore the fifth move for the purposes of this paper. 
\end{remark}

Introducing the ribbon moves allows us to convert one diagram of a ribbon graph to another.

\begin{theorem}[Baldridge, Kauffman, Rushworth, \cite{BKR}]\label{Thm:ribbons}
Two  ribbon diagrams represent  equivalent oriented ribbon graphs if and only if they are related by a finite sequence of ribbon moves and planar isotopy.
\end{theorem}

As a consequence of \Cref{Thm:ribbons}, we may equivalently define a ribbon graph as an equivalence class of ribbon diagrams, up to the ribbon moves. Therefore, in this paper, we will refer to both ribbon diagrams and ribbon graphs as ``ribbon graphs'' with the small caveat that a given ribbon diagram is assumed to be defined up to equivalent diagrams. This abuse of notation allows us to mimic Grothendieck's ``dessin d'enfants'' of surfaces using diagrams in the much the same way that a ``plane graph'' in practice refers to a specific drawing of a graph whose edges do not intersect each other.

\subsection{Perfect matching graphs}\label{Subsec:pm-graphs}
In \cite{BaldCohomology}, a plane graph with a perfect matching was called a perfect matching graph. This  notion of perfect matching graphs can be generalized to any ribbon graph by decorating the ribbon diagram with a perfect matching (cf. ``matched diagram'' of \cite{BKR} for an example of a perfect matching graph with further decorations). A perfect matching is a set of edges that ``match'' every vertex to exactly one other vertex:

\begin{definition}\label{Def:pm}
A \emph{perfect matching} of an abstract graph $G(V,E)$ is a subset of the edges of the graph, $M\subset E$, such that each vertex is incident to exactly one edge in the subset. 
\end{definition}

The term {\em matching} is used in graph theory for any subset of the edge set of graph; the term \emph{perfect} here refers to the fact that every vertex is incident to exactly one perfect matching edge.

The objects of the main theorems of this paper are equivalence classes of ribbon graphs together with perfect matchings of their underlying graphs. Henceforth we shall refer to these objects as \emph{perfect matching graphs}, and represent them pictorially using ribbon diagrams.

\begin{definition}\label{Def:pm-graph}
A \emph{perfect matching graph}, denoted $\Gamma_M$, is a ribbon graph, $i:G\ra \Gamma$, together with a perfect matching $M$ of the graph $G$. We represent the perfect matching in a ribbon diagram of $\Gamma$ using thickened edges. 
\end{definition}

\begin{figure}
\includegraphics[scale=0.65]{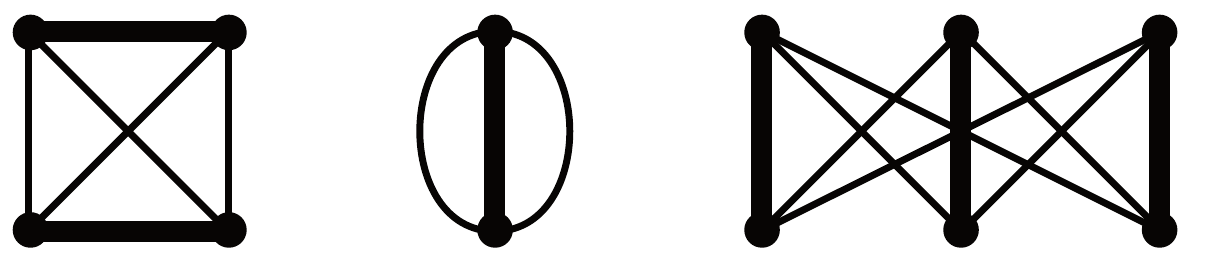}
\caption{Perfect matching graphs.}
\label{Fig:matcheddiagrams}
\end{figure}
Examples of perfect matching graphs are given in \Cref{Fig:matcheddiagrams}.  \Cref{Thm:ribbons} continues to hold for ribbon diagrams with thickened perfect matching edges.  Hence, we can think of a perfect matching graph simultaneously as a ribbon graph or a ribbon diagram up to ribbon move equivalence, together with a perfect matching, i.e., as the ordered pair $(\Gamma, M)$. 

\subsection{The blowup of a graph}\label{Subsec:blowup} While many of the main theorems stated in this paper are stated for perfect matching graphs, these theorems can always be used to provide results about the graphs themselves independent of a ``choice'' of perfect matching. This is because one can pass from a ribbon graph to an associated ribbon graph, called the {\em blowup of the graph}, which has a canonically defined perfect matching.  Since the blown-up graph often retains important features of the original ribbon graph, like its genus or nonzero count of $n$-face colorings, it is profitable to state some of the main theorems in their full generality and then use the blowup of a graph to get invariants of the graph itself when needed. 

\begin{definition}
Let $G(V,E)$ be a graph and $\Gamma$ be a ribbon graph of $G$ represented by a ribbon diagram.  Define the {\em blowup of $\Gamma$}, denoted $\Gamma^\flat$, to be the ribbon diagram  given by replacing every vertex of $\Gamma$ with a circle as in
\begin{center}
\includegraphics[scale=0.09]{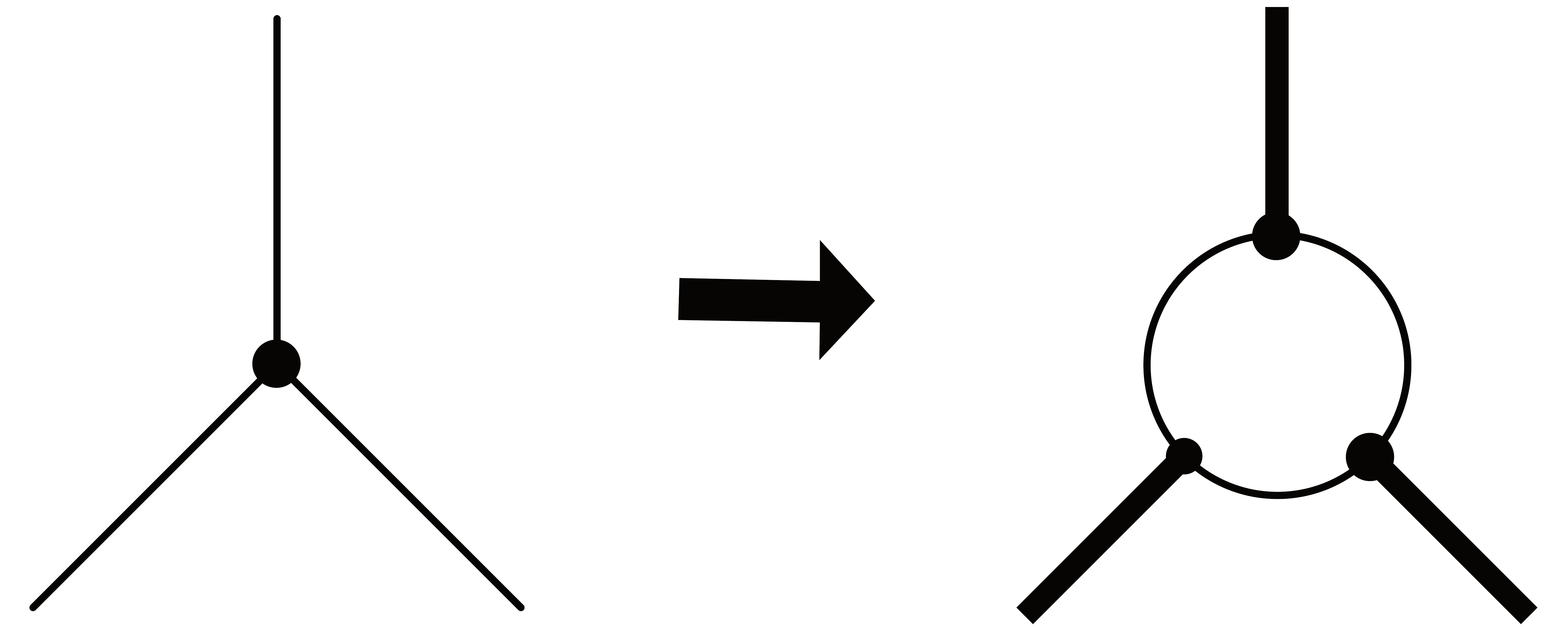}
\end{center}
A perfect matching can be associated to $\Gamma^\flat$  using the original edges $E$ of $\Gamma$ as shown in the picture above.  The  resulting perfect matching graph is $\Gamma^\flat_E$.  \label{def:blow-up-of-a-graph}
\end{definition}

The idea of a blowup originated from early arguments about the four color theorem: Kempe \cite{Kempe} used the related idea of ``patches'' in his attempted-proof of the four color theorem. Later, Tait \cite{Tait} used the blowup of a general plane graph to show that only trivalent plane graphs need be considered in proving the theorem. In fact, one clever observation about Tait's blowup was that the $3$-edge colorings of a trivalent graph are in one-to-one correspondence with the $3$-edge colorings of the blowup of the graph. To see this, label the edges in the left hand picture in the figure above with three different colors and note that there is only one way to properly $3$-color the edges in the blown-up picture on the right with the same colors on the corresponding edges.

\section{\Cref{MainTheorem:n-color-polynomial}: The $n$-color polynomial and its properties}\label{Sec:number-and-polynomial-invariants}

In the introduction  the $3$-color polynomial was briefly  introduced.  This section describes how to generalize this invariant to a family of number and polynomial invariants for every positive integer and briefly discusses some properties of the generalized polynomials. We show how they relate to  colorings of the edges and faces of a graph. At the end of the section we describe how this family of invariants relates to the Penrose polynomial and prove \Cref{MainTheorem:n-color-polynomial}.

\subsection{The definitions.} First, we define a family of polynomials invariants, one for each positive integer $n$. 

\begin{definition}
Let $G$ be a  trivalent graph with a perfect matching $M$, and let $\Gamma_M$ be a perfect matching graph for the pair $(G,M)$. For $n\in \BN$, the {\em $n$-color polynomial of $\Gamma_M$} is the element  $\langle \Gamma_M \rangle_{\! n} \in \BZ[q,q^{-1}]$ characterized by:
\begin{eqnarray}
\langle \PMEdgeDiag \rangle_{\! n} &=&   \langle \IIDiag \rangle_{\! n} \ - \ q^m \langle \XDiag \rangle_{\! n},  \ \ \ \ \ \ \mbox{$m=n/2$ if $n$ even, $m=(n-1)/2$ if odd,} \label{eq:EMformula}\\[.2cm]
\langle \bigcirc  \rangle_{\! n} & = & \left\{\begin{array}{ll} q^m+\cdots+q+1+q^{-1}+\cdots+q^{-m+1} \ \ \ \ & m=\frac{n}{2} \mbox{ if $n$ even} \label{eq:immersed_circle}\\[.3cm]
q^m+\cdots+q+1+q^{-1}+\cdots+q^{-m+1}+q^{-m} & m=\frac{n-1}{2} \mbox{ if $n$ odd},\\
\end{array}\right. \\[.2cm]
\langle \Gamma_1 \sqcup \Gamma_2 \rangle_{\! n} &=& \langle \Gamma_1 \rangle_{\! n} \cdot \langle \Gamma_2 \rangle_{\! n}. \label{eq:disjoint_graphs_identity} 
\end{eqnarray}
\label{defn:n-color-poly}
\end{definition}

The element $\langle \Gamma_M \rangle_{\! n}$ is called the {\em bracket} of $\Gamma_M$. Resolving a perfect matching edge by $\IIDiag$ is called a {\em $0$-smoothing} and by $\XDiag$ a {\em $1$-smoothing}. The immersed lines in the $1$-smoothing will sometimes be thought of as an immersion in the plane and sometimes as a virtual crossing (cf. \Cref{fig:saddle}). The bracket $\langle \Gamma_M \rangle_{\! n}$ depends only on the perfect matching graph and not the perfect matching diagram used to define it. In fact, by modifying the proof of Theorem~$1$ of \cite{BaldCohomology}, it can be shown that the $n$-color polynomial is invariant of the flip moves as well. 

The {\em loop value} of the $n$-polynomial is the evaluation of Equation~\ref{eq:immersed_circle} at one. Therefore, the loop value of the $n$-color polynomial is $n$. Loop values play a prominent role in $(1+1)$-dimensional  TQFTs in knot theory: Khovanov homology is based upon the $q$-variable Kauffman bracket that has a loop value of $2$; it is the dimension of the vector space that must be chosen to build the Khovanov complex (cf. \cite{Kho}). It is this loop value that gives the close relationship between the Jones polynomial in knot theory and the $2$-factor polynomial in graph theory  \cite{BaldCohomology, BKR}.  However, the $n$-color polynomials are new in the literature and quite different in character to polynomials like Jones polynomial in knot theory, {\em even when $n=2$}. We will see this difference emerge later---see \Cref{rem:different-than-khovanov} in \Cref{subsection:differential}.

In \cite{Penrose}, the loop value is the dimension of the abstract tensor system. Penrose presents an interesting example of a  negative dimensional abstract tensor system whose loop value is $-2$. One can ask if there are $n$-color polynomials for negative integers $n$.  We leave these interesting possibilities and potential categorifications into  homology theories for future research.

The $n$-color polynomial, as the name suggests, encodes information about the number of edge or face colorings, especially {\em when the graph is planar}.  To make this relationship explicit later, we make the following definition: 

\begin{definition}
Let $G$ be a trivalent graph with a perfect matching $M$, and let $\Gamma_M$ be a perfect matching graph for the pair $(G,M)$.   The {\em $n$-color number}, $[\Gamma_M]_{n}$, is the number found by evaluating the $n$-color polynomial at one, $[\Gamma_M]_{n} =\langle \Gamma_M \rangle_{\! n}(1)$, i.e., it is the same number calculated by applying the bracket $$\left[ \PMEdgeDiag \right]_{\! n} = \left[ \IIDiag  \right]_{\! n}  - \left[ \XDiag \right]_{\! n}$$ and $\left[ \bigcirc \right]_{\! n} = n$ to $\Gamma_M$.  \label{defn:n-color-number}
\end{definition}

The $n$-color number and $n$-color polynomial exist for any ribbon graph:  Let $G(V,E)$ be any connected graph (with any valence at each vertex), and let $\Gamma$ be a ribbon graph of $G$. Let the perfect matching graph $\Gamma^\flat_E$ be the blowup of $\Gamma$.  The blowup $\Gamma^\flat_E$ is trivalent with the canonical perfect matching given by the edges $E$ of $G$, hence we can use this blowup to define an invariant of the ribbon graph $\Gamma$ itself:

\begin{definition}
Let $G(V,E)$ be any connected graph (with any valence at each vertex), and let $\Gamma$ be a ribbon graph of $G$.  Then the {\em $n$-color polynomial} and {\em $n$-color number of $\Gamma$} are defined to be:
$$\langle \Gamma\rangle_{\! n} :=\langle \Gamma^\flat_E \rangle_{\! n}  \mbox{\ \ \  and  \ \ \ } \left[\Gamma\right]_n :=\left[\Gamma^\flat_E\right]_n.$$
\label{defn:n-color-poly-of-a-graph}
\end{definition}

\begin{remark}
The $n$-color number is a generalization of the {\em Penrose Formula}, $[G]$, which is the $n$-color number when $n=3$ (cf. \cite{Kauffman}). Penrose showed that  it counts the number of $3$-edge colorings (Tait colorings) when $G$ is planar \cite{Penrose}.  Ever since he defined this formula,  mathematicians have looked for ways to turn number invariants like the Penrose Formula into  polynomial invariants.  For example, Kauffman was studying the Penrose Formula when he discovered the Kauffman bracket in knot theory. 
\end{remark}

Before discussing the properties of the $n$-color number and polynomial, we show how to write both of them as a state sum. The state sum makes it clear that the polynomial is well defined independent of how the bracket is applied to the perfect matching graph.
 
 \subsection{The hypercube of states and state sum definition of the $n$-color polynomial}\label{section:smoothing-states-hypercubes} In this subsection, we introduce the hypercube of states to set notation, choices, and definitions that will be used throughout the paper.  The hypercube of states is generated from the smoothings $\IIDiag$ and $\XDiag$  (see \cite{BaldCohomology}).  The $n$-color polynomial can then be written as a sum over these states.  
  
Let $G(V,E)$ be a trivalent graph and $M$ be a perfect matching of $G$. The number of vertices is even and the number of perfect matching edges of $M$ is then $\ell=|V|/2$.  Order and label these edges  by $M=\{e_1, e_2, \dots, e_\ell\}.$ 
Let a perfect matching graph $\Gamma_M$ for $(G,M)$ be represented by a perfect matching diagram.  Resolve each perfect matching edge $e_i \in \Gamma_M$ in one of two possible ways according to the two smoothings, that is, replace a neighborhood of each perfect matching edge $e_i$ in $\Gamma_M$ with $\IIDiag$ or $\XDiag$. The resulting set of immersed circles in the plane is called a {\em state} of $\Gamma_M$.

There are $2^\ell$ states of $\Gamma_M$, each of which can be indexed by an $\ell$-tuple of $0$'s and $1$'s that stand for the type of smoothing.  For  $\alpha= (\alpha_1, \dots,\alpha_\ell)$ in $\{0,1\}^\ell$, let $\Gamma_\alpha$ denote the state where each perfect matching edge $e_i$ has been resolved by an $\alpha_i$-smoothing. We will often refer to the state $\Gamma_\alpha$ by $\alpha$ via this correspondence.

The bracket $\langle \Gamma_M \rangle_{\!n}$ can be expressed as follows:  Define $|\alpha|=\alpha_1+\cdots + \alpha_\ell$ to be the number of $1$'s in $\alpha$. Also, define a function $k$ from the set of states to the nonnegative integers that counts the number of circles in a state.  For notational convenience, let $k_\alpha$ be the number of (immersed) circles in $\Gamma_\alpha$, i.e., $k(\Gamma_\alpha)=k_\alpha$.  Then, for a perfect matching graph $\Gamma_M$ of a nonempty trivalent graph with perfect matching, the $n$-color polynomial is

\begin{equation}
\langle \Gamma_M \rangle_{\! n}(q) = \sum_{\alpha\in\{0,1\}^\ell} (-q^m)^{|\alpha|} (q^m+\cdots+q+1+q^{-1}+\cdots+q^{-m+1})^{k_\alpha}, \label{eqn:general_state-sum}
\end{equation}
where $m=n/2$ when $n$ is even.  There are similar sums for when $n$ is odd and for the $n$-color number.  If the graph is empty, set $\langle \emptyset \rangle_{\! n} = 1$.  If it is a vertex-free graph, i.e., a set of $k$ circles, then the perfect matching is the empty set, and for a perfect matching graph $\Gamma_\emptyset$ for $n$ even, $\langle \Gamma_\emptyset \rangle_{\!n} = \left(q^m+\cdots q^{-m+1}\right)^k$, with similar expression for $n$ odd. The right hand expression in \Cref{eqn:general_state-sum} is called a {\em state sum}.  

The set of states can be conceptualized as a hypercube in which each state is a vertex of the cube.  For example, \Cref{fig:theta_3} shows the cube of resolutions for the $\theta_3$ graph with indicated perfect matching.

\begin{figure}[H]
\includegraphics[scale=.2]{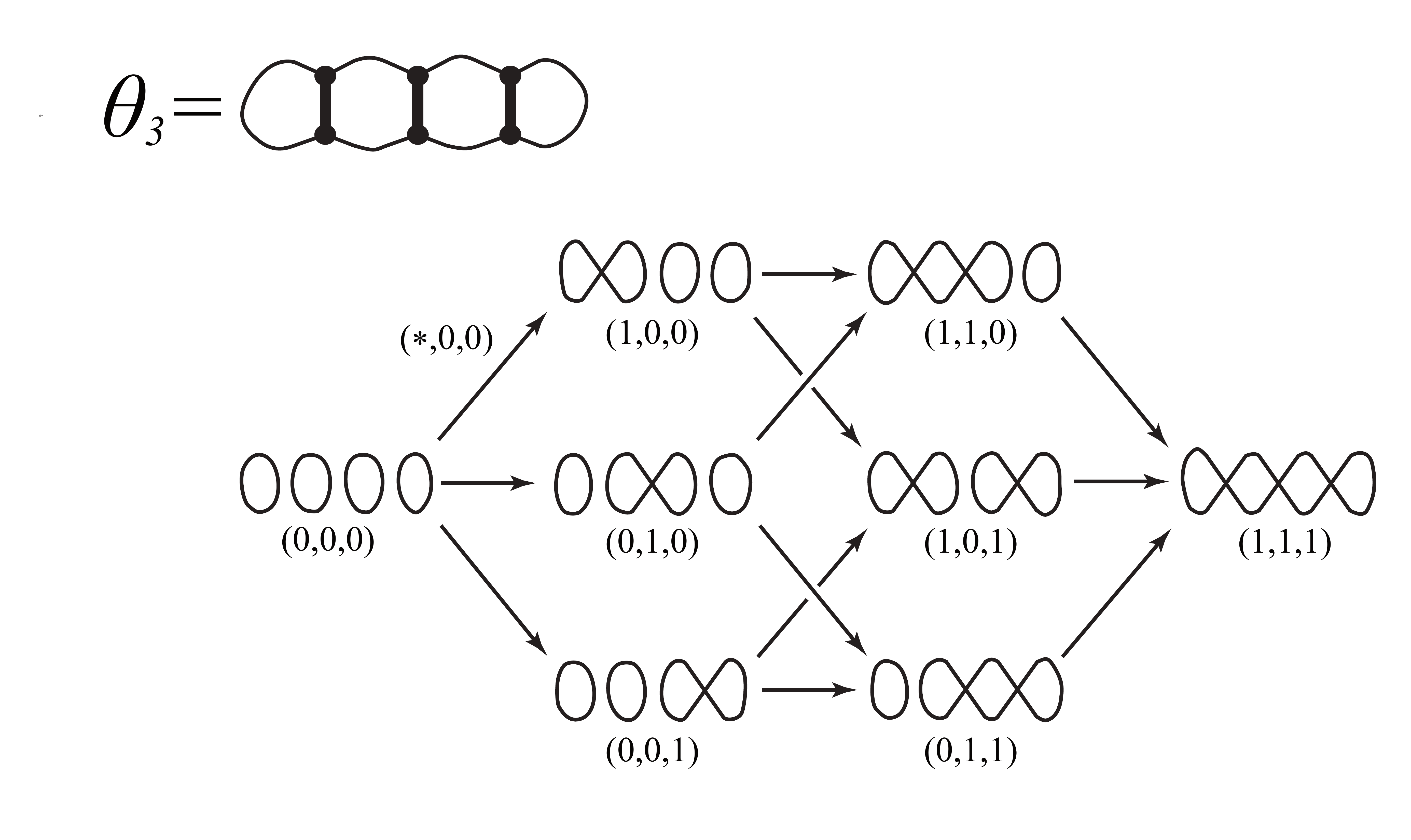}
\caption{Hypercube of resolutions for the graph the $\theta_3$ plane graph with indicated perfect matching.}
\label{fig:theta_3}
\end{figure}

The edges of the hypercube are determined as follows. Consider an edge $\zeta_{\alpha\alpha'}$ in the hypercube between two states $\Gamma_\alpha$ and $\Gamma_{\alpha'}$.  Edges occur when $\alpha_i=\alpha'_i$ for all $i$ except for one edge $e_k\in M$ where $\alpha_k=|\alpha'_k-1|$.  Label this edge by a tuple of $0$'s and $1$'s for the $\alpha_i$'s and $\alpha'_i$'s that are the same and  a ``$\ast$'' for for the $k$th position where $\alpha_k=|\alpha'_k-1|$.  For example, the edge in Figure~\ref{fig:theta_3} between $(0,0,0)$ and $(1,0,0)$ would be labeled $(\ast,0,0)$.  Turn each edge into a directed segment $\zeta_{\alpha\alpha'}:\Gamma_\alpha \rightarrow \Gamma_{\alpha'}$ by requiring the tail to be where $\ast=0$ and the head is where $\ast=1$, that is, the $0$-smoothing in $\Gamma_\alpha$ is changed into a $1$-smoothing in $\Gamma_{\alpha'}$.\\

\subsection{Examples} We provide examples of the $n$-color number and $n$-color polynomial for $n=2,3,4$ and describe some important features of each.

\begin{example}[The $2$-color number and polynomial] The $2$-color polynomial is found by applying the bracket $\langle \PMEdgeDiag \rangle_2 = \langle \IIDiag  \rangle_2  - q\langle \XDiag \rangle_2$ and $\langle \bigcirc \rangle_2 = q+1$. In general, the $n$-color polynomial and number depend upon the perfect matching and ribbon graph even when the graph is planar.    For example, computing the $2$-color number for the $\theta_2$ graph using two different perfect matchings 
\begin{figure}[H]
\includegraphics[scale = .6]{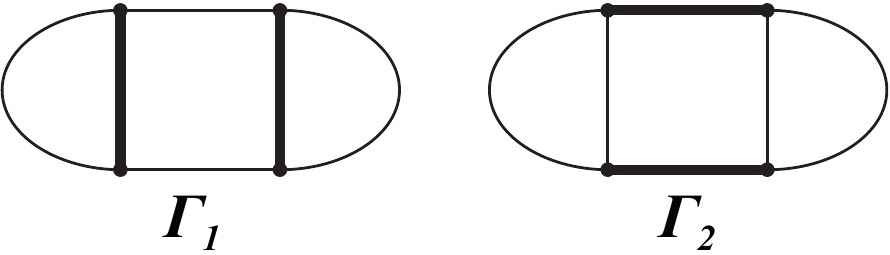} 
\caption{The $\theta_2$ graph with its two different perfect matchings.}
\label{fig:theta_2_with_different_pms}
\end{figure}
 
 \noindent shows that $[\Gamma_1]_2 = 2$ and $[\Gamma_2]_2 = 4$. In fact, given a perfect matching $M$ of a planar trivalent graph $G$ and a plane perfect matching graph $\Gamma_M$ of the pair $(G,M)$, if all of the cycles in $G\setminus M$ have an even number of edges, then it was shown in \cite[Proposition 3]{BLM} that  
 \begin{equation}
 [\Gamma_M]_2 = 2^{\mbox{\# of cycles of $G\setminus M$}}\label{eqn:2-factor-counts-even-cycles}
 \end{equation}
and is zero if any of the cycles in $G\setminus M$ have an odd number of edges. Therefore, if the $2$-color number is non-zero, $G$ is 3-edge colorable by coloring the perfect matching edges by one color, say purple, and coloring the cycle edges in alternating fashion: red, blue, red, etc. Hence, the $2$-color number provides information about specific types of edge colorings.
\end{example}

\begin{example}[The $3$-color number and polynomial] The $3$-color number of a planar perfect matching graph is one of the most important numbers in graph theory: It is equal to the Penrose Formula, which counts the number of $3$-edge colorings of a planar graph.  Unlike the other $n$-color numbers, Penrose showed that the $3$-color number is independent of the perfect matching chosen (cf. \Cref{fig:theta_2_with_different_pms} above).  This independence is a bit surprising and not obvious from the definitions. However, an  argument can be made for why $n=3$ is special using the proof of Statement (2) of \Cref{MainTheorem:n-color-polynomial} below.

One of the key strengths we will see of the bigraded and filtered $3$-color homologies is their ability to detect edge colorings on non-planar graphs when the $3$-color number does not.  As an example of this phenomenon  that will be used throughout the paper, a calculation of the $3$-color number of a perfect matching graph of  $K_{3,3}$ is 0 as shown in \Cref{fig:K33-is-zero}.

\begin{figure}[H]
\begin{eqnarray*}
\left[ \Kthreethree \right]_{\!3} &=& \left[ \Kthreethreezero \right]_{\!3} - \left[ \Kthreethreeone \right]_{\!3}\\
&=& \left[ \thetatwohor \right]_{\!3} - \left[ \thetatwovert \right]_{\!3}\\
&=& 12 -12\\
&=& 0.
\end{eqnarray*}
\caption{The $3$-color number of the $K_{3,3}$ perfect matching graph $\Gamma_M$ is zero.} \label{fig:K33-is-zero}
\end{figure}

While the $3$-color number is zero for the $K_{3,3}$ perfect matching graph $\Gamma_M$ in \Cref{fig:K33-is-zero}, the $3$-color polynomial is nonzero.  A similar bracket calculation to the one above shows that, for the perfect matching graph $\Gamma_M$ of $K_{3,3}$, the $3$-color polynomial is: 
\begin{equation} \langle \Gamma_M \rangle_{\!3}(q) =  q^{-3} + 3q^{-2} + 3q^{-1} +1 - q - 2 q^2 - 2 q^3 - 2 q^4 - q^5.\label{eqn:3-color-poly-of-K33}
\end{equation} 
This polynomial will become the graded Euler characteristic of the bigraded $3$-color homology theory discussed later. In general, the bigraded $n$-color homology is nontrivial for a plane perfect matching graph even when its $n$-color number is zero (cf. \Cref{theorem:CH-not-zero}).
\end{example}

\begin{example}[The $4$-color number and polynomial] Next we show how the $4$-color number (and hence $4$-color polynomial) relates to colorings on the graph.  Recall Tait's well-known relationship between edge colors and face colors of planar graphs using the Klein four-group $K_4=\BZ_2 \times \BZ_2$:  Suppose there is $3$-edge coloring of a bridgeless plane graph with the nonzero elements $\{ (1,0), (0,1), (1,1)\}$  of the Klein four-group (the edge ``colors''). Furthermore, choose one of the faces of the plane graph (the face ``at infinity'') to be labeled by the ``translucent'' color $(0,0)$. Then every other face can be colored using an adjacent edge color and an already-colored  face adjacent to that edge  by adding their elements from the Klein four-group together. Starting with a different color for the face at infinity results in a different $4$-face coloring of the plane graph. Hence,
\begin{equation}
\#\{\mbox{$4$-face colorings of $\Gamma_M$}\} = 4\cdot [\Gamma_M]_{3}\label{eqn:4-face-to-3-edge}
\end{equation}
when $\Gamma_M$ is any plane perfect matching graph of the plane graph $\Gamma$. Thus, the number of $4$-face colorings of a plane graph is related to the number of $3$-edge colorings of the graph.
\end{example}

The $4$-color number is linked to edge colorings through the blowup of a plane graph. Let $G(V,E)$ be a planar trivalent graph and $\Gamma$ be a plane ribbon graph, i.e., a plane graph for $G$. Then $\Gamma^\flat_E$ is the blowup of $\Gamma$ with perfect matching corresponding to the set of edges of $G$. Next, resolve $\Gamma^\flat_E$ by performing $0$-smoothings on all perfect matching edges to get a set of embedded circles in the plane. This set of circles is called the {\em all-zero state} of $\Gamma$ and is denoted $\Gamma_{\vec{0}}$.  (The same can be done for any perfect matching graph, but then the circles may possibly be immersed.) This set of circles corresponds to the disks that are glued into the ribbon graph to get the associated closed surface $\overline{\Gamma}$ discussed after \Cref{Def:ribbongraph}.

\begin{center}
\includegraphics[scale=.8]{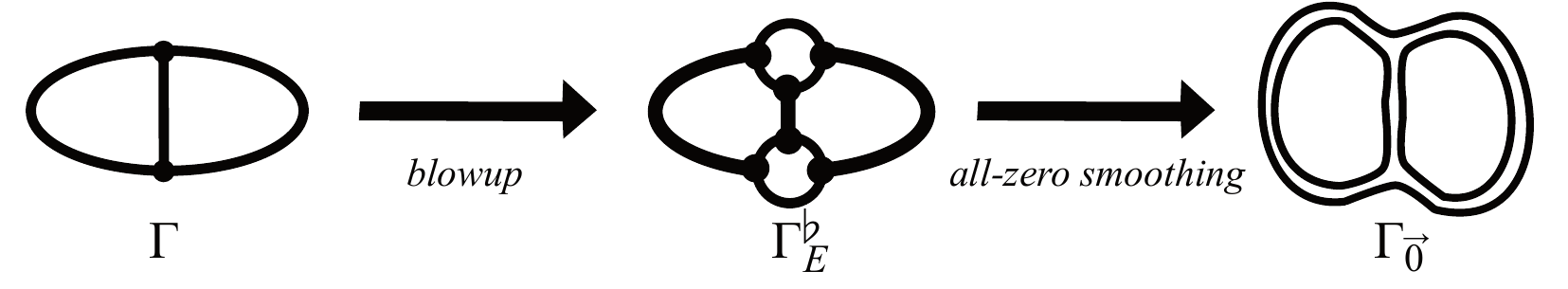}
\end{center}

Hence, an $n$-face coloring of $\Gamma$, or more generally, a choice of labelings of the faces of $\Gamma$, corresponds to a choice of colorings (choice of labelings) of the circles in $\Gamma_{\vec{0}}$. In the computation of the $4$-color number of $\Gamma^\flat_E$, the first term in the state sum in \Cref{eqn:general_state-sum} of $[\Gamma]_4$ ($= \langle\Gamma^\flat_E\rangle_{\! 4}(1)$)  is $4^k$, where $k$ is the number of circles in $\Gamma_{\vec{0}}$.  This suggests that $[\Gamma]_4$ potentially counts the number of $4$-face colorings of $\Gamma$ in some way.  This is indeed true:

\begin{theorem}[Aigner \cite{Aigner}, Proposition 7] Let $\Gamma$ be a connected plane graph. Then $[\Gamma]_4 >0$ if and only if $\Gamma$ is $4$-face colorable. \label{thm:4-color-number-positive}
\end{theorem} 

This theorem and proof is an existence result. Aigner's proof is nontrivial, in fact, we prove a generalization of it in Statement (2) of \Cref{Theorem:mainthm-four-color}. For now, we note that Statement (3) of \Cref{MainTheorem:Colorings-of-State-Graphs}, $$\dim \widehat{CH}^0_4(\Gamma;\BC) = \#\{\mbox{4-face colorings of $\Gamma$}\},$$  together with the fact that $\widehat{CH}_4^*(\Gamma;\BC)$ is nonzero only in even degrees for planar ribbon graphs (see \Cref{proposition:even-degree-non-zero-n-face-colorings}), is a partial answer to this question. Since $[\Gamma]_4$ is the Euler characteristic of $\widehat{CH}^*_4(\Gamma;\BC)$, it counts at least the number of $4$-face colorings and possibly more. However, \Cref{MainTheorem:Colorings-of-State-Graphs} does not imply that $\widehat{CH}^0_4(\Gamma;\BC)$ is nontrivial for planar bridgeless graphs, i.e., the four color theorem.

\subsection{The general case: the Penrose polynomial} The $n$-color numbers for the blowup of a ribbon graph are the values of the Penrose polynomial evaluated at $n$.  There is quite a bit of research on this polynomial, notably \cite{Aigner} or \cite{Moffat2013}, but see also \cite{Martin, Jaeger} for examples of what is known. Recent work on the Penrose polynomial has been on extending it to ribbon graphs and delta-matroids (cf. \cite{EMM, MM, EMMKM}). 

The Penrose polynomial as it is defined in the literature is generally difficult and nonintuitive to describe.  We give an intuitive  definition in this paper using brackets:

\begin{definition}\label{defn:penrose_poly}
Let $G(V,E)$ be any connected graph (with any valence at each vertex), and let $\Gamma$ be a ribbon graph of $G$. Then the {\em Penrose polynomial}, denote $P(\Gamma,n)$,  is found by applying the bracket $$\left[ \PMEdgeDiag \right]_{\! n} = \left[ \IIDiag  \right]_{\! n}  - \left[ \XDiag \right]_{\! n}$$ and $\left[ \bigcirc \right]_{\! n} = n$ to the blowup $\Gamma^\flat_E$ of the ribbon graph $\Gamma$.  
\end{definition}

In this definition, unlike the $n$-color polynomial,  $P(\Gamma,n)$ can be evaluated at negative values of $n$. For example, recall that Penrose showed that $P(\Gamma,-2)$ is a multiple of the total number of $3$-edge colorings \cite{Penrose} when the ribbon graph is planar.  

When $n$ is a positive integer, then $P(\Gamma,n)$ and the $n$-color number of the blowup are the same, i.e., 
\begin{equation}
P(\Gamma,n) = \left[\Gamma^\flat_{E}\right]_{\! n}.\label{eqn:Penrose-equals-n-color-number}
\end{equation}

Much is already known about the evaluation of $P(\Gamma,n)$ at different integral $n$, for example:

\begin{theorem}[cf. Aigner \cite{Aigner}, Penrose \cite{Penrose}, Jaeger \cite{Jaeger}] Let $\Gamma$ be a connected plane graph of a (possibly non-regular) graph $G(V,E)$. Then
\begin{enumerate}
\item $P(\Gamma, -2) = (-4)^{\frac12 |V|}[\Gamma]_3$ when $\Gamma$ is trivalent,
\item $P(\Gamma,-1)= \pm 2^{|E|}$ if $\Gamma$ is Eulerian and $|P(\Gamma,-1)| <2^{|E|}$ if not.
\item $P(\Gamma,0)=0$,
\item $P(\Gamma,1)=0$,
\item $P(\Gamma, 2)= 2^{|V|}$ if $\Gamma$ is Eulerian and zero otherwise,
\item $P(\Gamma,3) = \#\{\mbox{$3$-edge colorings of $\Gamma$}\}$ when $\Gamma$ is trivalent.
\item $P(\Gamma,4)>0$ if and only if $\Gamma$ is $4$-face colorable (see Theorem~\ref{thm:4-color-number-positive} above).
\item $P(\Gamma, n)\geq \chi(\Gamma^*,n)$ where $n\in\BN$  and $\chi(\Gamma^*, n)$ is chromatic polynomial of the geometric dual of the plane graph $\Gamma$.
\end{enumerate}\label{thm:penrose-polynomial-values}
\end{theorem}

The Penrose polynomial of a ribbon graph has a major drawback that prevents its use in proving important theorems like the four color theorem (a drawback that we feel is rectified with the homology theories defined in this paper). Due to the symmetry in the difference of the terms of the formula in the definition of the Penrose polynomial, it possible that it can be zero when evaluated at $n$ even if it the graph is planar. Of course, the four color theorem rules this possibility out for $n=3$ and $n=4$ for planar graphs. However, when the graph is non planar, the polynomial can be identically the  zero polynomial. There are well known examples of this when the ribbon graph is nonplanar:

\begin{example}\label{example:K33-computation-of-Penrose-poly} Let $\Gamma$ be a ribbon graph of the $K_{3,3}$ in \Cref{Fig:ribbondiagram}. Then $P(\Gamma,n) = 0$ for all $n\in \BN$.
\end{example}

This is a calculation found by applying the bracket of \Cref{defn:penrose_poly} using the blowup of the ribbon graph $\Gamma$ of $K_{3,3}$ in \Cref{Fig:ribbondiagram}.  Note that the $3$-color polynomial of $\Gamma$ is nonzero:

\begin{equation}
\langle \Gamma \rangle_{\! 3}(q) =\frac{1}{q^3} + \frac{2}{q^2} + \frac{1}{q}+4 - 2 q - 15 q^2 + 11 q^3 - 14 q^4 + 14 q^5 - 
 11 q^6 + 15 q^7 + 2 q^8 - 4 q^9 - q^{10} - 2 q^{11} - q^{12}\label{eqn:3-color-poly-of-blowup-of-K33}
\end{equation}

\begin{remark} This polynomial is different from the polynomial in \Cref{eqn:3-color-poly-of-K33} because it is the $3$-color polynomial of the ribbon graph $\Gamma$ shown in \Cref{Fig:ribbondiagram}, i.e., the $3$-color polynomial of the blowup of $K_{3,3}$, and not the perfect matching graph $\Gamma_M$ shown in \Cref{fig:K33-is-zero}, which is for that specific perfect matching $M$. 
\end{remark}

The reason the $3$-color polynomial is nontrivial even when the Penrose polynomial is identically zero is because  the $-q$ coefficient in front of the $1$-smoothing in Equation~\ref{eq:EMformula} breaks the symmetry generated by the Penrose bracket (where the coefficient is $-1$). In general, the $n$-color polynomials contain more information about the graph than the Penrose polynomial.  However,  one can combine the two notions to get a two-variable polynomial:

\begin{definition}
\label{defn:two_var_penrose_poly}
Let $G(V,E)$ be any connected graph (with any valence at each vertex), and let $\Gamma_M$ be a perfect matching graph of $G$. Then the {\em two-variable Penrose polynomial}, denote $P(\Gamma_M,q,n)$,  is found by evaluating the bracket $$\left[ \PMEdgeDiag \right]_{\! q,n} = \left[ \IIDiag  \right]_{\! q,n}  - q\left[ \XDiag \right]_{\! q, n}$$ and $\left[ \bigcirc \right]_{\! q, n} = n$. The $2$-variable Penrose polynomial of a ribbon graph $\Gamma$ is defined to be polynomial defined using the blowup $\Gamma^\flat_E$ of the ribbon graph.  
\end{definition}

We will not work with the two-variable Penrose polynomial in this paper, but note in passing that it too is a nonzero polynomial for $K_{3,3}$.

In the next section, after proving \Cref{MainTheorem:n-color-polynomial}, we explain how the main theorem can be used to prove most of the statements of \Cref{thm:penrose-polynomial-values}. As a warmup example, we use the $2$-color polynomial and \Cref{defn:penrose_poly} to prove Statement (5) above, i.e., $P(\Gamma,2)=2^{|V|}$ when $\Gamma$ is Eulerian.  First, Statement (1) of \Cref{MainTheorem:n-color-polynomial} follows from \Cref{eqn:2-factor-counts-even-cycles} and the definition of the $2$-color number when the perfect matching graph is planar:

\begin{proposition} Let $\Gamma_M$ be a perfect matching graph of a connected trivalent graph $G(V,E)$. If $M$ is an even perfect matching, i.e., all cycles in $G\setminus M$ have even length, then 
 $$ \langle \Gamma_M\rangle_{\! 2}(1)= 2^k,$$
\noindent where $k$ is the number of cycles of $G\setminus M$. If $M$ is odd, then the $2$-factor polynomial evaluates to $0$.\label{proposition:even-perfect-matching-is-a-power-of-2}
\end{proposition}

Note that the statement of this theorem does not require the graph to be planar.  The proof in this case is not trivial either.  It is proved using Theorem 5 of \cite{BaldKauffMc} by applying the functor described in \cite{BKR} and noting that all polynomials used in the proofs, when evaluated at one, are all equal to the $2$-color number of $\Gamma_M$.

This proposition can be used to prove Statement (5) of \Cref{thm:penrose-polynomial-values} for planar or nonplanar graphs. If $\Gamma$ is Eulerian, then the valence of each vertex is even. The blowup of $\Gamma$ therefore is a trivalent graph $\Gamma^\flat_E$ such that $\Gamma^\flat_E\setminus E$ is a set of even cycles.  By the proposition and the fact that $\langle \Gamma_E^\flat\rangle_{\! 2}(1) = P(\Gamma,2)$, we have $P(\Gamma,2) = 2^{|V|}$ since the cycles are all even and in one-to-one correspondence with the set of vertices of $\Gamma$.  If the graph was not Eulerian, then at least one vertex has an odd valence. This leads to an odd cycle and $P(\Gamma,2)=0$.

\subsection{Proof of \Cref{MainTheorem:n-color-polynomial}} \label{subsection:proof-of-theorem-A}

In this section we prove the first of the main theorems. 

To prove the first sentence of \Cref{MainTheorem:n-color-polynomial}, one must prove that the polynomial is independent of the choice of ordering of the perfect matching $M$ and that it is independent of the perfect matching diagram chosen.  The first follows from the state sum formula and the second follows from checking the moves of  \Cref{Thm:ribbons}. This is left as an exercise for the reader.

Statement (1) of \Cref{MainTheorem:n-color-polynomial} was  proven above in \Cref{proposition:even-perfect-matching-is-a-power-of-2}.  We leave the proof of Statement (2) until after the proof of Statement (5). Next, we prove Statement (3).

\begin{proof}[Proof of Statement (3)] Let $\Gamma_M$ be a not-necessarily-planar connected trivalent perfect matching graph of $G$. Suppose $G$ contains a bridge edge $e$. The graph $G' =(V,E\setminus M)$ is a collection of cycles. If a bridge edge $e$ was not contained in $M$, then $e$ is part of a cycle in $G'$. Thus, $e$ was part of a cycle in the original graph $G$, which contradicts the fact that bridges are not contained within cycles. Hence, $e\in M$.

Suppose $|M|=\ell$.  Resolving the edge $e\in M$ bifurcates the hypercube of states into two sub-hypercubes, each with $2^{\ell-1}$ states, together with an edge from each state with a $0$-smoothing of $e$, $\Gamma_\alpha$, in the first sub-hypercube to a corresponding state with a $1$-smoothing of $e$, $\Gamma_{\alpha'}$, in the second sub-hypercube.  Since $e$ is not in a cycle, each of these edges in the hypercube only represent introducing a self-intersection in the circle associated with the $0$-smoothing circle of edge $e$ to get the $1$-smoothing circle.  In particular, the number of circles in each of the corresponding states are the same.  Since one state comes with a positive sign and the other comes with a negative sign in the state sum (cf.  \Cref{eqn:general_state-sum}) when evaluated at $q=1$, i.e., $|\alpha'| = |\alpha|+1$, the theorem follows. 
\end{proof}

Statements (4) and (5) require the full power of the perfect matching version of \Cref{MainTheorem:Colorings-of-State-Graphs} to prove.  For a perfect matching graph (and not just the blowup of a graph), \Cref{MainTheorem:Colorings-of-State-Graphs} counts the number of ways to color the circles in a state with $n$ colors so that adjacent circles (circles that share a $0$- or $1$-smoothing edge) are labeled with different colors. Because they require the full machinery of the homology theories, readers may find it beneficial to wait until the end of \Cref{section:A-TQFT-approach-to-the-4CT} before reading these proofs.  

\begin{proof}[Proof of Statement (4)] When $\Gamma_M$ is a plane graph $\Gamma$ with perfect matching $M$, $\langle\Gamma_M\rangle_{n-1}(1)$ is equal to the nonnegative sum of dimensions of harmonic colorings of all states, $\widehat{\mathcal{CH}}_{n-1}(\Gamma_\alpha)$, of the hypercube of states generated by $\Gamma_M$ (not $\Gamma^\flat_E$).  This follows from \Cref{MainTheorem:Bigraded-n-color-homology}, \Cref{MainTheorem:spectralsequence}, and the perfect matching graph version of \Cref{MainTheorem:Colorings-of-State-Graphs} and  \Cref{proposition:even-degree-non-zero-n-face-colorings}.  Therefore, $\langle\Gamma_M\rangle_{n-1}(1)>0$ if and only if there exists an $\alpha\in\{0,1\}^{|M|}$ such that $\dim \widehat{\mathcal{CH}}_{n-1}(\Gamma_\alpha)>0$.  Thus, if the evaluation is nonzero, then there exists a coloring of the state $\Gamma_\alpha$ (see \Cref{subsection:colorbasis}) that represents an $(n-1)$-coloring of the circles in the state $\Gamma_\alpha$ so that no two adjacent circles have the same color.

The colors of the $(n-1)$-coloring of the state can be identified with nonzero elements of the generalized Klein group $K_n:=\BZ_2\times\cdots\times \BZ_2$. Note that the adjacent circles at each perfect matching edge are labeled with different nonzero elements of $K_n$.  These elements then determine a nowhere zero $K_n$-flow on $G$ as follows: (a) for each edge in $G\setminus M$, label the edge with the element of $K_n$ associated to the circle that contains (part of) that edge, and (b) for  each edge in $M$, add the two elements together associated to the two adjacent circles of that perfect matching edge. Since both elements are nonzero and different, the sum will also be nonzero and different from both.  By the construction and since $G$ is trivalent, the sum of all three elements of $K_n$ adjacent to each vertex must then be $0$ (mod $2$). Hence, this defines a nowhere zero $K_n$-flow.

A nowhere zero $K_n$-flow on $G$ can then be used to $n$-face color the plane ribbon graph $\Gamma$ using a generalization of Tait's algorithm (choose one face to be $0\in K_n$ and use the additive structure of $K_n$ and the flow to label each remaining face).  See the proof of Statement (2) of \Cref{Theorem:mainthm-four-color} (or \Cref{prop:state-graph-colorable-implies-plane-graph-colorable}) or the proof of Statement (5) below for the basic idea behind the algorithm.
\end{proof}

The proof of Statement (5) is almost a corollary of the proof of Statement (4). 

\begin{proof}[Proof of Statement (5)] Choose a one-to-one correspondence between the $n$ colors $\{c_0,c_1,\ldots,c_{n-1}\}$ and elements of $K_n$ such that $c_0$ corresponds to $0\in K_n$.  Using colors $\{c_1,\ldots,c_{n-1}\}$ and the correspondence, the proof of Statement (4) above shows that each $(n-1)$-coloring  of a state corresponds to a nowhere zero $K_n$-flow.  It is clear by the construction that each of the colorings on a given state corresponds to a unique nowhere zero $K_n$-flow.

To prove the lower bound, we need to show that nowhere zero $K_n$-flows generated on different states are unique with no double counting of flows. Suppose that there are two states $\Gamma_\alpha$ and $\Gamma_{\alpha'}$ and colorings on each that induce the same nowhere zero $K_n$-flow. Since $\alpha\not=\alpha'$ in $\{0,1\}^{|M|}$, there is, say, an edge $e_i\in M$ such that $\alpha_i=0$ and $\alpha'_i=1$, i.e., the first state has a $0$-smoothing at edge $e_i$ and the second has a $1$-smoothing. If the two colorings  induce the same nowhere zero $K_n$-flow, then the colors of the two circles adjacent to the edge $e_i$ in both states must be the same color. This contradicts the fact that they were a valid coloring of each state to begin with. Thus, all nowhere zero $K_n$-flows derived from valid colorings on and across each of the states must be different, proving the the lower bound.

For the upper bound, begin with a nowhere zero $K_n$-flow.   Choose any element $a \in K_n$ (possibly zero) for the outer face of the plane graph $\Gamma$.  Starting with the outer face, each adjacent face can be inductively colored using Tait's algorithm.  For example, at each vertex the edges are labeled by distinct nonzero elements of $K_n$, call them $r, s, t\in K_n$, such that $r+s+t = 0 \in K_n$.  Suppose that the outer face is adjacent to edges labeled $r$ and $s$. Then the other face adjacent to edge $r$ is colored $a+r$ and the other face adjacent to edge $s$ is colored $a+s$. The third edge, edge labeled $t$, transforms the colored face $a+r$ to $a+r+t$, which is equal to $a+s$ in $K_n$. Continuing this process at each new vertex provides a labeling of the faces of $\Gamma$ by elements of $K_n$ such that no two adjacent faces are labeled with the same element, i.e., the nowhere zero $K_n$-flow and the initial choice $a$ gives rise to an $n$-face coloring of $\Gamma$.  Since each nowhere zero $K_n$-flow generates $n$ different $n$-face colorings of $\Gamma$, and $\langle\Gamma\rangle_{\! n}(1)$ counts all $n$-face colorings of $\Gamma$ (and more) by \Cref{MainTheorem:Bigraded-n-color-homology}, 
\Cref{MainTheorem:spectralsequence}, and \Cref{MainTheorem:Colorings-of-State-Graphs}, this proves the upper bound.
\end{proof}

\begin{remark} The lower bound holds even when $\Gamma$ is nonplanar. The upper bound does not. See \Cref{eqn:3-color-poly-of-K33} for example.
\end{remark}

We can now  prove Statement (2). The proof below relies only on the $3$-color polynomial and its properties, not on facts already known in the literature. 

\begin{proof}[Proof of Statement (2)] Since the number of nowhere zero $K_4$-flows is equal to the number of $3$-edge colorings of $G$, Statement (5) implies that $$\langle\Gamma_M\rangle_{\! 3}(1) \leq \#\{\mbox{$3$-edge colorings of $G$}\}.$$

We prove that $\langle\Gamma_M\rangle_{\! 3}(1)$ is also an upper bound. Let $K_4=\BZ_2\times\BZ_2$ be presented as the group of elements $\{0=(0,0),a=(1,0),b=(0,1),c=(1,1)\}$.  Given a $3$-edge coloring of $G$, i.e., a ``coloring'' by $\{a,b,c\}$ of the edges of $\Gamma_M$, inspect a perfect matching edge of $e_i\in M$.  If the edge is labeled by, say, $c$, then the edges incident to each vertex of $e_i$ must be labeled $a$ and $b$. This must be true for both vertices since there are exactly two nonzero elements of $K_4$, $a$ and $b$, that add up to $c$.  Hence, the picture of $\Gamma_M$ at $e_i$ must look like (up to a permutation of the labels $a$, $b$, $c$) one of the two pictures in the left column of \Cref{fig:k4-flow-to-coloring-of-state}.

\begin{figure}[H]
\includegraphics[scale=1]{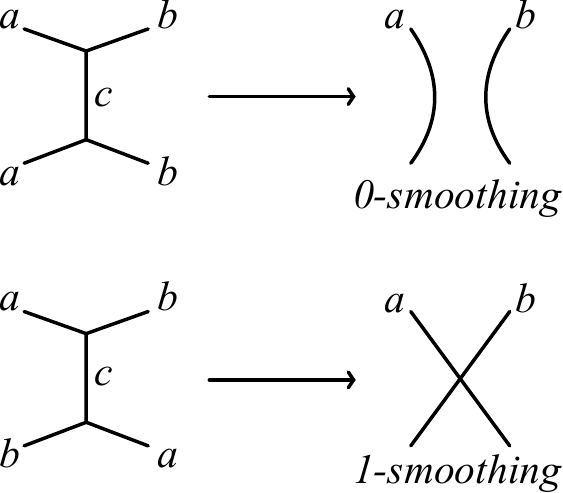}
\caption{A nowhere zero $K_4$-flow induces a coloring on some state in the hypercube of states of $\Gamma_M$.}
\label{fig:k4-flow-to-coloring-of-state}
\end{figure}

The right column of the figure shows how each local picture of the $K_4$-flow corresponds to a coloring of the two circles associated with $e_i$ in a state. Resolving each perfect matching edge of $\Gamma_M$ using the right column of \Cref{fig:k4-flow-to-coloring-of-state} produces a valid coloring on a specific state $\Gamma_\alpha$  in the hypercube of states of $\Gamma_M$.  This coloring is a basis element of $\widehat{\mathcal{CH}}_3(\Gamma_\alpha)$. Since this basis element is uniquely determined from the nowhere zero $K_n$-flow, and $\langle\Gamma_M\rangle_{\! 3}(1)$ is the sum of the dimensions of valid colorings on states $\Gamma_\alpha$, it is also an upper bound. 
\end{proof}

\begin{remark}
The proofs of Statement (5)  and Statement (2) hinge on the fact that there are exactly two nonzero elements  of $K_4$ that add up to $c$. This is not the case for $K_n$ when $n>4$; the proof cannot be easily generalized. However, using the blowup of $\Gamma$, one can make a similar statement and conjecture that the total face color polynomial evaluated at one is bounded from below by the number of nowhere zero $K_n$-flows (see \Cref{conjecture:total-color-poly-is-an-upper-bound} and \Cref{Theorem:mainthm-four-color}).
\end{remark}

We end this section with a few comments about \Cref{MainTheorem:n-color-polynomial} and the $n$-color polynomial:
\begin{itemize}
\item The proofs of Statements (3)--(8) of \Cref{thm:penrose-polynomial-values}  follow  from \Cref{MainTheorem:n-color-polynomial} using the facts: (a) $P(\Gamma,n)=\langle\Gamma\rangle_{\! n}(1)$,  (b)  $3$-edge colorings of the blowup are in one-to-one correspondence with $3$-edge colorings of the original graph when the graph is trivalent, Statement (2) of \Cref{Theorem:mainthm-four-color}, and (d)  $\chi(\Gamma^*,n)=\dim \widehat{CH}_n^0(\Gamma;\BC)$ from \Cref{MainTheorem:Colorings-of-State-Graphs}. Hence, one recovers many of the  known results about the Penrose polynomial when $n>1$ from the $n$-color polynomial.

\item The $n$-color polynomial/number for a perfect matching graph $\Gamma_M$ requires far fewer computations to calculate than the Penrose polynomial (which starts with the blowup $\Gamma_E^\flat$), in fact, the number of computations are on the order of a factor of $2^{2|E|/3}$. These invariants are also easier-to-calculate and less computation-heavy than the homology theories and other polynomials defined in this paper. Thus, the polynomial and number can be thought of as computationally-assessable, excellent first approximations to the Penrose polynomial and total face color polynomial. 

\item In this paper we do not address the research value of the $n$-color polynomial invariant beyond \Cref{MainTheorem:n-color-polynomial} and its relationship to the Penrose polynomial in \Cref{thm:penrose-polynomial-values}.  Like the Jones polynomial in knot theory, the polynomial itself can and should be studied for secrets it may unlock about perfect matchings of a graph and of the graph.   Hence we see  \Cref{MainTheorem:n-color-polynomial} as a first step in this line of research.


\end{itemize}

The $n$-color polynomial and $n$-color number are the easiest and simplest of the new invariants of this paper to calculate, but  not the strongest. That invariant, the bigraded $n$-color homology, is addressed next.

\section{\Cref{MainTheorem:Bigraded-n-color-homology}: Bigraded $n$-color homology}
\label{section:bigraded-n-color-homology}

We are now ready to define the bigraded $n$-color homology.  In short, the bigraded $n$-color homology replaces the hypercube generated for the $n$-color polynomial with a chain complex of graded vector spaces, whose graded Euler characteristic is the $n$-color polynomial.  Readers familiar with Khovanov homology will recognize this construction  (cf. \cite{Kho,DB1}).  The important result of this section, \Cref{MainTheorem:Bigraded-n-color-homology}, is in defining the chain complex $(C^{i,j},\del)$ for a perfect matching graph and showing $\del^2=0$.  The bigraded $n$-color homology is then the homology of this chain complex.

\subsection{Finite dimensional graded vector spaces} 
\label{subsection:finite-dim-graded-vs}
For bigraded $n$-color homology, we work with the $n$-dimensional algebra/vector space $V=\mathbbm{k}[x]/(x^n)=\langle 1, x, x^2, \dots, x^{n-1}\rangle$ where $\mathbbm{k}$ is any ring where $\sqrt{n}$ is defined. Later, in order to write down a color basis for the chain complex, $\mathbbm{k}=\BC$ will be taken. The quantum gradings of $V$ depend on whether $n$ is even or odd. Define the gradings as shown in \Cref{table:quantum-gradings}.

\medskip
\begin{table}[H]
\begin{tabular}{|l|l|}
\hline
{\bf $n$ even, $m=n/2$} & {\bf $n$ odd, $m=(n-1)/2$ }\\ \hline
$\deg 1 =m$ & $\deg 1 = m$ \\ 
$\deg x = m-1$ & $\deg x = m-1$\\
\ \ $\vdots$ & \ \ $\vdots$\\ 
 $\deg x^m = 0$ &  $\deg x^m = 0$\\
\ \ $\vdots$ & \ \ $\vdots$\\ 
$\deg x^{n-1} = 1-m$ & $\deg x^{n-2} = 1-m$\\
\mbox{} & $\deg x^{n-1} = -m$\\
\hline
\end{tabular}
\caption{Quantum gradings of $V$ when $n$ is even or odd.}
\label{table:quantum-gradings}
\end{table}

\medskip

The key to understanding the gradings in the table above is that  $\deg x^m = 0$ for  $n$ even or odd.  This ``centers'' the degrees around $x^m$, which will become important later in defining the $\Delta$ and $\eta$ maps between vector spaces.

Recall that the {\em graded (or quantum) dimension}, $\qdim$, of a graded vector space $V=\oplus_i V^i$ is the polynomial in $q$ defined by $$\qdim(V)=\sum_i q^i\dim(V^i).$$

\noindent For a graded vector space $V$, we can shift the grading up or down by $\ell$ by $(V\{\ell\})^i=V^{i-\ell}$.
Clearly, $\qdim(V\{\ell\}) = q^\ell\cdot\qdim(V)$.  The  $\qdim$ is a polynomial with integer powers. Thus, for example, when $n$ is even,  $$\qdim V^{\otimes k} = (q^m+\cdots + 1 +\cdots q^{1-m})^k.$$
Compare this formula to Equation~\ref{eq:immersed_circle} and \Cref{eqn:general_state-sum}.

\subsection{The differential chain complex for $\Gamma_M$}  
\label{subsection:differential}

We are now ready to associate graded vector spaces to the states of a perfect matching graph $\Gamma_M$ to build the chain complex for the $n$-color homology.  Recall from \Cref{section:smoothing-states-hypercubes} that $M=\{e_1,\dots,e_\ell\}$, $\alpha \in \{0,1\}^\ell$, and $k_\alpha$ is the number of circles in the state $\Gamma_\alpha$. Associate the  vector space $\large V_\alpha = \large V^{\otimes k_\alpha}\!\!\left\{m|\alpha|\right\}$ to the state $\Gamma_\alpha$.
For example, in Figure~\ref{fig:theta_3}, the vector space associated to $\Gamma_{(1,1,0)}$ for $9$-color homology is $V_{(1,1,0)} = V^{\ot2}\{8\}$.

Define the complex $C^{*,*}(\Gamma_M)$ by

$$C^{i,*}(\Gamma_M)=\bigoplus_{\substack{\alpha\in\{0,1\}^n \\ i=|\alpha|}}V_\alpha.$$
The internal grading ($q$-grading) is defined by the grading of the elements in $V_\alpha$.  The homological grading $i$ is also integer valued.   For an element $v\in V_\alpha\subset C^{*,*}(\Gamma_M)$, the homological grading $i$ and the $q$-grading $j$ satisfy:
\begin{eqnarray*}
i(v) & =& |\alpha|,\\
j(v) &=& \deg(v)+m|\alpha|,
\end{eqnarray*}
where $\deg(v)$ is the degree of $v$ as an element of  $V^{\otimes k_\alpha}$ of $V_\alpha=V^{\otimes k_\alpha}\{m|\alpha|\}$ before shifting the grading by $m|\alpha|$. The complex is trivial outside of $i=0, \dots, \ell$.

The differential can now be defined. Each  $C^{i,*}(\Gamma_M)$ is the direct sum of vector spaces of the hypercube given by $i=|\alpha|$.  The edges $\zeta_{\alpha\alpha'}$ in the hypercube of states (see \Cref{section:smoothing-states-hypercubes}) correspond to maps between the graded vector spaces $V_\alpha \subset C^{i,*}(\Gamma_M)$ to vector spaces in $V_{\alpha'}\subset C^{i+1,*}(\Gamma_M)$, where $|\alpha'|=|\alpha|+1 =i+1$.  On the level of vector spaces, the directed segment $\zeta_{\alpha\alpha'}$ between $\Gamma_\alpha$ and $\Gamma_{\alpha'}$ in the hypercube corresponds to a linear map, $\del_{\alpha\alpha'}:V_\alpha \rightarrow V_{\alpha'}$.  For example, in $9$-color homology, the edge $(\ast, 0,0)$ in \Cref{fig:theta_3} that takes $\alpha = (0,0,0)$ to $\alpha' = (1,0,0)$ corresponds to the map $$\del_{\alpha\alpha'}: V^{\otimes 4}\{0\} \ra V^{\otimes 3}\{4\}$$
where $V=\langle 1, x, x^2, \dots, x^8\rangle$. Unlike TQFTs in knot theory in which  the $\dim V$ must be two, the dimensions for graphs can grow large quickly: $\del_{\alpha\alpha'}$ in this example is a map from a 6,561 dimensional vector space to a 729 dimensional vector space. We will show in \Cref{section:filtered-n-color-homology} how filtered $n$-color homology can be used to make calculations with $\del_{\alpha\alpha'}$ simpler, manageable, and more meaningful with respect to edge- and face-colorings.

To define $\del_{\alpha\alpha'}$,  note that each circle $C$ in the state $\Gamma_\alpha$ has a corresponding vector space $V_C=\mathbbm{k}[x]/(x^n)$ as one of the tensors in  $V^{\otimes k_\alpha}$. The process of replacing the $0$-smoothing in $\Gamma_\alpha$ with the $1$-smoothing in $\Gamma_{\alpha'}$ either fuses two of these circles in $\Gamma_\alpha$ together, splits one circle into two, or introduces a double point to a circle.  The corresponding linear maps $\del_{\alpha\alpha'}$ between the vector spaces are determined by these three processes (here $m=n/2$ if $n$ is even, $m=(n-1)/2$ if $n$ is odd):
\begin{enumerate}
\item If the process fuses two circles $C_1,C_2$ in $\Gamma_\alpha$ into one circle $C'_1$ in $\Gamma_{\alpha'}$, define a multiplication map $m:V_{C_1}\otimes V_{C_2} \rightarrow V_{C'_1}$ for this situation by multiplication in the algebra $V=\mathbbm{k}[x]/(x^n)$, that is, $m(a\ot b) = a\cdot b$.
\item If the process splits one circle $C_1$ in $\Gamma_\alpha$ into two circles $C'_1,C'_2$ in $\Gamma_{\alpha'}$,  define a map $\Delta: V_{C_1} \rightarrow V_{C'_2}\ot V_{C'_3}$ by comultiplication: $$\Delta(x^k)= \sum_{\substack{i+j=k+2m\\0\leq i,j \leq n-1}} x^i \ot x^j.$$
\item If the process introduces a double point in a circle $C_1$ in $\Gamma_\alpha$ to get a circle $C'_1$ in $\Gamma_{\alpha'}$, define a map $\eta:V_{C_1}\rightarrow V_{C'_1}$ by: $$\eta(x^k) = \sqrt{n}\cdot x^{k+m}.$$
\end{enumerate}
The map $\del_{\alpha\alpha'}:V_\alpha\rightarrow V_{\alpha'}$ is defined on basis elements of $V_\alpha$ as the tensor product of maps given by the identity on the vector spaces associated with circles that do not change from $\Gamma_\alpha$ to $\Gamma_{\alpha'}$, and either $m, \Delta$ or $\eta$ on the vector spaces associated to circles that are modified by the change from a $0$-smoothing in $\Gamma_\alpha$ to a $1$-smoothing $\Gamma_{\alpha'}$. This map is then extended linearly.

\begin{remark}
Note that, because of the definition of $\eta$, it is necessary to choose a ring $\mathbbm{k}$ like $\BR$ or $\BC$ when $n$ is not a perfect square. For the filtered $n$-color homology, it is useful to choose $\mathbbm{k}=\BC$.  
\end{remark}

\begin{remark}\label{rem:different-than-khovanov}
The $\Delta$ map for the $2$-color homology is different than what one encounters in Khovanov theory \cite{Kho} or the first author's $2$-factor homology theory \cite{BaldCohomology}. In those theories, $\Delta(1) = 1\ot x + x\ot 1$ and $\Delta(x) = x\ot x$. However, in bigraded $2$-color homology, $\Delta(1) = x\ot x$ and $\Delta(x) =0$. This is because we are using a shifted comultiplication (see \Cref{def:HENFA}) in a hyperextended Frobenius algebra in our TQFT instead of the counital comultiplication in the standard Frobenius algebra used in knot theory (see \Cref{section:UnorientedTQFTTheory}).
\end{remark}

\begin{proposition} $\del_{\alpha\alpha'}$ preserves the quantum degree, i.e., for $v\in V_\alpha \subset C^{*,*}(\Gamma_M)$, $j(v) = j(\del_{\alpha\alpha'}(v))$.\label{prop:preserve-quantum-degree}
\end{proposition}

Readers familiar with Khovanov homology \cite{Kho} in knot theory should recognize the maps $m$ and $\Delta$.  The extra map, $\eta:V\ra V$, is due to the fact that states of perfect matching graphs have immersed circles rather than simple embedded circles as in Khovanov homology.  Immersed circles (circles with ``virtual crossings'') do show up in virtual knot theory.  As was noted in Proposition 1 of \cite{BKR}, the partial differential $\del_{\alpha\alpha'}$ is uniquely determined by the condition that $\del_{\alpha\alpha'}$ preserves the $q$-grading on enhanced states. In the case of virtual knot theory (cf. \cite{BaldKauffMc}), this means that $\del_{\alpha\alpha'}=\eta$ must be the zero map. It is a unique phenomenon that in ribbon graphs the $\eta$ map is nonzero without any special modifications to the chain complex or differentials. Hence, the $n$-color homology is an example of how TQFTs can differ in graph theory from those found in knot theory.

\begin{proof}
Since $\del_{\alpha\alpha'}$ for each $\zeta_{\alpha\alpha'}$ is one of three maps, $m$, $\Delta$, or $\eta$, the proposition is proven if it is true for each map.  The map $\eta$ takes $x^k$ of degree $m-k$ to $\sqrt{n}x^{k+m}$ of degree $m-(k+m)=-k$ in $V^{\ot k_{\alpha'}}$.  However, to get $V_{\alpha'}$, this vector space is shifted up by $\{m\}$ from where $V_\alpha$ was shifted ($m|\alpha'| = m|\alpha| +m$).  Therefore, $j(x^k)=\deg(x^k)+m|\alpha| = m-k + m|\alpha|$ and $j(\del_\eta(x^k)) = \deg(\sqrt{n}x^{k+m}) + m|\alpha'|= -k + m + m|\alpha|$. Hence $j(x^k) = j(\del_\eta(x^k))$.

The other two maps preserve the quantum degree by similar reasoning.  For example, in the case of $\Delta$, note that a different way to describe $\Delta(x^k)$ is the sum of all monomials $x^i\ot x^j$ of degree $m-i+m-j = -k$.
\end{proof}

The differential, $\del^i:C^{i,*}(\Gamma_M) \rightarrow C^{i+1,*}(\Gamma_M)$, is defined as the sum of appropriate $\del_{\alpha\alpha'}$'s.  For $v\in V_\alpha \subset C^{i,*}(\Gamma_M)$, 
\begin{equation}
\label{eqn:delta-definition}
\del^i(v) = \sum_{\zeta_{\alpha\alpha'}} \sign(\zeta_{\alpha\alpha
'}) \del_{\alpha\alpha'}(v).
\end{equation}
where $\sign(\zeta_{\alpha\alpha'}) = (-1)^{\#\{\mbox{$1$'s to the left of $\ast$ in $\zeta_{\alpha\alpha'}$}\}}$.

We now have enough to prove  the main theorem of this section:

\begin{theorem}  $(C^{i,*}(\Gamma), \del^i)$ is a  chain complex with differential that increases the homological grading by one and preserves the quantum grading, i.e., it has bigrading $(1,0)$. \label{thm:A-Chain-Complex}
\end{theorem}

\begin{proof}
To show that the square of the differential is zero, each diagram of maps corresponding to each possible face in the hypercube of states must be shown to commute.  The $\sign(\zeta_{\alpha\alpha'})$ is chosen to make each face anticommute and thus $\del^{i+1}\circ\del^i = 0$ as desired.

Recall the standard TQFT/Frobenius algebra argument in Khovanov homology (cf. \cite{DB1}): since saddles appear at different soothing sites and commute in the cobordism category, the induced maps must also commute. We can make a similar appeal using the TQFT of this homology, which is carefully developed in \Cref{section:UnorientedTQFTTheory}. Since the TQFT is new in the literature, it takes several pages to build in that section. Therefore, we give shorter arguments here using examples and calculations in the algebra so that the reader can get a feel for why they work in general.

Each possible pair of compositions of the maps $m, \Delta$ and $\eta$ can be analyzed on a case-by-case basis and be shown to lead to commuting diagrams, or ruled out as possible diagrams by the Jordan curve theorem.  Fortunately, all diagrams involving only $m$ and $\Delta$ commute due to what is already known about TQFTs and Frobenius algebras \cite{Kho}. The maps are different than in Khovanov (cf. \Cref{rem:different-than-khovanov}), but the key property for why the diagrams commute in Khovanov homology remains true for our diagrams: $\Delta$ of a basis element is the sum of {\em all possible} homogeneous tensors $x^i\ot x^j$ of a given degree.  This is enough to insure these diagrams commute.

The remaining diagrams are interactions of $m$ and $\Delta$ with the $\eta$ map.  An analysis of possible commuting diagrams involving $\eta$ with $m$ and $\Delta$ shows that they either (1) do not come from a hypercube face because they violate the Jordan curve theorem, (2)  commute due to the fact that the degree changes by $2m$ over two edges, or (3) they are the diagram $\eta\circ \eta = m\circ \Delta$.  

In Case (2), the change in degree by $2m$ over a composition of two maps is enough to zero out the extraneous terms.  For example, for $n=5$, the diagram $(\eta \ot Id) \circ \Delta = \Delta \circ \eta$ commutes nontrivially for $1, x, x^2$ and is trivial (both compositions are the zero map) for all other elements. The $\eta$ map is multiplication by $\sqrt{5}x^2$ in $V$, which is enough to leave only the homogeneous elements of the correct degree as on the second line in the equation below:
\begin{eqnarray*}
(\eta \ot Id)\circ \Delta (1)&=& (\eta\ot Id)(1\ot x^4+x^4\ot1+x\ot x^3+x^3\ot x +x^2\ot x^2) \\ 
&=& \sqrt{5}\left(x^2\ot x^4+0+x^3\ot x^3 +x^4\ot x^2\right),\\
\Delta\circ \eta(1) &=& \Delta(\sqrt{5} x^2) = \sqrt{5}(x^2\ot x^4+x^3\ot x^3+x^4\ot x^2).
\end{eqnarray*}
The other diagrams of this case commute using the formula $\Delta(v) = (v\ot 1)\Delta(1)$.

The last case, $\eta \circ\eta = m\circ\Delta$, is well-known to virtual knot theorists. It also appears frequently in hypercubes of perfect matching graphs. An example of such a diagram is shown in \Cref{fig:P_3_A-squared-ex}. First, since $\eta\circ\eta$ is multiplication by $n x^{2m}$ in $V$, $\eta\circ\eta(x^k) = 0$ for all but when $k=0$ and $n$ is odd. In that situation, $\Delta(1) = 1\ot x^{n-1} + x\ot x^{n-2} +\cdots +x^{n-2}\ot x + x^{n-1}\ot 1$, which is the sum of $n$ terms.  Hence $m(\Delta(1))= n x^{n-1}$. Since $n-1=2m$, $m\circ \Delta = \eta\circ\eta$.

\begin{figure}[h]
\psfragscanon
\psfrag{A}{$A$}
\psfrag{m}{$m$}
\psfrag{D}{$\Delta$}
\includegraphics[scale=.5]{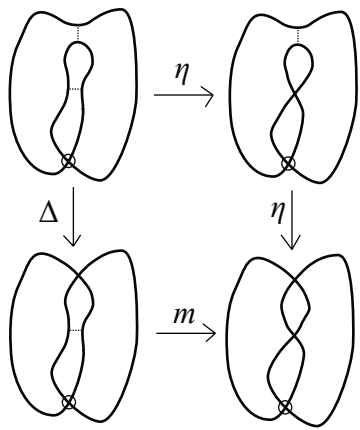}
\caption{An example of states that lead to a $\eta\circ\eta = m\circ\Delta$ commuting diagram. The dotted lines represent the locations of $0$-smoothings.}
\label{fig:P_3_A-squared-ex}
\end{figure}

Finally,  \Cref{prop:preserve-quantum-degree} shows that the bigrading of $\del^i$ is $(1,0)$. 
\end{proof}

\begin{remark} The choice of $\sqrt{n}$ in the definition of $\eta$ stems from the need to make $\eta\circ\eta =m\circ\Delta$ commute, which is why $\mathbbm{k}$ needs to be a ring that includes $\sqrt{n}$, like $\BR$ or $\BC$. If $n$ is a perfect square, like $n=4$, one could choose to work with $\mathbbm{k}=\BZ$ instead. Indeed, for reasons that will become evident shortly, $n=4$ is an interesting case to explore with $\mathbbm{k}=\BZ$. 
\end{remark}

\subsection{The bigraded $n$-color homology definition and proof of \Cref{MainTheorem:Bigraded-n-color-homology}} 
\label{section:cohomology-of-graphs-with-perfect-matchings}
We are now ready to define the bigraded $n$-color homology of a perfect matching graph $\Gamma_M$ of the pair $(G,M)$.  

\begin{definition} Let $(G,M)$ be a trivalent graph $G$ with perfect matching $M$.  Let $\Gamma_M$ be a perfect matching diagram of it.  The {\em bigraded $n$-color homology of $\Gamma_M$}  is
$$CH_n^{i,j}(\Gamma_M;\mathbbm{k}) = \frac{\ker \del:C^{i,j}(\Gamma_M) \ra C^{i+1,j}(\Gamma_M)}{\Ima \del:C^{i-1,j}(\Gamma_M) \ra C^{i,j}(\Gamma_M)}.$$
\end{definition}

Checking \Cref{Thm:ribbons} for perfect matching diagrams shows that the homology depends only on the ribbon graph $\Gamma$ up to equivalence and perfect matching $M$.  Therefore the bigraded $n$-color homology is an invariant of the perfect matching graph itself, which proves the first part of \Cref{MainTheorem:Bigraded-n-color-homology}.  The second part of the theorem, that the graded Euler characteristic is equal to the $n$-color polynomial, follows from the definition of $V_\alpha$, the gradings of $C^{*,*}(\Gamma_M)$, and that the $q$-dimension of $V$ is the same as the expressions in Equation~\ref{eq:immersed_circle}.

We do not know if the bigraded $n$-color homology is invariant under flip moves, or more generally, $2$-isomorphism moves. We leave this as a question for future research:

\begin{question}[See Baldridge~\cite{BaldCohomology}] Is the bigraded $n$-color homology of a perfect matching graph invariant under flip moves ($2$-isomorphism moves)? \label{question:invariant-under-flip-moves-q}
\end{question}

If it was invariant under $2$-isomorphism moves, this invariant would provide interesting information about $2$-cycles and perfect matchings in an abstract trivalent graph.

\subsection{The bigraded $n$-color homology of any ribbon graph}  The definition of bigraded $n$-color homology in the previous  subsection was for a perfect matching graph of any trivalent graph and {\em any} perfect matching for that graph. This can be used to distinguish different perfect matchings on the same graph, for example.  In this subsection, we consider the blowup of a ribbon graph together with its canonical perfect matching to get an invariant of the ribbon graph itself. We then explore the meaning of the graded Euler characteristic of this invariant.

\begin{definition} Let $G(V,E)$ be any connected graph (with any valence at each vertex), and let $\Gamma$ be a ribbon graph of $G$.  Then the {\em bigraded $n$-color homology of the ribbon graph $\Gamma$} is $$CH_n^{*,*}(\Gamma;\mathbbm{k}) : = CH^{*,*}_n(\Gamma^\flat_E;\mathbbm{k}).$$
\end{definition}

The graded Euler characteristic of the bigraded $n$-color homology of a ribbon graph $\Gamma$ is the $n$-color polynomial of $\Gamma$ (see~\Cref{defn:n-color-poly-of-a-graph}). In particular, substituting $q=1$ into the $n$-color polynomial and equating the $n$-color number with the Penrose polynomial gives the following proposition:

\begin{proposition}\label{prop:EulercharPenrose}
Let $G(V,E)$ be any connected graph (with any valence at each vertex), and let $\Gamma$ be a ribbon graph of $G$. Then $$\chi_q(CH^{*,*}_n(\Gamma;\mathbbm{k}))(1) = P(\Gamma,n)$$
for any integer $n$ with $n>1$.
\end{proposition}

Thus, \Cref{thm:penrose-polynomial-values} implies that the  bigraded $n$-color homology is nontrivial when the ribbon graph is planar and $n=2$ (when it is Eulerian), $n=3$, and $n=4$ (by the four color theorem).  We will see that it is nontrivial for all $n>4$ for all plane graphs.  It can be  nontrivial for non-plane graphs even when the Penrose polynomial is identically zero, see \Cref{Theorem:mainthm-four-color} or \Cref{theorem:CH-not-zero}. For example, the graded Euler characteristic of the $3$-color homology of $K_{3,3}$ is given in \Cref{eqn:3-color-poly-of-blowup-of-K33}, which implies the bigraded $3$-color homology is nontrivial. In the next section, we will explore the color homologies of plane graphs and the $K_{3,3}$ ribbon graph. 

\subsection{Examples of bigraded $n$-color homology for $n=2,3,$ and $4$}\label{sec:n-color-homology-examples-for-n-2-3-4}
In this section we work out a few  examples. These examples will help provide intuition for our work with the filtered $n$-color theory as the infinity page of a spectral sequence (\Cref{MainTheorem:spectralsequence}).  

\begin{example}[An example of bigraded $2$-color homology]\label{sec:example-of-2-color-homology} The bigraded $2$-color homology uses the maps on $V=\mathbbm{k}[x]/(x^2)$:
\begin{eqnarray}\nonumber
m_2(1\ot 1)=1, \ \ \ m_2(1\ot x) =x, &\mbox{} &  m_2(x\ot1)  = x, \ \ \ m_2(x\ot x)=0,\\ \nonumber
\Delta_2(1)=x\ot x, &\mbox{} & \Delta_2(x)=0, \\ \nonumber
\eta_2(1)=\sqrt{2} x, &\mbox{}& \eta_2(x)=0.
\end{eqnarray}
Note that these maps are different from Khovanov homology (cf. \Cref{rem:different-than-khovanov}) or even Lee's $\Phi$ map \cite{LeeHomo}.

Consider the Eulerian plane graph $\Gamma$ that has two vertices and two edges between them.  This plane graph has two faces, one at infinity and one between the edges. The blowup is a $\theta_2$ graph with perfect matchings the original edges of the graph, in fact, it is $\Gamma_2$ in \Cref{fig:theta_2_with_different_pms}.  The hypercube of the blowup is made up of one state in homological degree zero with two circles, two states in degree one with one circle each, and one state in degree three with two circles.  The edges of the hypercube consist of two $m$ maps followed by two $\Delta$ maps. Computing the bigraded $2$-color homology of blowup is then:

\begin{table}[H]
\renewcommand{\arraystretch}{1}
\caption{The bigraded $2$-color homology of $\Gamma_2$ in \Cref{fig:theta_2_with_different_pms} is $CH^{*,*}_2(\Gamma)$.}
\centering
\begin{tabular}{| c || c | c | c | c | c |}
\hline
$4$ & & &  $<1\ot 1>$ \mbox{}  \\ \hline
$3$ & & & $<1\ot x, x\ot 1>$  \\ \hline
$2$ & &   &  \\ \hline
$1$ & $<1\ot x - x\ot 1>$ & $<(x,0)>$   &   \\ \hline
$0$ & $<x\ot x>$ &&\\ \hhline{|=||=|=|=|=|}
\diagbox[dir=SW]{$j$}{$i$} & $0$& $1$ & $2$ \\ \hline
\end{tabular} \label{table:2-color-homology-of-theta-2}
\end{table}

Observe that $[\Gamma]_2 = 4$. The all-zero state $\Gamma_{\vec{0}}$ is two circles, which if colored with two different colors, say red and blue,  accounts for the two $2$-face colorings of $\Gamma$.  Where are the other two colorings in $[\Gamma]_2$?  They emerge in the homological degree $2$ state (the all-one smoothing state) after passing to filtered $2$-color homology described in the next section: the $E_1$ differential $d_1$ will map $d_1[(x,0)]_1= [1\ot x+ x\ot 1]_1$ in the bigraded $2$-color homology in \Cref{table:2-color-homology-of-theta-2}.  We will see this again in working out the $3$-color and $4$-color homology below.

\end{example}

\begin{example}[An example of bigraded $3$-color homology]\label{sec:example-of-3-color-homology} The bigraded $3$-color homology uses the following maps on $V=\mathbbm{k}[x]/(x^3)$:

\hspace{.5in}\begin{minipage}{2.in}
\begin{eqnarray}\nonumber
\Delta_3(1) &=&1\ot x^2 + x^2 \ot 1 + x\ot x \\ \nonumber
\Delta_3(x) &=& x\ot x^2 +x^2\ot x\\ \nonumber
\Delta_3(x^2) &=& x^2\ot x^2\\ \nonumber
\end{eqnarray}
\end{minipage}\hspace{.5in}\begin{minipage}{2in}\noindent
\begin{eqnarray}\nonumber
\eta_3(1)&=& \sqrt{3} x \\ \nonumber
\eta_3(x) &=& \sqrt{3} x^2\\ \nonumber
\eta_3(x^2) &=& 0\\ \nonumber
\end{eqnarray}
\end{minipage}

\noindent with the multiplication map given by multiplication in $V$.

We compute the bigraded $3$-color for the $K_{3,3}$ perfect matching graph $\Gamma_M$ displayed in \Cref{fig:K33-is-zero}. 

\begin{table}[H]
\renewcommand{\arraystretch}{1}
\caption{The bigraded $3$-color homology of $\Gamma_M$ of $K_{3,3}$ in \Cref{fig:K33-is-zero}.}
\centering
\begin{tabular}{| c || c | c | c | c | c |}
\hline
$5$ & & &  & $\mathbbm{k}$ \\ \hline
$4$ & & &  &$\mathbbm{k}\oplus\mathbbm{k}$  \\ \hline
$3$ & & & & $\mathbbm{k}\oplus\mathbbm{k}$ \\ \hline
$2$ & &  $\mathbbm{k}$  &  &$\mathbbm{k}$ \\ \hline
$1$ &  & $\mathbbm{k}$  &  & \mbox{} \\ \hline
$0$ & $\mathbbm{k}$ & &  & \mbox{}  \\ \hline
$-1$ & $\mathbbm{k}\oplus\mathbbm{k}\oplus \mathbbm{k}$ & &  & \mbox{}  \\ \hline
$-2$ & $\mathbbm{k}\oplus\mathbbm{k}\oplus \mathbbm{k}$ & &  & \mbox{}  \\ \hline
$-3$ & $\mathbbm{k}$  && &\mbox{} \\ \hhline{|=||=|=|=|=|}
\diagbox[dir=SW]{$j$}{$i$} & $0$& $1$ & $2$ & $3$ \\ \hline
\end{tabular} \label{table:3-color-homology-of-K33}
\end{table}

Note that the graded Euler characteristic of this homology is the $3$-color polynomial in \Cref{eqn:3-color-poly-of-K33}. The $-2q^2$ term in this polynomial is split between degree $1$ and degree $3$ in the $3$-color homology, which shows that  information is lost when taking the graded Euler characteristic of bigraded $3$-color homology. This idea holds in general: the information in the bigraded $n$-color homology is far richer than the information in the $n$-color polynomial, much like Khovanov homology is a more refined invariant than the Jones polynomial (cf. \cite{KM4}). 

One can guess what the filtered $3$-color homology is  due to the empty column in degree $2$ of \Cref{table:3-color-homology-of-K33}: 6 colors in homological degree 0 and $6$ colors in degree 3. So while the Penrose formula from \cite{Kauffman, Penrose} computes $0$ for the $K_{3,3}$, the filtered $3$-color homology is picking up all $12$ colors but in different degrees that differ by an odd number of homological degrees. 
\end{example}

\begin{example}[An example of bigraded $4$-color homology] \label{sec:example-of-4-color-homology}The bigraded $4$-color homology uses the following maps on $V=\mathbbm{k}[x]/(x^4)$:

\hspace{.5in}\begin{minipage}{2.in}
\begin{eqnarray}\nonumber
\Delta_4(1) &=&x\ot x^3 + x^3 \ot x + x^2\ot x^2 \\ \nonumber
\Delta_4(x) &=& x^2\ot x^3 +x^3\ot x^2\\ \nonumber
\Delta_4(x^2) &=& x^3\ot x^3\\ \nonumber
\Delta_4(x^3) &=& 0 \\ \nonumber
\end{eqnarray}
\end{minipage}\hspace{.5in}\begin{minipage}{2in}\noindent
\begin{eqnarray}\nonumber
\eta_4(1)&=& 2x^2 \\ \nonumber
\eta_4(x) &=& 2x^3\\ \nonumber
\eta_4(x^2) &=& 0\\ \nonumber
\eta_4(x^3) & = & 0\\ \nonumber
\end{eqnarray}
\end{minipage}

\noindent with the multiplication map given by multiplication in $V$.

\begin{remark} For $n=4$, one can use the ring $\mathbbm{k}=\BZ$ instead of a field.  This opens the possibility for torsion in $4$-color homology.  An interesting question is whether there are examples of ribbon graphs that have torsion in their $4$-color homology and what that might mean for such graphs. In Khovanov theory, torsion plays an important role. 
\end{remark}

To compute examples of this homology, unless the graph is small, requires a computer program. (Such a program is available upon request.) A small computable example is the plane graph $Loop$ with one vertex and one edge.  The blowup of this graph is the $\theta$ graph with perfect matching edge the original edge of the graph.  There are two states for this perfect matching graph: one state in degree $0$ with two circles, and one in degree $1$ with one circle.  There are 16 generators in degree $0$ ($1\ot 1$, $1\ot x$, $x \ot 1$, $\dots$, $x^3\ot x^3$), and $m$ maps $1\ot 1 \ra 1$, $1\ot x \ra x$, $1\ot x^2 \ra x^2$, $1 \ot x^3 \ra x^3$, and so on. From this the dimension of the bigraded $4$-homology can be calculated to be $16-4=12$, i.e., $CH_4^{0,*}(Loop;\mathbbm{k}) = \mathbbm{k}^{12}$ and is zero otherwise. Of course, the number of $4$-face colorings of $Loop$ is $12$. That the dimension of the homology reports exactly this number is a consequence of the fact that the bigraded $4$-color homology is isomorphic to the filtered $4$-color  homology in this case and the fact that the homology is only supported in homological grading zero.  However, this count in homological grading zero continues to hold for all graphs and for all $n$ as Statement (3) of \Cref{MainTheorem:Colorings-of-State-Graphs} describes.  Note that in this example it is easy to see that $\dim_{\mathbbm{k}} CH^{0,*}_n(Loop;\mathbbm{k}) = n(n-1)$, which is the number of $n$-face colorings of the $Loop$ plane graph.  
\end{example}

\subsection{Distinguishing ribbon graphs with the same underlying graph} There are two ribbon graphs in \Cref{Fig:ribbons} of the $\theta$ graph: $\Gamma_1$ and $\Gamma_2$. While there are many ways to distinguish these two ribbon graphs, they can also be used as an illustrative example of how bigraded $n$-color homology can be used to distinguish ribbon graphs.

First, we can choose any edge on either one to be a perfect matching (the other choices lead to  equivalent perfect matching graphs). The bigraded $4$-color homology of $\Gamma_1$ with a perfect matching $M_1$ was calculated just above: $CH^{0,*}_4(\Gamma_1, M_1) = \mathbbm{k}^{12}$ and zero elsewhere.  The states of $(\Gamma_2, M_2)$ are the reverse of $(\Gamma_1,M_1)$: there is one circle in the homological degree zero state and two circles in degree one. Calculating the homology of $(\Gamma_2,M_2)$ gives $CH^{0,*}_4(\Gamma_2,M_2) = \mathbbm{k}$ and $CH^{1,*}_4(\Gamma_2,M_2) = \mathbbm{k}^{13}$.

Thus, $\Gamma_1$ is different than $\Gamma_2$. This is a simple example that still illustrates the general rule: the $n$-color homologies of ribbon graphs, even ribbon graphs of the same genus and same underlying abstract graph, can be used to distinguish ribbon graphs.

We have sufficiently motivated the power of the bigraded $n$-color homology as an invariant of ribbon graphs---the homologies are rich and can be used to distinguish ribbon graphs. In some cases, it calculates the number of $n$-face colorings of the ribbon graph, and even when it does not, it looks like it almost does. In the next section we build the filtered $n$-color homology as the $E_\infty$-page of a spectral sequence based upon the bigraded $n$-color homology. This will show that both of these homologies are more than just powerful invariants, they are also meaningful in what they report about a graph.

\section{\Cref{MainTheorem:spectralsequence}: Filtered $n$-color  homology} 
\label{section:filtered-n-color-homology}

We now define a new operator, $\widehat{\del}:=\del+\widetilde{\del}$, on the chain complex $C^{*,*}(\Gamma_M)$. This map does not preserve the quantum grading, in fact, the $\widetilde{\del}$ differential satisfies
$$\widetilde{\del}:C^{i,j}(\Gamma_M) \ra C^{i+1,j+n}(\Gamma_M).$$
Still, for $v \in C^{*,*}(\Gamma_M)$, $$j(\widehat{\del}(v)) \geq j(v).$$
Thus $j$ can be used to filter the original chain complex, turning $\widehat{\del}$ into a filtered differential operator. This section describes how to use this new operator to define a spectral sequence and a homology theory that bears many similarities to Lee homology \cite{LeeHomo}.  In the context of graphs and in particular planar graphs, however, the Lee-type homology theory calculates something striking about the original graph: it counts $n$-face colorings on different ribbon graphs of the underlying abstract graph.

\subsection{The operator $\widetilde{\del}$} \label{subsection:defintion-of-tilde-del} We continue to work with the vector space $V=\mathbbm{k}[x]/(x^n)=\langle 1, x, x^2, \dots, x^{n-1}\rangle$, but now for convenience, we choose $\mathbbm{k}=\BR$ or $\mathbbm{k}=\BC$.  Later, when changing to the color basis, we will restrict further to $\mathbbm{k}=\BC$. The quantum gradings of $V$ depend on whether $n$ is even or odd: set $m=n/2$ if $n$ is even and $m=(n-1)/2$ if $n$ is odd as in previous sections.

On the chain level, $\widetilde{\del}$ is defined in the same way as $\del$.  Let $\zeta_{\alpha\alpha'}$ be an edge in the hypercube of states from $\Gamma_\alpha $ to $\Gamma_{\alpha'}$.  This gives rise to a linear map, $\widetilde{\del}_{\alpha\alpha'}:V_\alpha \rightarrow V_{\alpha'}$, which is defined based upon the circles in $\Gamma_\alpha$:

\begin{enumerate}
\item If $\zeta_{\alpha\alpha'}$ represents the fusing of two circles in $\Gamma_\alpha$ into one circle in $\Gamma_{\alpha'}$, define  multiplication by
$$\widetilde{m}(x^i\ot x^j) = x^{i+j-n} \mbox{\ \  if $i+j\geq n$},$$
and zero if $i+j<n$.
\item If $\zeta_{\alpha\alpha'}$ represents the splitting of a circle  in $\Gamma_\alpha$ into two circles in $\Gamma_{\alpha'}$,  define comultiplication by $$\widetilde{\Delta}(x^k)= \sum_{\substack{i+j=k+2m-n\\0\leq i,j < n}} x^i \ot x^j \mbox{\ \  if $k+2m \geq n$},$$
and zero if $i+2m<n$.
\item If $\zeta_{\alpha\alpha'}$ represents introducing a double point in a circle in $\Gamma_\alpha$ to get a circle in $\Gamma_{\alpha'}$, define the map by $$\widetilde{\eta}(x^k) = \sqrt{n}\cdot x^{k+m-n} \mbox{\ \  if $k+m\geq n$},$$
and zero if $k+m<n$.
\end{enumerate}
Like $\del_{\alpha\alpha'}$, the map $\widetilde{\del}_{\alpha\alpha'}:V_\alpha\rightarrow V_{\alpha'}$ is defined on basis elements of $V_\alpha$ as the tensor product of maps given by the identity on the vector spaces associated with circles that do not change from $\Gamma_\alpha$ to $\Gamma_{\alpha'}$, and either $\widetilde{m}, \widetilde{\Delta}$ or $\widetilde{\eta}$ on the vector space(s) associated to circles that are modified by the change from a $0$-smoothing in $\Gamma_\alpha$ to a $1$-smoothing $\Gamma_{\alpha'}$. Extend this map linearly.

The definition of $\widetilde{\del}^i:C^{i,*}(\Gamma_M) \ra C^{i+1,*}(\Gamma_M)$ is defined in the same way as in \Cref{eqn:delta-definition}.

\begin{remark} Note that when taking the operators $\del$ and $\widetilde{\del}$ together, it is helpful to think of multiplication in $V$ as multiplication in a slightly different algebra: $\widehat{V}=\mathbbm{k}[x]/(x^n-1)$.  This helps to explain why $\widetilde{m}$ and $\widetilde{\eta}$ are defined as they are. Furthermore, $\widetilde{\Delta}$ is the ``complement'' of $\Delta$ in the sense that $(\Delta+\widetilde{\Delta})(x^k)$ is always the sum of exactly $n$ terms whose powers on each term sum up to $k+2m\mod n$. For example, when $n=4$, $$(\Delta+\widetilde{\Delta})(x) = (x^2\ot x^3 +x^3\ot x^2)+(1\ot x + x\ot 1).$$
\end{remark}

\begin{theorem} The map $\widetilde{\del}$ satisfies $\widetilde{\del}^{i+1}\circ \widetilde{\del}^i = 0$ with bidegree $(1,n)$. \label{theorem:tilde-differential}
\end{theorem}

\begin{proof} The  proof follows the same reasoning as \Cref{thm:A-Chain-Complex}.
\end{proof}

\begin{theorem} The map $\widetilde{\del}$ anti-commutes with $\del$, that is, $\widetilde{\del}\circ\del + \del\circ\widetilde{\del} = 0$.
\end{theorem}

\begin{proof} 
The theorem follows if, for each diagram that corresponds to a face of the hypercube diagram, the  statement corresponding to $\del\widetilde{\del} +\widetilde{\del}\del$ for one path is equal to the statement corresponding to $\del\widetilde{\del} +\widetilde{\del}\del$ for the other path through the diagram. A direct call to similar calculations in Lee's proof of her homology (cf. \cite{LeeHomo}) cannot be made since the maps $m$, $\Delta$, $\widetilde{m}$, and $\widetilde{\Delta}$ are different than Lee's maps. 

Some diagrams automatically commute. For example, the face that corresponds to the commutative diagram, $\Delta\circ \eta = \Delta\circ \eta$, requires proving $$\Delta\circ\widetilde{\eta} + \widetilde{\Delta}\circ \eta =\Delta\circ\widetilde{\eta} + \widetilde{\Delta}\circ \eta,$$ which obviously commutes.  Besides these diagrams, the main diagrams that need to be checked (up to simple permutations like changing $\eta\ot Id$ to $Id\ot \eta$ for example) are:

\begin{enumerate}
\item $m\circ(\widetilde{m}\ot Id) + \widetilde{m}\circ (m\ot Id) = m\circ (Id \ot \widetilde{m})+\widetilde{m}\circ(Id \ot m),$
\item $(\Delta \ot Id) \circ \widetilde{\Delta} + (\widetilde{\Delta}\ot Id)\circ \Delta = (Id\ot\Delta)\circ \widetilde{\Delta} + (Id \ot \widetilde{\Delta})\circ\Delta$,
\item $\Delta \circ \widetilde{m} + \widetilde{\Delta}\circ m = (m\ot Id)\circ (Id \ot \widetilde{\Delta}) + (\widetilde{m}\ot Id)\circ(Id \ot \Delta)$,
\item $\eta\circ \widetilde{\eta} + \widetilde{\eta}\circ \eta = m\circ \widetilde{\Delta} + \widetilde{m}\circ \Delta$,
\item $m\circ(\widetilde{\eta}\ot Id) + \widetilde{m}\circ(\eta \ot Id) = \eta\circ \widetilde{m} + \widetilde{\eta}\circ m$, and
\item $\Delta\circ\widetilde{\eta} + \widetilde{\Delta}\circ\eta = (\eta \ot Id) \circ \widetilde{\Delta} + (\widetilde{\eta} \ot Id) \circ \Delta$.
\end{enumerate}

Showing each of these equations hold is a series of calculations from the definitions above and therefore left to the reader to verify. Alternatively, readers who are familiar with Bar Natan's work \cite{BN3} can work through the ideas presented in \Cref{section:UnorientedTQFTTheory} with $t>0$ to see why the theorem holds from a category theory/cobordism perspective.  In particular, \Cref{fig:cobordism-with-U} shows an example of a cobordism associated to a $\widetilde{\Delta}$ map.  We work through three examples to highlight some aspects of these calculations and why they work in general. 

The first three equations are the same equations as in Section 4.2.1 of \cite{LeeHomo} and are true for similar (but slightly different) calculations shown in Lee's paper. We highlight examples of the last three equations, starting with an example of Equation~(4) when $n=4$:

\begin{center}
\begin{tabular}{cccccc}
$n=4$ \ \  & &$\eta\circ\widetilde{\eta}$ \ \  & $\widetilde{\eta}\circ\eta$ \ \  & $m\circ \widetilde{\Delta}$ \ \  & $\widetilde{m}\circ \Delta$\ \  \\
$1$ &$\mapsto$ & $0$ & $4$ & $1$ & $3$\\
$x$ &$\mapsto$ & $0$ & $4x$ & $2x$ & $2x$\\
$x^2$ &$\mapsto$ & $4x^2$ & $0$ & $3x^2$ & $1x^2$\\
$x^3$ &$\mapsto$ & $4x^3$ & $0$ & $4x^3$ & $0$\\
\end{tabular}
\end{center}
See the maps defined in \Cref{sec:example-of-4-color-homology} to help with these calculations.  Note how the $\sqrt{4}$ coefficient on the $\eta$ and $\widetilde{\eta}$ maps in the first two columns of the table compensates for the different coefficient sums in the last two columns. This works the same way for any $n$.

For Equation~(5), we display the example where $n=3$ (cf. \Cref{sec:example-of-3-color-homology} for maps):
\begin{center}
\begin{tabular}{cccccc}
$n=3$ \ \  & & $m\circ(\widetilde{\eta}\ot Id)$ \   & $\widetilde{m}\circ(\eta\ot Id)$ \   & $\eta\circ\widetilde{m}$ \   & $\widetilde{\eta}\circ m$\   \\
$1\ot 1$ &$\mapsto$ & $0$ & $0$ & $0$ & $0$\\
$1\ot x$ &$\mapsto$ & $0$ & $0$ & $0$ & $0$\\
$x\ot 1$ &$\mapsto$ & $0$ & $0$ & $0$ & $0$\\
$x^2\ot 1$ &$\mapsto$ & $1$ & $0$ & $0$ & $1$\\
$1\ot x^2$ &$\mapsto$ & $0$ & $1$ & $0$ & $1$\\
$x\ot x$ &$\mapsto$ & $0$ & $1$ & $0$ & $1$\\
$x^2\ot x$ &$\mapsto$ & $x$ & $0$ & $x$ & $0$\\
$x\ot x^2$ &$\mapsto$ & $0$ & $x$ & $x$ & $0$\\
$x^2\ot x^2$ &$\mapsto$ & $x^2$ & $0$ & $x^2$ & $0$\\
\end{tabular}
\end{center}
In general,  the maps $\eta$ and $\widetilde{\eta}$ act as multiplication by $x^m$ modulo $x^n=1$.  This matches with the multiplication maps $m$ and $\widetilde{m}$, which together work as multiplication modulo $x^n=1$ as well.

For Equation~(6), the example when $n=4$ is given below:
\begin{center}
\begin{tabular}{cccccc}
$n=4$ \ \  & & $\Delta \circ \widetilde{\eta}$ \   & $\widetilde{\Delta}\circ\eta$ \   & $(\eta \ot Id)\circ \widetilde{\Delta}$ \   & $(\widetilde{\eta}\ot Id)\circ \Delta$\   \\[.2cm]
$1$ &$\mapsto$ & $0$ & $1\!\ot\! x^2\!+\! x^2\!\ot\! 1 \! +\! x\!\ot\! x$ & $x^2\ot 1$ & $x\ot x + 1 \ot x^2$\\[.2cm]
$x$ &$\mapsto$ & $0$ & $\substack{x\ot x^2+ x^2 \ot x\\+x^3\ot 1 + 1\ot x^3}$ & $x^2\!\ot\! x \! +\! x^3\! \ot\! 1$ & $1\! \ot\! x^3\! +\! x\! \ot\! x^2$\\[.2cm]
$x^2$ &$\mapsto$ & ${x\!\ot\! x^3}\!+\!{x^3\!\ot\! x}\!+\!{x^2 \!\ot\! x^2}$ & $0$ & $x^2\!\ot\! x^2\!+\!x^3\!\ot\! x$ & $x\!\ot\! x^3$\\[.2cm]
$x^3$ &$\mapsto$ & $x^2\!\ot\! x^3\! + \! x^3\!\ot\! x^2$ & $0$ & $x^3\!\ot\! x^2\!+\! x^2\!\ot\! x^3$ & $0$\\
\end{tabular}
\end{center}
The reason the equation holds for any $n$ is harder to see, but the elements of why it is true can be seen from this case. First, notice that the first two columns start with $\eta$ or $\widetilde{\eta}$, and therefore the result is {\em always} the image of $\Delta$ or $\widetilde{\Delta}$ for some $x^k$. This means that the result is the sum of {\em all} terms of a certain degree, say $\ell$.  The last two columns, after applying $\widetilde{\Delta}$ (third column) or $\Delta$ (fourth column), are also sums of all terms of a certain degree. Applying $\eta\ot Id$ (third column) or $\widetilde{\eta}\ot Id$  (fourth column) to these sums results in terms of degree $\ell$ as well.  

The remaining issue to see in the general case is that the sum of the third column and fourth column is the sum of {\em all} terms of degree $\ell$.  Because one begins with the computation of $\eta\ot Id$ (third column) or $\widetilde{\eta}\ot Id$ (fourth column) with a sum of all terms of a certain degree,  each of these maps exactly compensate for the other in those calculations.  For example, applying the third column operator to $x$ gives:
$$(\eta\ot Id)\circ\widetilde{\Delta}(x) = (\eta\ot Id)(1\ot x + x\ot 1) = x^2\ot x + x^3\ot 1,$$
leaving only the terms $1\ot x^3 $ and $x\ot x^2$ of the $q$-degree $1$ terms. For the fourth column operator,  $\widetilde{\eta}\ot Id$ acts on the sum of all terms of degree $-1$ terms, which picks out exactly the needed degree $1$ terms:
$$(\widetilde{\eta}\ot Id)\circ \Delta(x) = (\widetilde{\eta}\ot Id)(x^2\ot x^3 + x^3\ot x^2) = 1\ot x^3 + x\ot x^2.$$

This pattern holds for all $n$ as can be computed directly from the definitions.
\end{proof}

\subsection{A spectral sequence for $n$-color homology} In \cite{LeeHomo}, Lee introduced a deformation of the Khovanov chain complex by changing the Frobenius algebra to $V=\BQ[x]/(x^2-1)$ and working with a filtered differential.  We can do the same in this paper using the quantum-filtered operator $\widehat{\del}:=\del +\widetilde{\del}$ and $\widehat{V}=\mathbbm{k}[x]/(x^n-1)$, and defining {\em filtered $n$-color homology} to be
\begin{equation}
\widehat{CH}_n^*(\Gamma_M,\mathbbm{k}):=H(\widehat{C}^{*,*}(\Gamma_M), \widehat{\del}),
\end{equation}
where the chain complex $\widehat{C}^{*,*}(\Gamma_M)$ is the the same as $C^{*,*}(\Gamma_M)$ but with each vector space corresponding to a state constructed using $\widehat{V}=\mathbbm{k}[x]/(x^n-1)$ instead of $V=\mathbbm{k}[x]/(x^n)$.  This complex can be viewed as the original complex with additional differentials, resulting in a spectral sequence $E_r(\Gamma_M)$ from the bigraded $n$-color homology $CH_n^{*,*}(\Gamma_M,\mathbbm{k})$ to the filtered $n$-color homology $\widehat{CH}_n(\Gamma_M,\mathbbm{k})$.  

The spectral sequence is constructed as follows: The $E_0$-page is the original bigraded complex $C^{*,*}(\Gamma_M)$ with the usual differential $d_0:=\del$. Thus, the first page $E_1$ is just the bigraded $n$-color homology.  The higher differentials $d_r$ on $E_r(\Gamma_M)$ have bigrading $(1,rn)$ and are based upon $\widetilde{\del}$.

A standard theorem (cf. \cite{McCleary}) of spectral sequences then implies:

\begin{theorem} The $E_\infty$-page of this spectral sequence is isomorphic to $\widehat{CH}_n(\Gamma_M;\mathbbm{k})$.\label{thm:spectral-sequence-for-filtered-homology}
\end{theorem}

We now compute the filtered $n$-color homology of the examples of \Cref{sec:n-color-homology-examples-for-n-2-3-4} using this theorem.

\begin{example} Recall that the blowup of the Eulerian plane graph $\Gamma$ with two vertices and two edges between them in \Cref{sec:example-of-2-color-homology} is $\Gamma_2$ in \Cref{fig:theta_2_with_different_pms}. Inspecting \Cref{table:2-color-homology-of-theta-2}, the spectral sequence for this graph collapses at the $E_2$-page ($d_2$ has grading $(1,4)$). Thus, the $E_\infty$-page is the homology of the $E_1$-page with operator $d_1:=\widetilde{\del}$. The $E_1$-page is given in \Cref{table:2-color-homology-of-theta-2}. Since $d_1$ has grading $(1,2)$, the only potential nonzero map is between gradings $(1,1)$ and $(1,3)$.  In this case, $d_1[(x,0)]_1 =[\widetilde{\Delta}_2(x,0)]_1 = [1\ot x+x\ot 1]_1$. Thus, $\widehat{CH}^i_2(\Gamma,\mathbbm{k}):=\widehat{CH}^i_2(\Gamma_E^\flat,\mathbbm{k})$  is $\mathbbm{k}\oplus\mathbbm{k}$ in homological grading $i=0$ and $i=2$ as was suggested at the end of \Cref{sec:example-of-2-color-homology}.
\end{example}

\begin{example}
The bigraded $3$-color homology of the $K_{3,3}$ perfect matching graph in \Cref{fig:K33-is-zero} is given in \Cref{table:3-color-homology-of-K33} of \Cref{sec:example-of-3-color-homology}. A little bit of work shows that $d_1$ is nonzero from degree $(0,-2)$ to degree $(1,1)$ and also from degree $(0,-1)$ to degree $(1,2)$. The differentials $d_r =0$ for all $r>1$.  Hence, $\widehat{CH}^0_3(K_{3,3};\mathbbm{k})=\mathbbm{k}^6$ and $\widehat{CH}^3_3(K_{3,3};\mathbbm{k}) = \mathbbm{k}^6$. Thus, while the Penrose Formula (the $3$-color number) of $K_{3,3}$ is zero,  the filtered $3$-color homology is counting $6$ Tait colorings in degree $0$ and $6$ Tait colorings in degree 3. The sum, 12, is equal to the number of $3$-edge colorings of $K_{3,3}$, which hints at \Cref{cor:3-edge-color-equals-sum-of-3-face-color} proven later.  Note that the closed surface given by the ribbon graph  for  $K_{3,3}$ in \Cref{fig:K33-is-zero} has exactly six $3$-face colorings as can be computed by hand. 
\end{example}

\begin{example} Recall from \Cref{sec:example-of-4-color-homology} that the dimension of the bigraded $4$-color homology of the $Loop$ graph is $CH_4^{0,*}(Loop;\mathbbm{k}) = \mathbbm{k}^{12}$ and is zero otherwise. Hence, the spectral sequence collapses at the $E_1$-page and the bigraded $4$-color homology is isomorphic to the filtered $4$-color  homology.  
\end{example}

\subsection{The color basis for filtered $n$-color theory} \label{subsection:colorbasis} In the previous section, each example showed that the degree zero homology of the filtered $n$-color homology of the blowup of a ribbon graph counts the number of $n$-face colorings  for the closed surface associated to that ribbon graph. In this section, we introduce a new basis, called the color basis, that can be used to derive this fact.

For the remainder of this paper, we specialize to $\mathbbm{k}=\BC$ when talking about filtered $n$-color homology, taking $\widehat{V}=\BC[x]/(x^n-1)$ as the algebra. Furthermore, we will now only work with the blowup of a ribbon graph to get invariants of the ribbon graph itself, i.e., given a ribbon graph $\Gamma$ of a graph $G(V,E)$, 
\begin{equation}
\widehat{CH}^i_n(\Gamma;\BC) := \widehat{CH}^i_n(\Gamma^\flat_E;\BC).
\end{equation}
This restriction to the blowup is to allow us to work in the setting of coloring $2$-cells (faces) of the $2$-dimensional CW complex $\overline{\Gamma}$ of the ribbon graph $\Gamma$. However, all theorems and proofs below that employ perfect matching graphs $\Gamma_M$ continue to hold for such structures and not just blowups, i.e., the theorems are  not just statements about $\Gamma^\flat_E$.  For such theorems, the results are about coloring collections of certain cycles of the perfect matching graph $\Gamma_M$ instead, which is clearly less intuitive than coloring faces of $\overline{\Gamma}$.

The operator for filtered $n$-color  homology, $\widehat{\del}=\del+\widetilde{\del}$, can be succinctly written using the maps,
\begin{eqnarray} \widehat{m}(x^i \ot x^j) &=&  x^{i+j}, \label{eq:wide-hat-differential-m}\\
\widehat{\Delta}(x^k) &=&  \sum_{\substack{0 \leq i,j < n \\ i+j \equiv (k + 2m) \!\!\!\!  \mod n}} x^i \ot x^j, \nonumber \label{eq:wide-hat-differential-delta}\\
\widehat{\eta}(x^k)&=& \sqrt n x^{k+m}, \nonumber \label{eq:wide-hat-differential-eta}
\end{eqnarray}
for tensor powers of $\widehat{V}$  where $x^n=1$ in the algebra.

To interpret the meaning of the elements in the vector space $\widehat{V}_\alpha$ for a state $\Gamma_\alpha$, it is  advantageous to switch to a different basis.  In this basis, the elements can be thought of as coloring the circles in the state $\Gamma_\alpha$. For the all-zero state, this can then be interpreted as coloring the faces with $n$ different colors of the closed surface associated to the ribbon graph $\Gamma$: the circles in the all-zero state correspond to the boundaries of the disks glued to the ribbon graph to construct the closed surface. First, the definition:

\begin{definition} Let $n$ be a positive integer with $n>1$ and set $\lambda = e^{\frac{2\pi \mathrm{i}}{n}}$. The {\em color basis} of $\widehat{V}=\BC[x]/(x^n-1)$ is
$$c_i := \frac{1}{n}\left(1+ \lambda^i x +\lambda^{2i}x^2+ \lambda^{3i}x^3+\cdots+\lambda^{(n-1)i}x^{n-1}\right)$$
for $0\leq i \leq n-1$. \label{def:colorbasis}
\end{definition}

The $c_i$'s are the different colors of the theory.  Hence, when $n=4$, there are four colors $\{c_0, c_1, c_2, c_3\}$ for filtered $4$-color homology and so on. Also, note that choosing $\mathbbm{k}=\BC$ is now necessary to make the $c_i$'s well-defined for $n>2$ since $\lambda$ is an $n$th root of unity.

\begin{lemma}\label{lem:widehat-maps}
In the color basis, the following equations hold:
\begin{enumerate}
\item $c_i \cdot c_j =  \delta^{ij} c_j,$ hence $\widehat{m}(c_i \ot c_j) =  \delta^{ij} c_j$,
\item $\widehat{\Delta}(c_i) =  n \lambda^{-2mi} c_i \ot c_i,$
\item $\widehat{\eta}(c_i) = \sqrt n \lambda^{-mi} c_i$, and
\item $(\lambda^i x) \cdot c_i =  c_i$.
\end{enumerate}
\end{lemma}

\begin{proof}
Equation (4) follows from the definition of $c_i$ and can be used to prove other equations. For example, Equation~(4) implies that $x^m\cdot c_i = \lambda^{-mi}c_i$, from which Equation~(3) immediately follows. 

Also, from Equation~(4), the following equation holds,
$$(\lambda^{ik}x^k)\cdot c_j = \lambda^{k(i-j)}c_j.$$
Applying each term of $c_i$ to $c_j$ using this equation gives
$$c_i\cdot c_j = \frac{1}{n}\left(1+(\lambda^{(i-j)})^1 + (\lambda^{(i-j)})^2+\cdots+(\lambda^{(i-j)})^{n-1}\right)c_j.$$
Since $\lambda$ is an $n$th root of unity, the sum of the $\lambda$-terms in the expression above is $0$ if $i\not=j$ and $n$ if $i=j$.  

Finally, Equation~(2) is computed as follows:
\begin{eqnarray*}
n\lambda^{-2mi}c_i \ot c_i & = & n\lambda^{-2mi}\left(\frac{1}{n^2} \sum_{0 \leq r,s < n} (\lambda^{ri} x^r) \ot (\lambda^{si} x^s)\right),\\
&=& \frac{1}{n}  \sum_{0 \leq r,s < n} \lambda^{(r+s-2m)i} (x^r \ot  x^s),\\
&=& \frac{1}{n}\sum_{k=0}^{n-1}  \sum_{\substack{0 \leq r,s < n \\ r+s-2m \equiv k \!\!\!\!  \mod n}} \lambda^{ki} x^r \ot x^s,\\
&=& \frac{1}{n}\sum_{k=0}^{n-1} \widehat{\Delta}(\lambda^{ki}x^k),\\
&=&\widehat{\Delta}(c_i).
\end{eqnarray*}
\end{proof}
Note that the formulas in \Cref{lem:widehat-maps} verify $\widehat{\eta}\circ\widehat{\eta} = \widehat{m}\circ \widehat{\Delta}$ for the special diagram shown in \Cref{fig:P_3_A-squared-ex}.

Using the color basis, we can think of the vector space $\widehat{V}_\alpha$ as generated by tensor products of colors $c_i$. That is, if there are $k$ circles in a state $\Gamma_\alpha$, then we can write $\widehat{V}_\alpha$ as
$$\widehat{V}_\alpha = \langle c_{i_1}\ot c_{i_2} \ot \cdots \ot c_{i_k} \ \ | \ \ 0\leq i_1,i_2,\ldots i_k <n \rangle.$$
Call each basis element $c_{i_1}\ot c_{i_2} \ot \cdots \ot c_{i_k}$ a {\em coloring  of the state}. Call a linear combination of colorings on a state and/or between states a {\em mixture}. We will often abbreviate colorings to $c_I$ where $I\in \{0,1,\dots,{n-1}\}^k$.

The color basis  simplifies calculations in $\widehat{CH}^*_n(\Gamma;\BC)$ and gives meaning to the differentials between chain groups $\widehat{C}^*(\Gamma)$.  For example, if there is a directed edge $\zeta_{\alpha\alpha'}$ between $\Gamma_\alpha$ and $\Gamma_{\alpha'}$ in the hypercube of states and this edge corresponds to the fusion of the first two circles in $\Gamma_\alpha$ into one circle in $\Gamma_{\alpha'}$, then the corresponding map in $\widehat{\del}$ is $\widehat{\del}_{\alpha\alpha'} = \widehat{m}$. Thus, by Equation~(1) of \Cref{lem:widehat-maps}, a coloring $c_{i_1}\ot c_{i_2} \ot \cdots \ot c_{i_k}$ is in the kernel of $\widehat{m}_{12}$ if and only if $c_{i_1}$ is a different color from $c_{i_2}$. Suppose further that these two circles correspond to two disks used to construct the associated closed surface of the ribbon graph.  Then to be in the kernel of $\widehat{m}$ means that these two faces must be two different colors. Next, we will use this idea to see that a coloring of the all-zero state is in the kernel of $\widehat{\del}:\widehat{C}^0(\Gamma) \ra \widehat{C}^1(\Gamma)$ for a plane graph $\Gamma$, then that coloring corresponds to a valid $n$-face coloring of the graph.

\subsection{The meaning of the classes in the filtered $n$-color homology} \label{subsection:proof-of-main-thm2} 
In this section, we show that the dimension of $\widehat{CH}^0_n(\Gamma;\BC)$ of a plane graph $\Gamma$ is equal to the number of $n$-face colorings of $\Gamma$, thus proving the main part of Statement (3) of \Cref{MainTheorem:Colorings-of-State-Graphs}.  The proof follows from investigating  the edge maps from the all-zero state to degree $1$ states of the blowup of a plane graph.  Let $G(V,E)$ be a connected planar graph and let $\Gamma$ be a plane graph of $G$, i.e., a plane ribbon graph represented by a ribbon diagram with no immersed intersections of vertices or edges. Let $\Gamma_E^\flat$ be the blowup of $\Gamma$ with perfect matching given by the edges of $E$. Label these edges $E=\{e_1,\dots, e_\ell\}$.   

The all-zero state $\Gamma_{\vec{0}}$ of the blowup of $\Gamma$ is then a set of embedded circles in the plane. These circles correspond to the discs that are glued to $\Gamma$ to form $S^2$. Hence if $\Gamma$ has $f$ faces, then there are $f$ circles in $\Gamma_{\vec{0}}$.

Assume for a moment that $\Gamma$ is bridgeless.  For each edge $e_i\in E$, denote the  state  given by $\vec{\alpha}_i=(0,\dots,0,1,0,\dots,0)$ with the ``$1$''  in the $i$-th position by $\Gamma_{\vec{\alpha}_i}$. Then for each edge $e_i\in E$ of $\Gamma$, there is an edge in the hypercube of states from $\Gamma_{\vec{0}} \ra \Gamma_{\vec{\alpha}_i}$ in the blowup of $\Gamma$.  Each $\Gamma_{\vec{\alpha}_i}$ state can be characterized by having $(f-1)$-circles where one of the circles has an immersed double point coming from the $1$-smoothing $\XDiag$ of two circles  in $\Gamma_{\vec{0}}$ that are adjacent along the edge $e_i$ in $\Gamma$. (Since $\Gamma$ is bridgeless, there are always two faces adjacent to each edge.) Hence, for all edges $e_i \in E$, the map $\widehat{\del}_{\vec{0}\vec{\alpha}_i}:\widehat{V}_{\vec{0}} \ra \widehat{V}_{\vec{\alpha}_i}$ is a multiplication map, i.e., $\widehat{\del}_{\vec{0}\vec{\alpha}_i}=\widehat{m}$.  

\begin{theorem} Let $G(E,V)$ be a connected planar graph and $\Gamma$ be a plane graph of it. Then
$$\dim \widehat{CH}_n^0(\Gamma;\BC) = \#\{\mbox{$n$-face colorings of $\Gamma$}\}.$$
\label{prop:dim-of-CH0-equals-n-face-colorings}
\end{theorem}

\begin{remark}
Note that  \Cref{prop:dim-of-CH0-equals-n-face-colorings} proves the main part of Statement (3) of \Cref{MainTheorem:Colorings-of-State-Graphs}. We present this proof here to help build the intuition needed for the more delicate proofs of Statements (1) and (2) of \Cref{MainTheorem:Colorings-of-State-Graphs} for (nonplanar) ribbon graphs. Their proofs are the content of \Cref{section:colorings-of-state-graphs}.
\end{remark}

\begin{proof}
First, assume that $\Gamma$ is bridgeless so that the discussion above the theorem holds.  Let $c_I=c_{i_1}\ot \cdots \ot c_{i_f}$ be a coloring of $\Gamma_{\vec{0}}$ where $I=(i_1,\dots,i_f)$ is a multi-index. Then $c_I$ is in $\widehat{CH}_n^0(\Gamma;\BC)$ if it is in the kernel of each $\widehat{\del}_{\vec{0}\vec{\alpha}_i}$.  This means that for every two faces adjacent to an edge $e_i \in \Gamma$, the two faces must be colored with different colors from $\{c_0,c_1,\ldots,c_{n-1}\}$. But this is exactly the condition that $c_I$ is an $n$-face coloring of $\Gamma$. Hence, $\dim \widehat{CH}_n^0(\Gamma;\BC)$ must be at least as large as the number of $n$-face colorings of $\Gamma$.

The only thing left to show is that a linear combination of colorings for $\Gamma_{\vec{0}}$ is in the kernel of $\widehat{\del}$ if and only if each coloring in the sum is in the kernel of $\widehat{\del}$. For example, it may be that  two different colorings $c_I, c_J$ of $\Gamma_{\vec{0}}$  satisfy $\widehat{\del}(c_I) = \widehat{\del}(c_J)\not=0$, and thus $c_I-c_J$ is in the kernel of $\widehat{\del}$. Such a sum is not an $n$-face coloring of $\Gamma$. However, we show that this cannot happen.

By  Equation~(1) of \Cref{lem:widehat-maps}, for a coloring $c_I$ of $\Gamma_{\vec{0}}$, $\widehat{\del}_{\vec{0}\vec{\alpha}_i}(c_I)$ must either be $0$ or a coloring of $\Gamma_{\vec{\alpha}_i}$. Also by Equation~(1) of \Cref{lem:widehat-maps}, for a coloring $d_{I'}$ of $\Gamma_{\vec{\alpha}_i}$, there is exactly one coloring $c_I$ of $\Gamma_{\vec{0}}$ such that $\widehat{\del}_i(c_I)=d_{I'}$.  Hence, if $c=\sum_I a^I c_I \in \widehat{C}^0(\Gamma)$ where $I\in \{0,\dots,n-1\}^f$ and $a^I\in \BC$, and $\widehat{\del}(c)=0$, then for each $I$  either $\widehat{\del}(c_I)=0$ or, if $\widehat{\del}(c_I)\not=0$, then $a^I=0$. (A robust version of this idea is discussed in \Cref{lemma:colorings-to-colorings-lemma}.)

Thus, $\dim \widehat{CH}_n^0(\Gamma;\BC)$ counts the number of $n$-face colorings of a plane graph $\Gamma$, which concludes the proof when $\Gamma$ is bridgeless.

If $\Gamma$ has a bridge at an edge, say $e_k$, then the edge in the hypercube $\Gamma_{\vec{0}}\ra \Gamma_{\vec{\alpha}_k}$ corresponds to the map $\widehat{\eta}_{\vec{0}\vec{\alpha}_k}:\widehat{V}_{\vec{0}} \ra \widehat{V}_{\vec{\alpha}_k}$.  Since $\widehat{\eta}_{\vec{0}\vec{\alpha}_k}$ is injective by Equation~(3) of \Cref{lem:widehat-maps}, we have that $\ker \widehat{\del}:\widehat{C}^0(\Gamma) \ra \widehat{C}^1(\Gamma)$ is trivial, which implies $\dim \widehat{CH}_n^0(\Gamma;\BC) = 0$. This equals the total number of $n$-face colorings of $\Gamma$.
\end{proof}

In the next section we will see how to generalize this theorem to each state in the hypercube of states.

\section{\Cref{MainTheorem:Colorings-of-State-Graphs}: Filtered $n$-color homology counts $n$-face colorings of state graphs}\label{section:colorings-of-state-graphs}

In this section we prove Statements (1) and (2) of \Cref{MainTheorem:Colorings-of-State-Graphs}. (The heart of Statement (3) was proved just above.)  We will use \Cref{MainTheorem:Colorings-of-State-Graphs} to characterize the Penrose polynomial described in \Cref{Theorem:mainthmPenrose}, which is proved in \Cref{section:a-TQFT-approach-to-the-Penrose-Polynomial}.

We build up to the proof of \Cref{MainTheorem:Colorings-of-State-Graphs}  over the next few subsections. The first subsection defines a metric on the chain complex that can be used to define a Laplacian and different orthogonal subspaces. The second subsection describes how colorings on different states map to each other. The third subsection proves a ``Poincar\'{e} Lemma'' for elements of the chain group $\widehat{C}^i(\Gamma)$ that fit nicely into a sub-hypercube of the total hypercube of $\Gamma^\flat_E$.  This is the main engine for proving \Cref{MainTheorem:Colorings-of-State-Graphs}. In the fourth subsection, the space of harmonic colorings of a state, $\widehat{\mathcal{CH}}_n(\Gamma_\alpha)$, is defined and shown to be nontrivial only in even homology gradings when the graph is a plane graph. The fifth subsection uses the Poincare\'{e} Lemma to prove a restatement of Statement (1) of \Cref{MainTheorem:Colorings-of-State-Graphs}. In the sixth subsection, we show how to turn the data given in a state of the hypercube into a ribbon graph of the original graph, which is used to prove Statement (2) of \Cref{MainTheorem:Colorings-of-State-Graphs}.   And in the final subsection, we generalize all theorems in this paper to {\em any} ribbon graph, oriented or not, using {\em signed perfect matching graphs}. Assembling the theorems and propositions of the earlier subsections together proves \Cref{MainTheorem:Colorings-of-State-Graphs}. 

\subsection{A metric on the chain complex and a Hodge decomposition} In Section~4.4.2 of \cite{LeeHomo}, Lee described a metric on the chain complex of a knot and used it to define a Hodge-dual operator to her version of a filtered differential.  She only mentioned the metric and dual's existence but did not do much with them in that paper.  However, a similar adjoint operator plays an essential role in this paper.

Let $G(E,V)$ be a connected graph and $\Gamma$ be a ribbon graph of $G$. Let  $\langle c_I\rangle$  be a basis of colorings for the vector space $\widehat{V}_\alpha$. As before, ${I\in \{0,\ldots, n-1\}^{k}}$ is a multi-index that is short for $c_I=c_{i_1}\ot\cdots \ot c_{i_{k}}$  where $k$ is the number of circles in the state $\Gamma_\alpha$. We continue to work with $\mathbbm{k}=\BC$ throughout this section.

\begin{definition} 
For two colorings $c_I, c_J \in \widehat{V}_\alpha$ and two complex numbers $a,b\in\BC$, define a Hermitian metric $\langle \ , \ \rangle: \widehat{V}_\alpha \ot \widehat{V}_\alpha \ra \BC$
by $$\langle ac_I, bc_J \rangle = a\bar{b}\delta_{IJ},$$ where ${\delta_{IJ}=1}$ if $I=J$ and is zero otherwise. Extend this metric to all of $C^*(\Gamma)$.
\end{definition}

We introduce this metric to define harmonic elements of homology classes. The metric defines an adjoint operator to $\widehat{\del}$ as follows: the adjoint $\widehat{\del}^*:\widehat{C}^i(\Gamma) \ra \widehat{C}^{i-1}(\Gamma)$ is the operator that satisfies
$$\langle \widehat{\del}(c),d \rangle = \langle c , \widehat{\del}^*(d)\rangle.$$ This operator can be described by edge-differentials:

\begin{lemma}
For an edge in the hypercube of states given by $\widehat{\del}_{\alpha\alpha'}:\widehat{V}_{\alpha} \ra \widehat{V}_{\alpha'}$ with $|\alpha| = |\alpha'| -1$, the adjoint of $\widehat{\del}_{\alpha\alpha'}$, denoted $\widehat{\del}^*_{\alpha\alpha'}:\widehat{V}_{\alpha'} \ra \widehat{V}_{\alpha}$, is given by the following adjoints of $\widehat{m}$, $\widehat{\Delta}$, and $\widehat{\eta}$ on colors $c_i, c_j$:
\begin{eqnarray}\label{eqn:adjoint-of-m-delta-eta}
\widehat{m}^*(c_i) &=& c_i\ot c_i\\ \nonumber
\widehat{\Delta}^*(c_i\ot c_j) &=& n \lambda^{2mi}\delta^{ij}c_i\\ \nonumber
\widehat{\eta}^*(c_i) &=& \sqrt{n} \lambda^{mi}c_i.\nonumber
\end{eqnarray}
\end{lemma}

The adjoint operator $\widehat{\del}^*:\widehat{C}^i(\Gamma)\ra \widehat{C}^{i-1}(\Gamma)$ is then the sum of such adjoints by modifying \Cref{eqn:delta-definition} appropriately.  Since the maps of \Cref{eqn:adjoint-of-m-delta-eta} are, in essence, slight modifications of the $\widehat{m}$, $\widehat{\Delta}$, and $\widehat{\eta}$ maps, $\widehat{\del}^* \circ \widehat{\del}^* =0$. Define the Laplace operator to be $\slashed{\Delta} := (\widehat{\del}+\widehat{\del}^*)^2$. Denote the space of {\em harmonic mixtures}, i.e., $c\in \widehat{C}^i(\Gamma)$ such that $\slashed{\Delta}(c)=0$, by $\widehat{\mathcal{CH}}_n^i(\Gamma)$.  Standard arguments using the metric show that $\slashed{\Delta}(c)=0$ if and only if $\widehat{\del}(c)=0$ and $\widehat{\del}^*(c)=0$. Furthermore, again by standard arguments, the space $\widehat{C}^i(\Gamma)$ can be decomposed into subspaces using a Hodge-like Theorem for $\widehat{\del}$:

\begin{lemma}[Hodge Decomposition Theorem for $\widehat{\del}$] \label{lem:hodge-decomposition-theorem} The space $\widehat{C}^i(\Gamma)$ can be decomposed as
$$\widehat{C}^i(\Gamma) = \widehat{\mathcal{CH}}_n^i(\Gamma)\oplus \widehat{\del}\widehat{C}^{i-1}(\Gamma)\oplus \widehat{\del}^*\widehat{C}^{i+1}(\Gamma).$$
In particular, each element $c\in\widehat{C}^i(\Gamma)$ can be uniquely decomposed into the sum $c=c_h +\widehat{\del}c_- + \widehat{\del}^*c_+$.
\end{lemma}

Finally, this Hodge decomposition can be used to show $\widehat{CH}_n^i(\Gamma;\BC)\cong \widehat{\mathcal{CH}}_n^i(\Gamma)$.  Hence there is a unique harmonic element of a class in $\widehat{CH}_n^i(\Gamma;\BC)$ with respect to the metric $\langle \cdot, \cdot \rangle$.

\subsection{Mapping colorings to colorings lemma} We are now ready to generalize one of the ideas used in \Cref{subsection:proof-of-main-thm2} for the proof of \Cref{prop:dim-of-CH0-equals-n-face-colorings}. It is a key lemma for the proof of \Cref{MainTheorem:Colorings-of-State-Graphs} in the sense that it allows one to focus on a single coloring in a linear combination of like colorings (i.e., colorings on the same state) without worrying that one of the other colorings in that linear combination will map to the same nonzero coloring when applying a differential.  It, together with the uniqueness of the Hodge decomposition, is the reason why one can work with colorings in a vector space/basis sense and at the same time think of them as ``labelings of the faces by colors'' in a graph theory sense. 

\begin{lemma}[Color Basis Lemma] \label{lemma:colorings-to-colorings-lemma}Let $G(V,E)$ be a connected graph and $\Gamma$ a ribbon graph of $G$ represented by a ribbon diagram.  Let $\widehat{\del}_{\alpha\alpha'}:\widehat{V}_\alpha \ra \widehat{V}_{\alpha'}$ be the edge-differential ($\widehat{m}$, $\widehat{\Delta}$, $\widehat{\eta}$) corresponding to an edge in the hypercube of states of the blowup of $\Gamma$ from $\Gamma_\alpha$ to $\Gamma_{\alpha'}$.  Let $\langle c_I \rangle$ and $\langle c'_I \rangle$ be bases of colorings of $\widehat{V}_\alpha$ and $\widehat{V}_{\alpha'}$ respectively. Then  for all $c_I$, either
\begin{enumerate}
\item $\widehat{\del}_{\alpha\alpha'}(c_I)=0$, or if nonzero,
\item $\widehat{\del}_{\alpha\alpha'}(c_I)$ is a nonzero multiple of one and only one coloring $c'_I \in \widehat{V}_{\alpha'}$. 
\end{enumerate}
In particular, let $a=\sum_I a^I c_I \in \widehat{V}_\alpha$ such that $a^I\in \BC$ for all $I$. If $\widehat{\del}_{\alpha\alpha'}(a)=0$, then for each $I$,  either $\widehat{\del}_{\alpha\alpha'}(c_I)=0$ or, if $\widehat{\del}_{\alpha\alpha'}(c_I)\not=0$, then $a^I=0$.
 All statements hold for $\widehat{\del}^*_{\alpha\alpha'}: \widehat{V}_{\alpha} \ra \widehat{V}_{\alpha'}$ as well.
\end{lemma} 

This lemma implies that no two unique colorings $c_I, c_J\in \widehat{V}_\alpha$ map to the same coloring $c'_I$ of $\widehat{V}_\alpha'$ or vice versa. In this sense, the edge-differentials $\widehat{\del}_{\alpha\alpha'}$ and $\widehat{\del}^*_{\alpha\alpha'}$ are one-to-one on color basis elements that are not in their kernels.

\begin{proof}The proof of \Cref{lemma:colorings-to-colorings-lemma}  follows from equations for the maps given in \Cref{lem:widehat-maps} and \Cref{eqn:adjoint-of-m-delta-eta}.
\end{proof}

\begin{remark} There is no lemma equivalent to \Cref{lemma:colorings-to-colorings-lemma} for bigraded $n$-color homology and therefore there is no equivalent \Cref{MainTheorem:Colorings-of-State-Graphs} for bigraded $n$-color homology either.  In fact, it turns out that most homology classes in $CH^{i,j}(\Gamma;\mathbbm{k})$ are a linear combination of elements from different states (the harmonic elements of the states in a particular degree do not capture all nonzero homology classes in that degree). Thus, only once reaching the $E_\infty$-page of the spectral sequence does \Cref{MainTheorem:Colorings-of-State-Graphs}  fully hold. We will see that the homology classes that  live to the $E_\infty$-page are harmonic colorings on the different states themselves (see \Cref{defn:harmonic-coloring-of-a-state} below).
\end{remark}

Based upon \Cref{lemma:colorings-to-colorings-lemma}, we make the following definition:

\begin{definition} An edge differential $\widehat{\del}_{\alpha\alpha'}:\widehat{V}_\alpha \ra \widehat{V}_{\alpha'}$  is called a {\em nonzero mapping} for a coloring $c_{\alpha}$ if it satisfies Part (2) of  \Cref{lemma:colorings-to-colorings-lemma}, that is, $\widehat{\del}_{\alpha\alpha'}(c_{\alpha})$ is a nonzero multiple of some coloring $c_{\alpha'}\in \widehat{V}_{\alpha'}$.  Likewise, the edge differential  is also called a nonzero mapping for $c_{\alpha'}\in \widehat{V}_{\alpha'}$ if there exists a $c_\alpha\in \widehat{V}_\alpha$ such that $\widehat{\del}_{\alpha\alpha'}(c_{\alpha})$ is a nonzero multiple of $c_{\alpha'}$. An edge differential is called a {\em zero mapping} for coloring  $c_{\alpha}$  (or coloring $c_{\alpha'}$) if $c_{\alpha}$ is in the kernel (or if $c_{\alpha'}$ is not in the image).
\end{definition}

\subsection{A ``Poincar\'{e} Lemma'' for elements of $\widehat{C}^i(\Gamma)$ with nonzero mappings} \label{subsection:Poincare-Lemma} If a coloring has a nonzero mapping from or to it, then there exits a sub-hypercube of the total hypercube of $\Gamma^\flat_E$ based upon that coloring.  This subsection defines this sub-hypercube and explains why it consists of only nonzero mappings based upon the coloring. This result will be useful for the proof of \Cref{prop:direct-sum-equals-all-harmonics}, which is Statement (1) of \Cref{MainTheorem:Colorings-of-State-Graphs}.

\begin{lemma} Let $\Gamma^\flat_E$ be the blowup of a ribbon graph $\Gamma$ of a connected graph, $G(V,E)$. Suppose that $c_\alpha$ is a coloring of $\widehat{V}_\alpha$ that has a nonzero mapping to or from it.  Then there exists a sub-hypercube of states of the original hypercube of $\Gamma$ with two properties:
\begin{enumerate}
\item the sub-hypercube is nontrivial, i.e., it has more than one state, and
\item there is a unique coloring on each state of the sub-hypercube such that the coloring on one state gets mapped to (via Part (2) of \Cref{lemma:colorings-to-colorings-lemma}) the coloring on another state if and only if there exists an edge in the sub-hypercube between the two states.
\end{enumerate}
Furthermore, any other coloring in the sub-hypercube of states generates the same sub-hypercube. \label{lem:sub-hypercube-lemma}
\end{lemma}

Call this sub-hypercube of nonzero mappings between colorings on states the {\em color hypercube of $c_\alpha$} and denote it $\mathcal{H}(c_\alpha)$. Note that if $c_\beta$ is any other coloring in the color hypercube, then $\mathcal{H}(c_\alpha)=\mathcal{H}(c_\beta)$.

\begin{proof} Let $\Gamma^\flat_E$ be represented by a ribbon diagram and set $m=|E|$.  Then  the state $\Gamma_\alpha$ is a set of (immersed) circles in the plane and $c_\alpha$ can be represented as an ``enhanced state,'' i.e., the set of circles of $\Gamma_\alpha$ with each circle labeled with a color from $\{c_0,c_1,\dots,c_{n-1}\}$. Bifurcate and relabel if necessary the $\alpha_i$'s in the state $\alpha=(\alpha_1,\alpha_2,\dots,\alpha_m)$ into two sets, $\{\alpha_1,\alpha_2,\dots,\alpha_k\}$ and $\{\alpha_{k+1},\ldots,\alpha_m\}$, according to the following rule: $\alpha_i$ belongs to the first set ($1\leq i\leq k$) if the colors of the circle(s) of $c_\alpha$ associated to the smoothing $\alpha_i$ are the same color and $\alpha_i$ belongs to the second set ($k+1\leq i \leq m$) if the colors are different. Note that the size $k$ is positive since $c_\alpha$ has a nonzero mapping to or from it.

The sub-hypercube is given by $2^k$ states starting with the ``all-zero state'' $(0,0,\dots,0,\alpha_{k+1},\dots,\alpha_m)$, taking all possible $0$- and $1$-smoothings for the first $k$-entries, and ending with the ``all-one state'' $(1,1,\dots,1,\alpha_{k+1},\dots,\alpha_m)$. For $k+1\leq i\leq m$, keep $\alpha_i$ fixed throughout.  Since the color is the same on each circle at the smoothing site of $\alpha_i$ for $1\leq i\leq k$, there is a well-defined coloring on the state regardless if $\alpha_i=0$ or $\alpha_i=1$. By the rule, the sub-hypercube of states has a unique coloring on each state and every map to another state in the sub-hypercube is a nonzero mapping. All other edges in the complement of the sub-hypercube are zero mappings with respect to the sub-hypercube colorings.

Finally, since the colorings on the states in the sub-hypercube are defined by the rule, choosing any other coloring  on one of the other states in the sub-hypercube will generate the same sub-hypercube.
\end{proof}

\begin{proposition}[A ``Poincar\'{e} Lemma''] \label{prop:Poincare-Lemma} Let $\Gamma^\flat_E$ be the blowup of a ribbon graph $\Gamma$ of a connected graph, $G(V,E)$. Let $\mathcal{H}(c_\alpha)$ be a color hypercube for the coloring $c_\alpha$ on a state $\Gamma_\alpha$ where $|\alpha|=i$.  Let $k>0$ be the size of the color hypercube $\mathcal{H}(c_\alpha)$ and $\ell$ be the number of $1$'s in $\alpha$ from $1$ to $k$ in the tuple $\alpha\in\{0,1\}^{|E|}$. Let $\omega\in\widehat{C}^i(\Gamma)$ such that $$\omega=\sum_{j=1}^{C(k,\ell)} a_j c_{\alpha_j},$$
where the $c_{\alpha_j}$'s are the colorings on the states $\Gamma_{\alpha_j}$ in the color hypercube that have the same degree as $\alpha$ (set $\alpha_1:=\alpha$) and $a_j \in \BC$. Then there exists a $\mu\in \widehat{C}^{i-1}(\Gamma)$ and $\psi\in C^{i+1}(\Gamma)$ such that $$\omega = \widehat{\del}(\mu) + \widehat{\del}^*(\psi).$$
\end{proposition}

This proposition says that, as long as $\omega$ is constructed out of any linear combination of colorings in the color hypercube of the same degree as $\alpha$, then the form is not harmonic according to the Hodge decomposition theorem for $\widehat{\del}$.  It is like a Poincar\'{e} Lemma in that, if $\ell > (k+1)/2$ with $k>1$, then $\omega$ can be written as $\omega=\widehat{\del}(\mu)$ and if $\ell< (k-1)/2$, then $\omega=\widehat{\del}^*(\psi)$.  It is only in the ``middle'' column of the color hypercube where  both $\widehat{\del}(\mu)$ and $\widehat{\del}^*(\psi)$ may be needed.

\begin{proof} Form the following elements of $\widehat{C}^{i-1}(\Gamma)$ and $\widehat{C}^{i+1}(\Gamma)$ respectively:
$$
\mu = \sum_{j=1}^{C(k,\ell-1)} c_j c_{\gamma_j} \mbox{ \ \ \ \ \ and \ \ \ \ \ } \psi = \sum_{j=1}^{C(k,\ell+1)} b_j c_{\beta_j},
$$
where $c_j, b_j\in \BC$, $c_{\gamma_j}$'s are colorings on states $\Gamma_{\gamma_j}$ in the color hypercube with $|\gamma_j|=|\alpha|-1$, and $c_{\beta_j}$'s are the colorings on the states $\Gamma_{\beta_j}$ in the color hypercube with $|\beta_j|=|\alpha|+1$. 

Since the only nonzero mappings from $\widehat{C}^{i-1}(\Gamma)$ to $\widehat{C}^{i}(\Gamma)$ in $\mu$ correspond to edges in the color hypercube, we have that 
$$\widehat{\del}(\mu)|_{\Gamma_{\alpha_j}}$$
is a multiple of $c_{\alpha_j}$.  Note that there may be multiple edge-differentials to the state $\Gamma_{\alpha_j}$, but each of them will take their corresponding term in $\mu$ to a multiple of $c_{\alpha_j}$ and no other coloring.  Any edge-differential outside of the color hypercube will take the elements $c_j c_{\gamma_j}$ to zero, so these edge-differentials do not count in the calculation.

Similarly, the only nonzero mappings from $\widehat{C}^{i+1}(\Gamma)$ to $\widehat{C}^i(\Gamma)$ for colorings in $\psi$ correspond to edges in the color hypercube such that
$$\widehat{\del}^*(\psi)|_{\Gamma_{\alpha_j}},$$
and these will also a multiple of $c_{\alpha_j}$.  

Hence, for each state $\Gamma_{\alpha_j}$ in the color hypercube, we must solve the following linear equation:
$$\widehat{\del}(\mu)|_{\Gamma_{\alpha_j}} + \widehat{\del}^*(\psi)|_{\Gamma_{\alpha_j}} = a_j c_{\alpha_j},$$
for constants $a_j$ and variables $c_j$ and $b_j$.  There are $C(k,\ell)$ such equations.  On the other hand, there are $C(k,\ell-1)+C(k,\ell+1)$ number of variables. Since $$C(k,\ell-1)+C(k,\ell+1) \geq C(k,\ell),$$ there are {\em always} at least as many variables as equations, and therefore one can find a (not necessarily unique) solution.   
\end{proof}

\subsection{The harmonic colorings of a state} Next we define the harmonic colorings of a state, $\widehat{\mathcal{CH}}_n(\Gamma_\alpha)$, which can be thought of as the harmonic elements that exists only on the state $\Gamma_\alpha$. This subspace of $\widehat{\mathcal{CH}}_n^i(\Gamma)$ is the  harmonic elements of the vector space $\widehat{V}_\alpha$ that do not depend on elements of other state vector spaces in $\widehat{C}^i(\Gamma)=\oplus_{|\alpha| = i}\widehat{V}_\alpha$ to form a harmonic mixture in $\widehat{\mathcal{CH}}_n^i(\Gamma)$. 

Let $\Gamma_\alpha$ be a state of the hypercube for the blowup of a ribbon graph $\Gamma$.  Consider all states $\Gamma_{\alpha'}$ such that $|\alpha'|=|\alpha|+1$ where there is an edge in the hypercube between $\Gamma_\alpha$ and $\Gamma_{\alpha'}$.  Denote the union of these states by $\Gamma_\alpha^+ = \cup \Gamma_{\alpha'}$.  Then $\widehat{C}^{i+1}(\Gamma^+_\alpha) \subset \widehat{C}^{i+1}(\Gamma)$ is made up of the direct sum of vector spaces $\oplus V_{\alpha'}$. The restriction of the metric $\langle \ , \ \rangle$ to this subspace remains a metric.

Similarly, define $\widehat{C}^{i-1}(\Gamma_\alpha^-) \subset \widehat{C}^{i-1}(\Gamma)$ consisting of all vector spaces $\widehat{V}_{\gamma}$ such that there is an edge from $\Gamma_\gamma$ to $\Gamma_\alpha$ in the hypercube of states. 

Define $\widehat{\del}_\alpha:\widehat{V}_\alpha \ra \widehat{C}^{i+1}(\Gamma^+_\alpha)$ by taking the sum of all  differentials from  $\widehat{V}_\alpha$ to the $(|\alpha|+1)$-states. Similarly, define  $\widehat{\del}_\alpha^*: \widehat{V}_\alpha \ra \widehat{C}^{i-1}(\Gamma^-_\alpha)$ to be the sum of all nontrivial adjoint maps from $\widehat{V}_\alpha$ to $(|\alpha|-1)$-states.

\begin{definition}\label{defn:harmonic-coloring-of-a-state}
The {\em harmonic colorings of a state $\Gamma_\alpha$}, denoted $\widehat{\mathcal{CH}}_n(\Gamma_\alpha)$, is the set of elements of $\widehat{V}_\alpha$ that is in the kernel of $\widehat{\del}_\alpha$ and the kernel of $\widehat{\del}^*_\alpha$.
\end{definition}

\Cref{prop:dim-of-CH0-equals-n-face-colorings} shows that $\widehat{\mathcal{CH}}_n(\Gamma_\alpha)$ can be nonzero: for any connected ribbon graph $\Gamma$, $\widehat{\mathcal{CH}}_n(\Gamma_{\vec{0}})\cong\widehat{CH}^0_n(\Gamma)$. Thus, by \Cref{prop:dim-of-CH0-equals-n-face-colorings} and the four color theorem,  $\widehat{\mathcal{CH}}_n(\Gamma_{\vec{0}})$ is nonzero when $n\geq 4$ for all plane graphs $\Gamma$ that do not contain a bridge. This raises the natural question of when other states might have nonzero harmonic colorings.

\begin{proposition} \label{proposition:wide-hat-m-and-wide-hat-delta-maps} Let $\Gamma$ be a ribbon graph of $G(V,E)$. If $\widehat{\mathcal{CH}}_n(\Gamma_\alpha) \not=0$, then
\begin{enumerate}
\item all maps $\widehat{\del}_{\alpha\alpha'}:\widehat{V}_\alpha \ra \widehat{V}_{\alpha'}$ where $|\alpha'|=|\alpha|+1$ are $\widehat{m}$ maps, and
\item all maps $\widehat{\del}_{\alpha'\alpha}:\widehat{V}_{\alpha'} \ra \widehat{V}_{\alpha}$ where $|\alpha'|=|\alpha|-1$ are $\widehat{\Delta}$ maps.
\end{enumerate}
\end{proposition}

Thus, when $\widehat{\mathcal{CH}}_n(\Gamma_\alpha) \not=0$, all ``incoming'' edges into $\Gamma_\alpha$ in the hypercube correspond to $\widehat{\Delta}$ maps and all ``outgoing'' edges out of $\Gamma_\alpha$ in the hypercube correspond to $\widehat{m}$ maps. It is possible that a state may have this property about incoming edges and outgoing edges and still have no harmonic colorings. For example, many plane graphs of many connected planar graphs do not have $3$-face colorings. Likewise, only Eulerian plane graphs $\Gamma$ have $\widehat{\mathcal{CH}}_2(\Gamma_{\vec{0}})\not=0$ out of all plane graphs.

\begin{proof} Let $\Gamma$ be represented by a ribbon diagram and consider the hypercube of the blowup of it. If $a\in\widehat{\mathcal{CH}}_n(\Gamma_\alpha)$, $a\not=0$, then $\widehat{\del}_\alpha(a)=0$, which means $\widehat{\del}_{\alpha\alpha'}(a)=0$ for all edge-differentials $\widehat{\del}_{\alpha\alpha'}:\widehat{V}_\alpha \ra \widehat{V}_{\alpha'}$ that make up $\widehat{\del}_\alpha$. The equations for $\widehat{m}$, $\widehat{\Delta}$, and $\widehat{\eta}$ in \Cref{lem:widehat-maps} together with \Cref{lemma:colorings-to-colorings-lemma} show that only $\widehat{m}$ maps can have a kernel. Hence, $\widehat{\del}_{\alpha\alpha'} = \widehat{m}$ for each edge from $\Gamma_\alpha$ to $\Gamma_\alpha'$ where $|\alpha'|=|\alpha|+1$.  
Similarly, if $a\in \ker \widehat{\del}^*_\alpha$, $a\not=0$, then by the equations in \Cref{eqn:adjoint-of-m-delta-eta} and  \Cref{lemma:colorings-to-colorings-lemma}, the only maps between states that can have kernel are $\widehat{\Delta}^*$ maps. Reversing these maps shows that incoming maps must be $\widehat{\Delta}$ maps.
\end{proof}

We now have enough machinery to show that for plane graphs, $n$-face colorings of a state only exist on states in even degrees. This proves the second part of Statement (1) of \Cref{MainTheorem:Colorings-of-State-Graphs}.  It is also the reason that taking the Euler characteristic of the filtered $n$-color homology reduces to a sum of nonnegative terms, which is needed for several other theorems described in this paper when restricting to planar graphs, including main theorems like \Cref{Theorem:mainthmPenrose}. 

\begin{proposition} Let $\Gamma$ be a plane graph of a connected graph $G(E,V)$. If $\widehat{\mathcal{CH}}_n(\Gamma_\alpha) \not=0$, then $|\alpha|$ is even. \label{proposition:even-degree-non-zero-n-face-colorings}
\end{proposition}

\begin{proof} Without loss of generality, assume $\Gamma$ is represented by a ribbon diagram with no crossings, i.e., the graph is embedded in $\BR^2$ with no immersed intersections. Then the all-zero state $\Gamma_{\vec{0}}$ of the blowup of $\Gamma$ is a set of disjoint circles in the plane.  All other states $\Gamma_\alpha$ will have $|\alpha|$ $0$-smoothings $\IIDiag$ in $\Gamma_{\vec{0}}$ replaced with $1$-smoothings $\XDiag$. If $|\alpha|$ is odd, then by the Jordan Curve Theorem, $\Gamma_\alpha$ must have at least one immersed circle with an odd number of immersed intersections, i.e., it has at least one circle with at least one self-intersection.  This self-intersection corresponds to a $1$-smoothing along some edge in $\Gamma$. Undoing this $1$-smoothing  gives a state $\Gamma_\gamma$ with $|\gamma| = |\alpha|-1$.  The only way to introduce a self-intersection in a single circle when going from $\Gamma_\gamma$ to $\Gamma_\alpha$ is if the map $\widehat{\del}_{\gamma\alpha}:\widehat{V}_\gamma\ra\widehat{V}_\alpha$ is an $\widehat{m}$ map or an $\widehat{\eta}$ map. (A $\widehat{\Delta}$ map would correspond to the creation of two circles, which violates the fact that we ended with one circle with a self-intersection.) Both of these possibilities contradict \Cref{proposition:wide-hat-m-and-wide-hat-delta-maps}. Hence, if $\widehat{\mathcal{CH}}_n(\Gamma_\alpha) \not=0$, then $|\alpha|$ must be even.
\end{proof}

Note that the proposition {\em does not hold} for non-planar graphs. For example, for the ribbon graph of $K_{3,3}$ in \Cref{Fig:ribbondiagram}, the $\dim \widehat{CH}^i_4(K_{3,3};\BC)$ is for each $i\geq 0$ equal to $$0,0,24,48,24,24,48,24,0,0.$$
The $48$ $4$-face colorings for $i=3$ and $i=6$ are supported on two states with $24$ $4$-face colorings each in each of those degrees (cf. \Cref{section:table-of-examples}). The remaining nontrivial homological dimensions have $24$ $4$-face colorings supported on a single state in each degree. Furthermore, the Euler characteristic of $\widehat{CH}^i_4(K_{3,3};\BC)$ is zero, which equals what was computed in \Cref{example:K33-computation-of-Penrose-poly} for $P(\Gamma,4)$.

\subsection{The proof of Statement (1) of \Cref{MainTheorem:Colorings-of-State-Graphs}} \Cref{prop:direct-sum-equals-all-harmonics} below is one of the key propositions of this paper for understanding how the filtered $n$-color homology ``counts'' $n$-face colorings of ribbon graphs of a graph.  In this subsection, this count is about colorings of the circles in a state but later in \Cref{subsec:face-colorings-of-ribbon-graphs} we show how these immersed circles can also be thought of as ribbon graphs themselves, called state graphs.   \Cref{prop:direct-sum-equals-all-harmonics} is also the main step in proving Statement (1) of \Cref{MainTheorem:Colorings-of-State-Graphs}.

\begin{proposition} Let $\Gamma$ be a ribbon graph of a connected graph $G(V,E)$.  Then $$\widehat{\mathcal{CH}}^i_n(\Gamma) = \bigoplus_{|\alpha|=i} \widehat{\mathcal{CH}}_n(\Gamma_\alpha).$$ 
\label{prop:direct-sum-equals-all-harmonics}
\end{proposition}

The rest of this subsection is devoted to proving this proposition.\\

Let $\Gamma$ be represented by a ribbon diagram and consider the hypercube of states of its blowup $\Gamma_E^\flat$.  First, note that $\widehat{\del}_\alpha:\widehat{V}_\alpha \ra \widehat{C}^{i+1}(\Gamma^+_\alpha)$ is the subset of maps $\widehat{\del}_{\alpha\alpha'}:\widehat{V}_\alpha \ra \widehat{V}_{\alpha'}$ that are used to define $\widehat{\del}$. The same is true for $\widehat{\del}^*_\alpha$ in reference to $\widehat{\del}^*$. Hence, $\ker (\widehat{\del}_\alpha+\widehat{\del}^*_\alpha) \subset \ker (\widehat{\del}+\widehat{\del}^*)$ and 
\begin{equation}\label{eqn:subset-of-harmonic-mixtures}
\bigoplus_{|\alpha|=i} \widehat{\mathcal{CH}}_n(\Gamma_\alpha) \subset \widehat{\mathcal{CH}}_n^i(\Gamma).
\end{equation}
We need to show \Cref{eqn:subset-of-harmonic-mixtures} is an equality. Use the metric to orthogonally decompose harmonic filtered $n$-color homology into the direct sum: 
$$
 \widehat{\mathcal{CH}}_n^i(\Gamma) = \left(\bigoplus_{|\alpha|=i} \widehat{\mathcal{CH}}_n(\Gamma_\alpha)\right) \oplus W,
$$
where $W$ are the elements $\omega \in \widehat{\mathcal{CH}}^i(\Gamma)$ such that there is a nonzero mapping to or from $\omega$ for some edge-differential.  (The first summand is the set of all $\omega$ that have no nonzero mappings, i.e., every edge differential or adjoint edge-differential maps $\omega$ in the first summand to zero.)

We will show that there is a subset $W' \subset W$ with the following two properties:  
\begin{enumerate}
\item For every $\omega\in W$ with $\omega\not=0$, there exists a $\omega' \in W'$ such that $\omega'\not=0$, $\omega'$ can be written as a linear combination of elements from a color hypercube $\mathcal{H}(c_\alpha)$ for some $c_\alpha$, and $\omega=\omega'+\nu$ for some $\nu \in \widehat{\mathcal{CH}}_n^i(\Gamma)$.
\item If $\omega' \in W'$, $\omega'\not=0$, then there exists a $\mu\in \widehat{C}^{i-1}(\Gamma)$ and $\psi\in \widehat{C}^{i+1}(\Gamma)$ such that $$\omega' = \widehat{\del}(\mu) + \widehat{\del}^*(\psi).$$ 
\end{enumerate}

These properties can be used to immediately prove that $W=\{0\}$. By (2),  $\omega'= \widehat{\del}(\mu) + \widehat{\del}^*(\psi)$ for all $\omega' \in W'$, therefore $$W'\subset \widehat{\del}\widehat{C}^{i-1}(\Gamma)\oplus \widehat{\del}^*\widehat{C}^{i+1}(\Gamma).$$
By the Hodge decomposition theorem, since $W' \subset \widehat{\mathcal{CH}}^i(\Gamma)$, $W' = \{0\}$. By (1),  every $\omega\in W$ with $\omega\not=0$ can be written as a sum with a term $\omega'\in W'$ in it with $\omega'\not=0$. Thus, the only element that can exist in $W$ is the zero element.

Property (1) follows from the Color Basis Lemma (\Cref{lemma:colorings-to-colorings-lemma}) and \Cref{lem:sub-hypercube-lemma}.  
Property (2) follows from the ``Poincar\'{e} Lemma'' (\Cref{prop:Poincare-Lemma}) once it is shown how to construct $\omega'$ from $\omega$. 

We prove Property (1). Since $\omega\in W$, there exists some nonzero edge-differential such that $$\widehat{\del}_{\alpha_1\beta_1}(\omega|_{\Gamma_{\alpha_1}}) \not=0,$$ for some coloring $c_{\alpha_1}$ on some state $\Gamma_{\alpha_1}$ where $\omega|_{\Gamma_{\alpha_1}} = a_1 c_{\alpha_1} + (\mbox{other terms})$ and $a_1\not=0$. 

\begin{remark} Suppose the nonzero mapping is {\em to} the coloring $c_{\alpha_1}$ on $\Gamma_{\alpha_1}$ {\em instead}.  If $\widehat{\del}_{\gamma_1\alpha_1}(c_{\gamma_1})$ is a nonzero multiple of $c_{\alpha_1}$, then $\widehat{\del}^*_{\gamma_1\alpha_1}(c_{\alpha_1})$ is a nonzero multiple of $c_{\gamma_1}$.   Since $\widehat{\del}^*(\omega)=0$, there must be another coloring $c_{\alpha_2}$ such that $\widehat{\del}(c_{\alpha_2})|_{\Gamma_{\gamma_1}}$ is also a nonzero multiple of $c_{\gamma_1}$. This nonzero mapping must exist to cancel out (possibly partially) the term coming from $\widehat{\del}^*_{\gamma_1\alpha_1}(c_{\alpha_1})$.  Using the ideas in the proof of \Cref{lem:sub-hypercube-lemma}, there must exists a $c_{\beta_1}$ on a state $\Gamma_{\beta_1}$ such that $$\widehat{\del}^*_{\alpha_1\beta_1}(c_{\beta_1}) \mbox{ \ \ \ \ and \ \ \ \ }  \widehat{\del}^*_{\alpha_2\beta_1}(c_{\beta_1})$$
are both nonzero and multiples of $c_{\alpha_1}$ and $c_{\alpha_2}$ respectively.  Reversing the direction of the first map, this means that $\widehat{\del}_{\alpha_1\beta_1}$ is a nonzero mapping from $c_{\alpha_1}$ to $c_{\beta_1}$. Hence, we can assume without loss of generality that the nonzero mapping is always a ``from'' map.\label{remark:wlog-nonzero-mappings}
\end{remark}

Since there is a nonzero mapping for $c_{\alpha_1}$, there exists a nontrivial color hypercube $\mathcal{H}(c_{\alpha_1})$ by  \Cref{lem:sub-hypercube-lemma}.  Following \Cref{prop:Poincare-Lemma}, let $k$ be the size of the color hypercube and $\ell$ be the number of $1$'s in $\alpha_1$ from $1$ to $k$ in the tuple $\alpha_1 \in \{0,1\}^{|E|}$. (Here assume that the first $k$ entries make up the hypercube as constructed in the proof of \Cref{lem:sub-hypercube-lemma}.)  By \Cref{remark:wlog-nonzero-mappings},  $k\geq 2$ for $c_{\alpha_1}$.

Use $\mathcal{H}(c_{\alpha_1})$ and $\omega$ to define $\omega'$.  Define $\omega' \in\widehat{C}^i(\Gamma)$ such that
$$\omega' = \sum_{j=1}^{C(k,\ell)} a_j c_{\alpha_j},$$
where $c_{\alpha_j}$'s are the colorings in $\mathcal{H}(c_{\alpha_1})$ with the same degree as $\alpha_1$ and where $a_j$ is the coefficient of the $c_{\alpha_j}$-term in $\omega$.  The color hypercube $\mathcal{H}(c_{\alpha_1})$ captures every nonzero mapping to and from each $c_{\alpha_j}$ in $\omega'$. By \Cref{lemma:colorings-to-colorings-lemma}, none of the other colorings in $\omega$ map to the same terms that the terms in $\omega'$ map to under the differential or its adjoint, or vice versa.  Thus, we must have $$\widehat{\del}(\omega')=0 \mbox{\ \ \ \ and  \ \ \ \ } \widehat{\del}^*(\omega') = 0.$$
Furthermore, $\omega'\in W$ because $c_{\alpha_1}$ has a nonzero mapping from it and $a_1\not=0$. Therefore,  $\omega-\omega'\in \widehat{\mathcal{CH}}^i_n(\Gamma)$, proving Property (1).

By construction, $\omega'$ satisfies the hypotheses of \Cref{prop:Poincare-Lemma}. Therefore $$\omega' \in\widehat{\del}\widehat{C}^{i-1}(\Gamma)\oplus \widehat{\del}^*\widehat{C}^{i+1}(\Gamma),$$ proving Property (2).

\subsection{Face colorings of ribbon graphs for all ribbon graphs of a graph} \label{subsec:face-colorings-of-ribbon-graphs} We have already seen that, for a plane graph $\Gamma$,  the all-zero state $\Gamma_{\vec{0}}$ of the blowup of $\Gamma$ is a set of embedded circles in the plane that correspond to the faces of the plane graph $\Gamma$ (cf. \Cref{subsection:proof-of-main-thm2}). In this section, we generalize this idea to all states in the hypercube for any connected ribbon graph, whether it is planar or not. Such ribbon graphs will be called state graphs. We also interpret the dimension of the harmonic colorings of a state as $n$-face colorings of the state graph (Statement (2) of \Cref{MainTheorem:Colorings-of-State-Graphs}). Finally, we establish conditions under which the state graphs are in one-to-one correspondence with all ribbon graphs up to equivalence.

Let $G(V,E)$ be a connected graph and let $\Gamma$ be a ribbon diagram of it. Let $\Gamma^\flat_E$ be the blowup of $\Gamma$. In this subsection, we define the map $$\mathcal{R}:\{\mbox{states of the hypercube of $\Gamma^\flat_E$}\} \ra \{\mbox{ribbon graphs of }G(V,E)\}.$$
To describe this map, first consider $\Gamma_{\vec{0}}$, the all-zero state of $\Gamma^\flat_E$.  As described before, this state is a set of (immersed) circles where each circle corresponds to a face of the ribbon graph.  We can build this ribbon graph explicitly from the ribbon diagram as follows: replace each vertex of $\Gamma$ with a disk and each edge with a flat ribbon, attached in the order around the disk given by the ribbon diagram (as in \Cref{Def:ribbondiagram}) to get the ribbon surface, which we will also call $\Gamma_{\vec{0}}$. The boundary of this surface, projected to the plane, will then be in one-to-one correspondence with the circles of the state $\Gamma_{\vec{0}}$, as in \Cref{fig:state-surface}.

\begin{figure}[h]
\includegraphics[scale=.9]{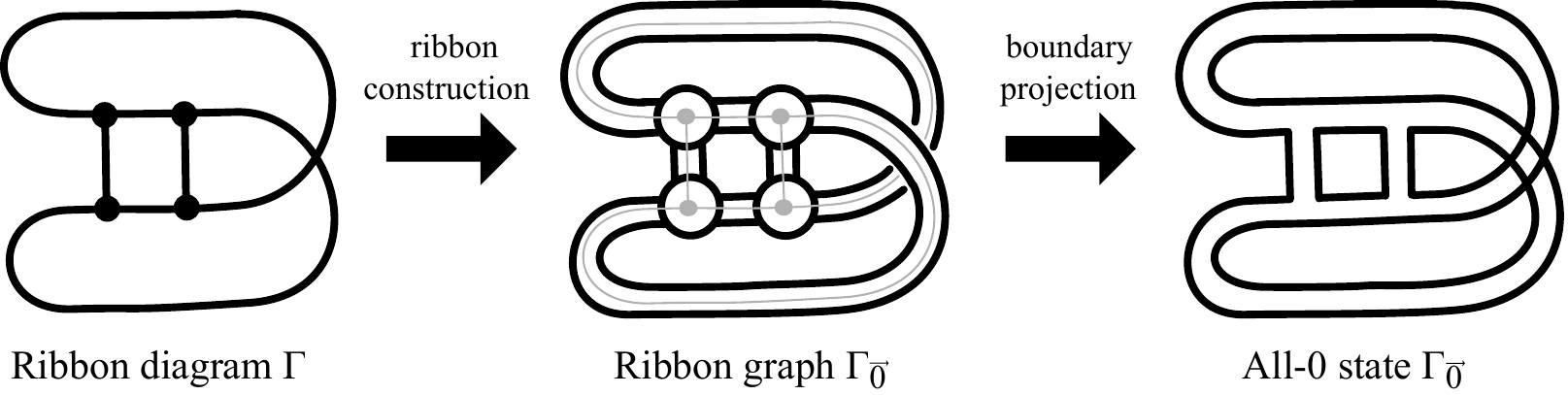}
\caption{Building the all-zero state surface $\Gamma_{\vec{0}}$ from the ribbon diagram $\Gamma$.}
\label{fig:state-surface}
\end{figure}

Define $\mathcal{R}(\Gamma_{\vec{0}})$ to be the ribbon graph, also named $\Gamma_{\vec{0}}$, where it is thought of as the embedding $i:G\ra \Gamma_{\vec{0}}$ (it is the reverse of the second arrow in \Cref{fig:state-surface}).  If $\Gamma_{\alpha}$ is another state in the hypercube with $\alpha=(\alpha_1,\dots,\alpha_{|E|})\in \{0,1\}^{|E|}$, then define the ribbon graph, which we will also continue to call $\Gamma_\alpha$, found by inserting half-twists into the ribbons of the surface $\Gamma_{\vec{0}}$ that correspond to the edges where $\alpha_i = 1$. Thus, $\mathcal{R}(\Gamma_\alpha)$ is the ribbon graph  $i:G\ra \Gamma_\alpha$. We can think of the symbol $\Gamma_\alpha$ as either a state or a ribbon graph, and  context will determine which one is intended.  For example, the boundary projection of $\Gamma_\alpha$ to the plane is the state $\Gamma_\alpha$. Unlike the surfaces associated to ribbon diagrams, the surface $\Gamma_\alpha$ is often not orientable. It cannot be represented by ribbon diagram, which is always orientable, but by a signed ribbon diagram, see \Cref{def:signedribbongraph} below.

In the literature, the ribbon graphs $\Gamma_\alpha$ are examples of what are called {\em partial Petrials} (cf. \cite{EMM}). We call them state graphs to emphasize the relationship with the states they correspond to:

\begin{definition} Let $\Gamma$ be a ribbon graph of $G(V,E)$ represented by a ribbon diagram, and let $\Gamma^\flat_E$ be the blowup  of it.  A {\em state graph $\Gamma_\alpha$ of  $\Gamma$} is the ribbon graph $i:G\ra \Gamma_\alpha$ for the state $\Gamma_\alpha$ of the hypercube of $\Gamma^\flat_E$. \label{defn:state-graph}
\end{definition}

Since states are synonymous with state graphs, one should expect the dimension of the harmonic colorings of a state,   $\dim \widehat{\mathcal{CH}}_n(\Gamma_\alpha)$, to count $n$-face colorings of the state graph. This is the content of \Cref{prop:dim-of-CHhat-state-equals-n-face-colorings} below. In fact, it is Statement (2) of \Cref{MainTheorem:Colorings-of-State-Graphs}. This is a  general version of \Cref{prop:dim-of-CH0-equals-n-face-colorings}, which was only about planar graphs. Recall that the closed associated surface $\overline{\Gamma}_\alpha$ to $\Gamma_\alpha$ is constructed by gluing disks into the boundary of $\Gamma_\alpha$ (cf. \Cref{Def:ribbongraph}).

\begin{proposition}  Let $G(E,V)$ be a connected graph and $\Gamma$ be a ribbon diagram of it. For all states $\Gamma_\alpha$ of the hypercube of states of $\Gamma^\flat_E$, the number of $n$-face colorings of the closed associated surface $\overline{\Gamma}_\alpha$ to the state graph $\Gamma_\alpha$ is equal to the dimension of the harmonic colorings of the state $\Gamma_\alpha$, i.e., 
$$\dim \widehat{\mathcal{CH}}_n(\Gamma_\alpha) = \#\{\mbox{$n$-face colorings of $\overline{\Gamma}_\alpha$}\}.$$
In particular,  $\dim \widehat{CH}^0_n(\Gamma;\BC) =\#\{\mbox{$n$-face colorings of $\overline{\Gamma}$}\}$.
\label{prop:dim-of-CHhat-state-equals-n-face-colorings}
\end{proposition}

\begin{proof} This theorem follows from the proof of \Cref{prop:dim-of-CH0-equals-n-face-colorings} using  \Cref{defn:harmonic-coloring-of-a-state} together with \Cref{proposition:wide-hat-m-and-wide-hat-delta-maps} and \Cref{lemma:colorings-to-colorings-lemma}. If $\widehat{\mathcal{CH}}_n(\Gamma_\alpha)\not=0$, then the outgoing maps and the adjoint of the incoming maps to $\widehat{V}_\alpha$ are ``multiplication'' maps via \Cref{proposition:wide-hat-m-and-wide-hat-delta-maps} and \Cref{eqn:adjoint-of-m-delta-eta}.  Thus, a coloring of $\Gamma_\alpha$ that is in the kernel of $\widehat{\del}_\alpha$ and $\widehat{\del}_\alpha^*$ must represent an $n$-face coloring of $\overline{\Gamma}_\alpha$: every two faces adjacent to an edge must be colored with two different colors in order to be in the kernel of a multiplication map. \Cref{proposition:wide-hat-m-and-wide-hat-delta-maps} then shows that the only harmonic colorings of \Cref{defn:harmonic-coloring-of-a-state} are linear combinations of colorings in the kernel of these maps following the same argument as in the proof of \Cref{prop:dim-of-CH0-equals-n-face-colorings}.
\end{proof}

\begin{remark}
When a state graph is oriented, the proof above can be used to show that $\widehat{CH}^0_n(\Gamma_\alpha;\BC) \cong  \widehat{\mathcal{CH}}_n(\Gamma_\alpha)$, where $\Gamma_\alpha$ on the left hand side of the equation is a ribbon graph and $\Gamma_\alpha$ on the right hand side of the equation is a state. Furthermore, the orientation condition can be dropped by using signed ribbon diagrams discussed later in this section (cf. \Cref{theorem:all-ribbon-graphs-theorem}). Therefore, we can define $\widehat{\mathcal{CH}}_n(\Gamma):=\widehat{CH}^0_n(\Gamma;\BC)$ and call this group the {\em harmonic colorings of a ribbon graph}. Thus, the space of harmonic colorings of a ribbon graph can be thought of by itself or as part of a hypercube of states.
\end{remark}

Two state graphs of different states in the hypercube of states of $\Gamma_E^\flat$ can be equivalent to each other as ribbon graphs. Hence, $\mathcal{R}$ is usually not injective. (The simplest example is rotating the state graphs of the theta graph by $180^\circ$.) Furthermore, for a ribbon graph $\Gamma$ of a generic graph $G$, the image of $\mathcal{R}$ for it is rarely the entire set of ribbon graphs of $G$. However, for the most important class of graphs, i.e., trivalent ones, more can be said:

\begin{theorem}
Let $\Gamma$ be a ribbon graph of a connected graph $G(V,E)$ represented by a ribbon diagram.  
\begin{enumerate}
\item If $G$ is trivalent, then $\mathcal{R}$ is surjective.  
\item If the graph automorphism group of $G$ is trivial, $G$ is loopless,  and the valency of each vertex is greater than two, then $\mathcal{R}$ is injective.
\end{enumerate}\label{thm:map-R-theorem}
\end{theorem}

\begin{proof}
For Statement (1), first we show that $\mathcal{R}$ is onto the set of oriented ribbon graphs of $G(V,E)$.  Recall that oriented ribbon graphs are represented by ribbon diagrams where the edges are attached in a specific cyclic order around each vertex. In the case of trivalent graphs, there are only two possible cyclic orders for attaching the three edges to each vertex of a ribbon diagram, as in shown in \Cref{fig:R-is-surjective-1}. 

\begin{figure}[h]
\includegraphics[scale=.17]{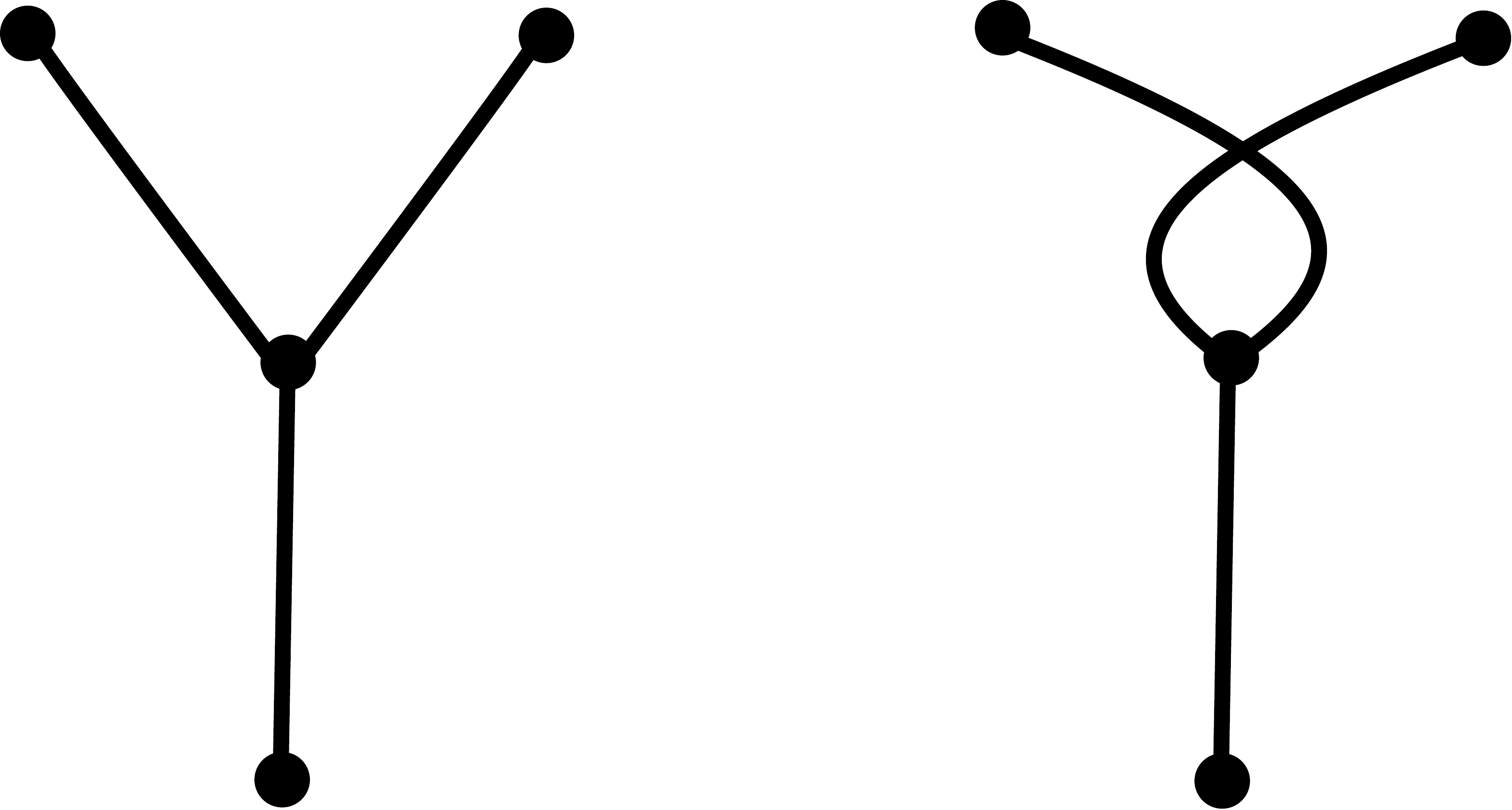}
\caption{The two cyclic orders of edges at a vertex for a ribbon diagram.}
\label{fig:R-is-surjective-1}
\end{figure}

Thus, after fixing an initial ribbon diagram, there are $2^{|V|}$ number of possibly different ribbon diagrams given by choosing one of the two cyclic orders at each vertex---either by keeping it the same as the initial ribbon diagram or switching to the other cyclic order. Let $\Gamma$ be the initial ribbon diagram. The state graph $\Gamma_{\vec{0}}$ of its blowup is the surface for this ribbon diagram, i.e., $i:G\ra \Gamma_{\vec{0}}$.   Let $\Gamma'$ be represented by one of the other $2^{|V|}$  ribbon diagrams.  This ribbon diagram can be realized as state graph of the hypercube of states for $\Gamma^\flat_E$ as follows: For each vertex of $\Gamma'$, if the edges of the ribbon diagram $\Gamma'$ are attached in the same cyclic order as $\Gamma$ as in the left hand picture in \Cref{fig:R-is-surjective-1}, then do nothing. If the edges at the vertex are attached using the other cyclic order when compared to $\Gamma$, then insert a half-twist into each of the three attaching ribbons of $\Gamma_{\vec{0}}$ at that vertex. The three half-twisted ribbon graph  is then equivalent to the picture on the right hand side of \Cref{fig:R-is-surjective-1}, see \Cref{fig:R-is-surjective-2}.

\begin{figure}[h]
\includegraphics[scale=.2]{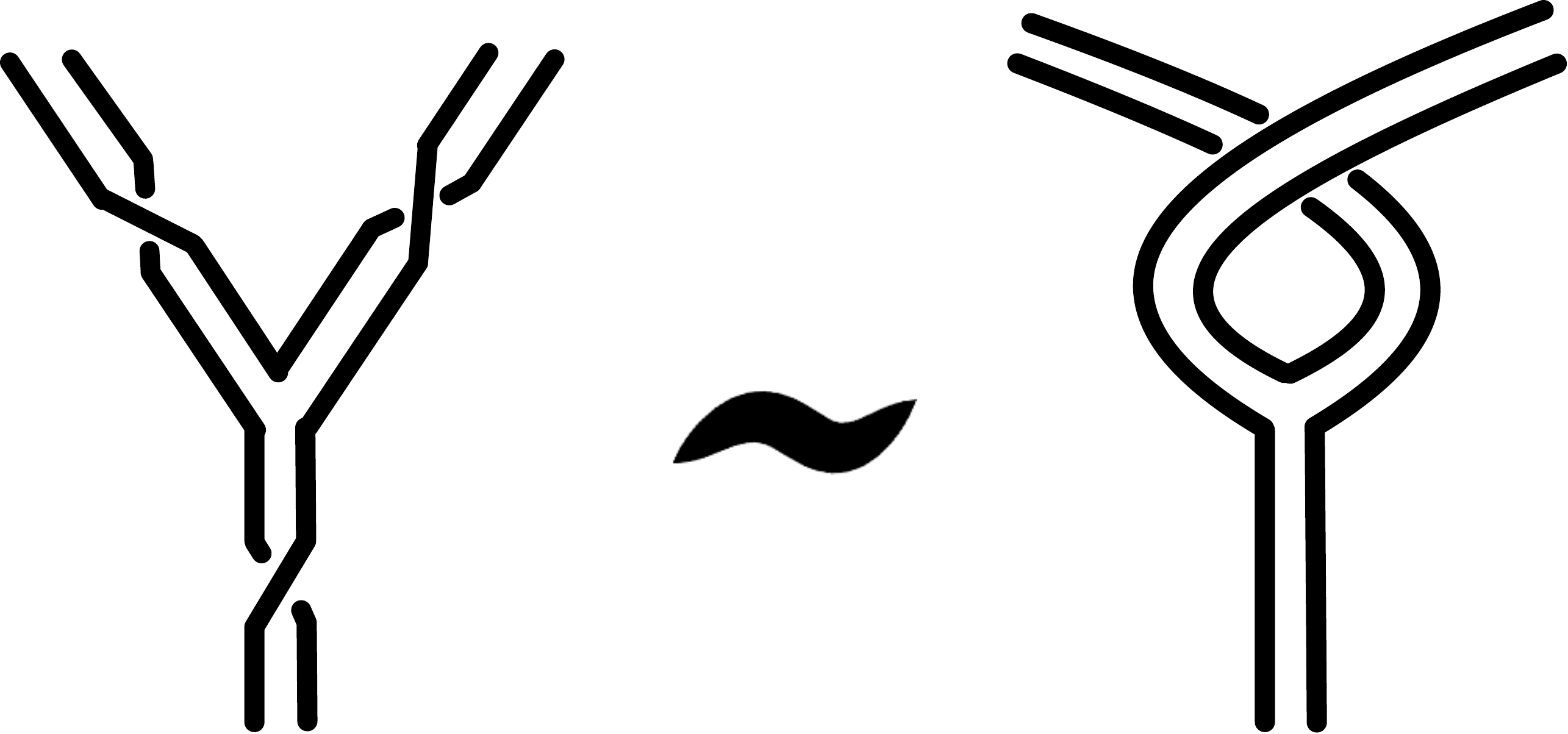}
\caption{The three half-twisted ribbon graph and the second cyclic order of \Cref{fig:R-is-surjective-1} are locally equivalent.}
\label{fig:R-is-surjective-2}
\end{figure}

After performing this operation for each vertex of $\Gamma_{\vec{0}}$, some of the ribbons will have two half-twists in them. When this happens, replace the two half-twists with a flat ribbon. The result is a state graph $\Gamma_{\alpha}$ for some $\alpha\in\{0,1\}^{|E|}$, which is equivalent to the original ribbon diagram $\Gamma'$.  

Next, note that each non-orientable ribbon graph of $G$, up to equivalence, can be found by putting half-twists into some of the ribbons of the surface associated to one of the ribbon diagrams.  Doing this operation may leave some ribbons of the non-orientable ribbon graph with two half-twists in its associated surface.  Like above, if this happens, replace the two half-twists with a flat ribbon. The result is again a state graph $\Gamma_{\alpha}$ for some $\alpha \in\{0,1\}^{|E|}$ that is equivalent to the non-orientable ribbon graph. Thus, $\mathcal{R}$ is onto.

For Statement (2), suppose $\varphi:\Gamma_{\alpha} \ra \Gamma_{\alpha'}$ is an equivalency of state graphs where $\alpha \not=\alpha'$ (cf. \Cref{Def:ribbongraph}).  Because the automorphism group $\mbox{Aut}(G)$ is trivial, we can assume that $\varphi$ is the identity on vertices and edges. (The underlying graphs of $\Gamma_\alpha$ and $\Gamma_{\alpha'}$ are identical.)  By construction of the state graphs, the only type of nontrivial ribbon isomorphism from $\Gamma_{\alpha}$ to $\Gamma_{\alpha'}$ is one that puts extra twists into selected bands. Hence the map $\varphi$  either (a) does nothing to a band if $\alpha_i=\alpha_i'$, (b) adds a twist if $\alpha_i=0$ and $\alpha_i'=1$, or (c) undoes a twist if $\alpha_i=1$ and $\alpha_i'=0$.
 
Suppose Case (b), that is, $\varphi$ adds a twist at edge $e_i$ (Case (c) is similar). A portion of the boundary of the face (or faces) adjacent  to $e_i$ are made up of two paths of edges: say $\dots a e_i b \dots$ on one side of $e_i$ and  $\dots c e_i d\dots$ on the other. Without loss of generality, keep edges $a$ and $c$ fixed before and after applying $\varphi$. Since $\varphi$ ``exchanges'' these paths at $e_i$, we get $$\varphi(\dots a e_i  b \dots) =\dots a e_i d \dots \mbox{ \ \  and \ \ } \varphi(\dots c e_i d \dots ) = \dots c e_i b \dots.$$
In order for $\varphi$ to be an equivalence, $\varphi$ must preserve each face path (see Lemma 4.1.2 of \cite{MT} for example). This can only happen if either
\begin{enumerate}
\item the vertex between $e_i$ and $b$ has valency greater than 2 and $b=d$, i.e., $b$ is a loop, or
\item the vertex between $e_i$ and $b$ has valency 2 and $b=d$.
\end{enumerate}
Both cases are ruled out by the hypotheses. Hence no such function $\varphi$ exists and $\mathcal{R}$ is injective.
\end{proof}

A graph without loops or multiple edges between two vertices is called {\em simple}.  Since a graph with multiple edges between two vertices  has a nontrivial automorphism, we can restate the second statement of  \Cref{thm:map-R-theorem} to say that $\mathcal{R}$ is injective for valence greater than two simple graphs with trivial automorphism group. It is likely that there are weaker conditions than a trivial automorphism group where $\mathcal{R}$ is still injective.

\begin{question}
What is the weakest set of conditions on $G$ or $\Gamma$ under which $\mathcal{R}$ remains injective?
\end{question}

\subsection{Non-orientable surfaces, signed ribbon diagrams, and their color homologies} \label{subsection:non-orientable-surfaces} All theorems of the blowup of a ribbon graph in this paper can be extended to ``signed ribbon diagrams,'' which is a slight modification of a ribbon diagram motivated by the work on state graphs in the last subsection. This allows us to talk about ribbon graphs whose associated closed surface is non-orientable.  As mentioned in the preliminary section, we could have worked with signed ribbon diagrams from the start, but this would have unnecessarily complicated  formulas, statements, proofs, etc.  

\begin{definition} \label{def:signedribbongraph} Let $G(V,E)$ be a connected graph. A {\em signed ribbon diagram} is a ribbon diagram $\Gamma$ of $G(V,E)$ together with a map $\lambda:E\ra \{1,-1\}$. The value $\lambda(e)$ for $e\in E$ is called the {\em sign} of the edge (or sometimes the {\em signature}).  If the sign of an edge is negative it is drawn in the diagram as a dotted edge. 

A {\em signed perfect matching graph} is a perfect matching graph $\Gamma_M$ (cf. \Cref{Def:pm-graph}) together with a sign function $\lambda:E\ra \{1,-1\}$ where only the signs of the perfect matching edges may possibly be negative, i.e., the signs of non-perfect matching edges are always positive.
\end{definition}

The construction of a ribbon graph from a signed ribbon diagram follows the construction after \Cref{Def:ribbondiagram} with disks and flat bands between them, except a half-twisted band is inserted for each negative edge instead of a flat band. Signed ribbon diagrams are sometimes called ``signed rotation systems'' (cf. \cite{MT}, \cite{Moffat2013}).

\begin{remark}  
Each state graph $\Gamma_\alpha$ of an all-positive-edged ribbon diagram $\Gamma$ can be represented by a signed ribbon diagram with a dotted edge for each $\alpha_i=1$. The projection of the boundary of the ribbon graph $\Gamma_\alpha$ constructed from disks and bands is in fact the state $\Gamma_\alpha$. \label{remark:state-graphs-as-signed-ribbon-diagrams}
\end{remark}

Besides local isotopies and moves of \Cref{Def:ribbonmoves}, there is one new way to change a signed ribbon diagram without changing the ribbon surface it generates: at a vertex $v$, attach the bands in the same order but with the opposite orientation around the vertex as described by the original ribbon diagram and flip the sign from $\lambda(e)$ to $-\lambda(e)$ for each edge $e$ incident to the vertex $v$. This amounts to ``flipping the vertex disk of the ribbon surface over,'' which does not change surface.  For example, flip the vertex disk in the righthand picture of \Cref{fig:R-is-surjective-2} over to get the lefthand picture. As a signed ribbon diagram, the lefthand picture is then a standard ``3-spoke graph'' but with negative edges. A slightly more complicated example of a signed ribbon diagram is shown below:

\begin{figure}[H]
\includegraphics[scale=.15]{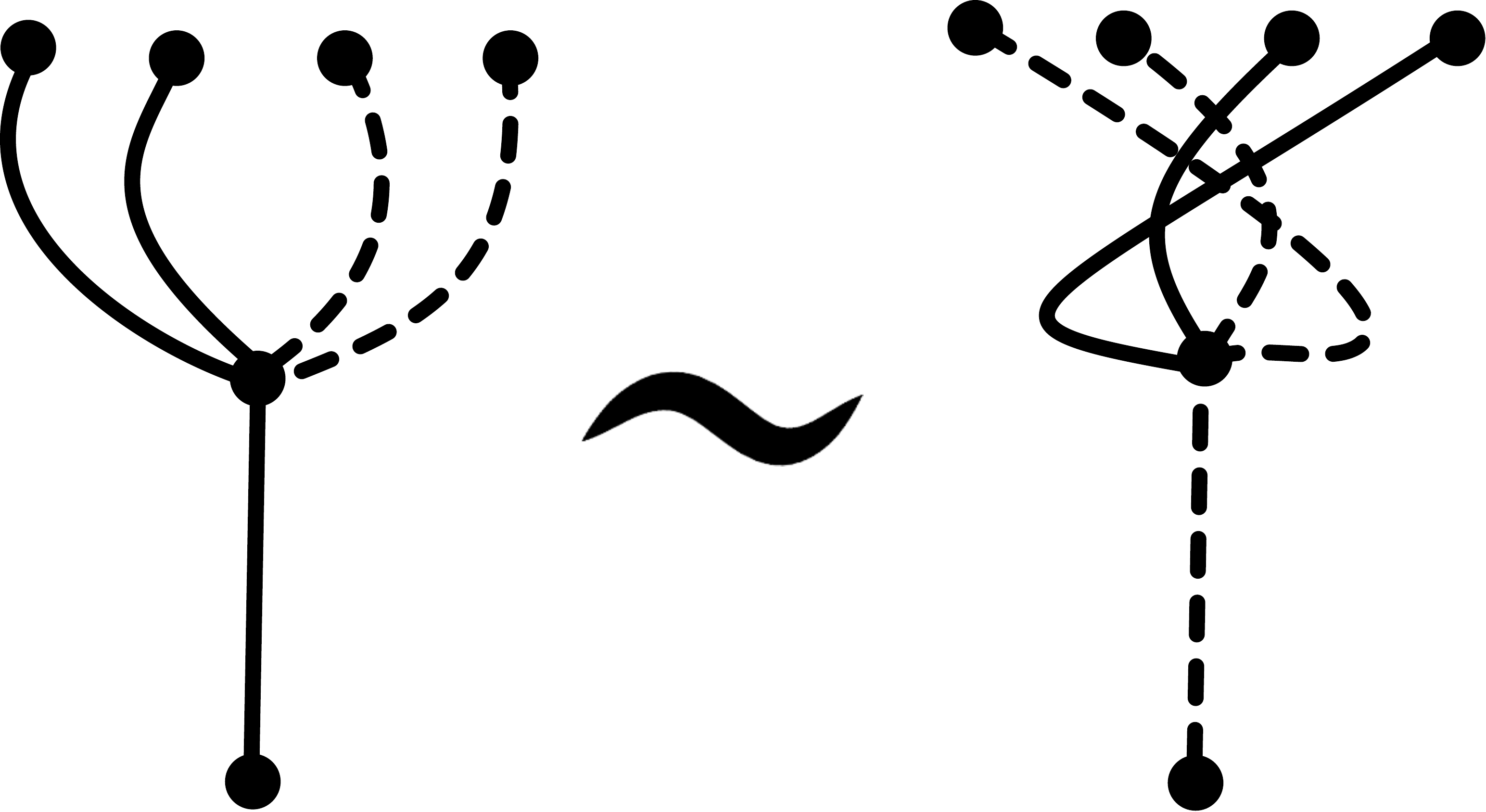}
\end{figure}

We can now describe the hypercube of states for a signed ribbon diagram. Note that this procedure {\em only} makes sense if we are working with signed perfect matching graphs. Although all theorems of this paper can be stated for signed perfect matching graphs, for what comes next, we work with the blowup $\Gamma_E^\flat$ of a signed ribbon diagram $\Gamma$.  Note that the blowup is naturally a signed perfect matching graph if we take the edges of the blowup circle to all be positive edges (cf. \Cref{def:blow-up-of-a-graph}).

Let $\Gamma$ be a signed ribbon diagram with labeled  edges $e_i$ and blowup $\Gamma^\flat_E$. The all-zero state graph $\Gamma_{\vec{0}}$ is constructed from $\Gamma$ using disks and ribbons;  the all-zero state $\Gamma_{\vec{0}}$ is  the projection of the boundary of  this state graph. The state $\Gamma_\alpha$ where $\alpha\in \{0,1\}^{|E|}$ is defined by the signed ribbon diagram as follows: For all edges $e_i$, if $\alpha_i=0$,  keep the sign of $e_i$ the same, and if $\alpha_i=1$, flip the sign of $e_i$ to $-\lambda(e_i)$. According to \Cref{remark:state-graphs-as-signed-ribbon-diagrams}, this signed ribbon diagram corresponds to a ribbon graph $i:G\ra \Gamma_\alpha$.

The formula in Definitions \ref{defn:n-color-poly}, \ref{defn:n-color-number}, \ref{defn:penrose_poly}, and \ref{defn:two_var_penrose_poly} all slightly change for negative perfect matching edges in a way that makes sense with how the hypercube of states is constructed: one simply  exchanges $\IIDiag$ and $\XDiag$ in the formulas. For example, the first formula in \Cref{defn:n-color-poly} for a negative perfect matching edge is:
\begin{eqnarray*}
\langle \NegPMEdgeDiag \rangle_{\! n} &=&   \langle \XDiag \rangle_{\! n} \ - \ q^m \langle  \IIDiag\rangle_{\! n}, 
\end{eqnarray*}
The reason for the exchange follows from understanding how the state graphs of a blowup change from a $0$-smoothing to a $1$-smoothing. In the positive edge case, the band corresponding to the edge starts out flat for a $0$-smoothing and a half-twist is inserted into the flat band to get the $1$-smoothing. In the negative case, the band starts out with a half-twist ($\XDiag$) and another half-twist is inserted into the twisted banded to get the $1$-smoothing. But a band with two half-twists (or a full twist) is the same as a state graph with a flat band ($\IIDiag$). See \Cref{fig:saddle} for a local cobordism version of this process. This same idea works in general for signed perfect matching graphs that are not blowups.

Finally, even though the meaning of a $\alpha_i$-smoothing in $\alpha \in \{0,1\}^{|E|}$ changes depending upon whether the edge is positive or negative, we continue to refer to the ``before'' smoothing as the $0$-smoothing and the ``after'' smoothing as the $1$-smoothing. This keeps the formulas for calculating the internal grading ($q$-grading) and homological grading the same.

With the above modifications understood, the main theorem of this section is:

\begin{theorem} All polynomials, homology theories, definitions, and theorems in this paper that hold for ribbon diagrams also hold for signed ribbon diagrams. Thus, ``oriented ribbon graphs''  can be replaced with ``ribbon graphs'' throughout the paper, unless the theorem (like \Cref{Theorem:mainthm-four-color}) explicitly states that the ribbon graph requires an orientation.
\label{theorem:all-ribbon-graphs-theorem}
\end{theorem}

\section{\Cref{Theorem:mainthmPenrose}: A TQFT approach to the Penrose Polynomial} \label{section:a-TQFT-approach-to-the-Penrose-Polynomial} We are now ready to assemble the theorems  and propositions of this and the previous sections together to prove \Cref{Theorem:mainthmPenrose}. First, \Cref{theorem:fulltheoremA} below is a ribbon graph version of \Cref{Theorem:mainthmPenrose} that we use to build up to the main theorem. To state it, we need:

\begin{definition} \label{def:even-odd-ribbon-graph} Let $\Gamma$ be a signed ribbon diagram of a connected graph $G(V,E)$. A state graph $\Gamma_\alpha$ of the blowup of $\Gamma$ is {\em even} if $|\alpha |$ is even, and {\em odd} otherwise.  More generally, a ribbon graph is called  {\em even} with respect to $\Gamma$ if it is equivalent to $\Gamma$ after inserting an even number of half-twists into some or all of its edges, and {\em odd} otherwise.
\end{definition}

Not all ribbon graphs are even or odd with respect to each other, unless the ribbon graphs are trivalent (cf. \Cref{thm:map-R-theorem}).  Recall that the closed associated surface $\overline{\Gamma}$ to a ribbon graph $\Gamma$ is constructed by gluing disks into the boundary of the ribbon graph (cf. \Cref{Def:ribbongraph}). 

\begin{theorem}
Let $\Gamma$ be a signed ribbon diagram of a connected graph $G(V,E)$. Let $\Gamma_\alpha$ be the state graph for the state $\alpha \in \{0,1\}^{|E|}$ in the hypercube of states of the blowup. Then, for $n\in\BN$,
$$P(\Gamma,n)= \sum_{\mbox{$\alpha$ even}} \# \{\mbox{$n$-face colorings of $\overline{\Gamma}_\alpha$}\} \  - \ \sum_{\mbox{$\alpha$ odd}} \#\{\mbox{$n$-face colorings of $\overline{\Gamma}_\alpha$}\}.$$
Furthermore, if $\Gamma$ is a plane graph, then $P(\Gamma,n)$ is the total number of $n$-face colorings on all state graphs of the blowup $\Gamma^\flat_E$. \label{theorem:fulltheoremA}
\end{theorem}

Since all ribbon graphs (oriented or not) can be represented by signed ribbon diagrams, one needs only choose a signed ribbon diagram representative of the ribbon graph to calculate the polynomial for that ribbon graph.

\begin{proof} The theorem follows from \Cref{prop:EulercharPenrose}, \Cref{MainTheorem:spectralsequence},  Statement (1) of  \Cref{MainTheorem:Colorings-of-State-Graphs}, \Cref{prop:dim-of-CHhat-state-equals-n-face-colorings}, and \Cref{theorem:all-ribbon-graphs-theorem}. The statement about planar graphs follows from \Cref{proposition:even-degree-non-zero-n-face-colorings}.
\end{proof}

At this point, the polynomial distinctly depends upon the ribbon graph and not just the underlying abstract graph. We show how to upgrade this ribbon graph invariant to an abstract graph invariant when $G$ is trivalent.  By \Cref{thm:map-R-theorem}, if $G$ is trivalent, the map $\mathcal{R}$ is onto the set of all ribbon graphs. Therefore any two different ribbon graphs of $G$ are even or odd with respect to each other. Hence, when $G$ is trivalent, \Cref{theorem:fulltheoremA} implies $P(\Gamma,n) = \pm P(\Gamma',n)$ for all $n\in\BN$ for any two ribbon graphs $\Gamma$ and $\Gamma'$ of $G$.  Since the Penrose polynomials of all trivalent ribbon graphs of $G$ are the same up to sign, the polynomial of a ribbon graph with positive leading coefficient can be chosen as the Penrose polynomial of the abstract graph $G$:

\begin{definition} Let $G$ be a connected trivalent graph. A signed ribbon diagram $\Gamma$ of $G$ is called {\em positive} ({\em negative}) if the leading coefficient of $P(\Gamma,n)$ is positive (negative) whenever $P(\Gamma, n)$ is nonzero. It is called {\em zero} if $P(\Gamma,n)$ is identically zero.  Then the {\em Penrose polynomial of $G$} is defined to be $$P(G,n):=P(\Gamma,n)$$ for any nonnegative signed ribbon diagram $\Gamma$ of $G$.\label{def:penrose-poly-abstract-graph}
\end{definition}

The positive leading coefficient guarantees that $P(G,n)>0$ for large enough values of $n$ for graphs when the Penrose polynomial is nonzero.  This definition matches the original definition for trivalent plane graphs that Penrose gave in \cite{Penrose}. Such graphs are always positive.  

We  now have all the necessary ingredients to finish the proof of \Cref{Theorem:mainthmPenrose}. Note that some of the state graphs in the sums of the previous theorem may actually be equivalent as ribbon graphs. Under certain conditions (cf. \Cref{thm:map-R-theorem}), we can remove the need to count the $n$-face colorings of certain graphs more than once. In fact,  \Cref{theorem:fulltheoremA} together with \Cref{thm:map-R-theorem} proves \Cref{Theorem:mainthmPenrose}.

\begin{proof}[Proof of \Cref{Theorem:mainthmPenrose}]
By \Cref{thm:map-R-theorem}, the map $\mathcal{R}$ is a bijection when $G$ is trivalent, loopless, connected, and $Aut(G)=1$.  When $G$ is loopless, we can work with all ribbon graphs of $G$ and not just the subset that makes up the state graphs for the blowup of $\Gamma$. If $G$ has a loop, since $G$ is trivalent, the loop must be at the end of a bridge edge.  Since $G$ has a bridge, there are no $n$-colorings on {\em any} of the ribbon graphs of $G$. Applying \Cref{defn:penrose_poly} directly to the bridge edge and loop shows that the Penrose polynomial is zero. Thus, the theorem holds whether or not $G$ has a loop.

To get an abstract graph invariant, one need only choose a signed ribbon graph representative of $G$ that is nonnegative, that is, the leading coefficient of $P(\Gamma,n)$ is positive if the polynomial is not identically zero.  Since all plane graphs of $G$ are positive signed ribbon graphs, this together with \Cref{proposition:even-degree-non-zero-n-face-colorings} proves the statement about plane graphs in \Cref{Theorem:mainthmPenrose}.
\end{proof}
 
The condition that $Aut(G)=1$ is important in \Cref{Theorem:mainthmPenrose}, even for planar graphs. For example, the automorphism group of the $\theta_3$ plane graph is nontrivial; it is generated by two perpendicular reflections and one automorphism that each exchange a pair of outer edges.  By \Cref{Theorem:mainthmPenrose} or Statement (2) of \Cref{MainTheorem:n-color-polynomial}, $P(\theta_3,3)=[\theta_3]_3 = 24$. However, all ribbon graphs of $\theta_3$ that support $3$-face colorings are equivalent to the plane graph of $\theta_3$, which has six  $3$-face colorings. Therefore when all the $3$-face colorings are added up on all {\em distinct} ribbon graphs of $\theta_3$ the sum is $6$. For example, one pair of state graphs is equivalent by a rotation, and each is equivalent to the plane graph of $\theta_3$ by flipping along the $1$-smoothings:

\begin{figure}[H]
\includegraphics[scale=.4]{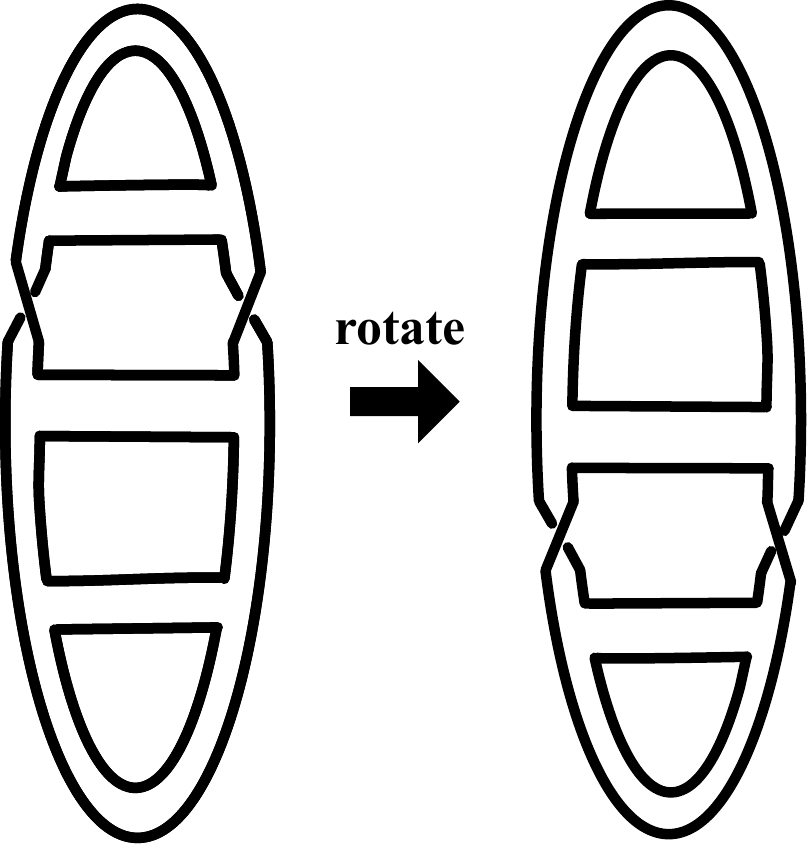}
\end{figure}

\subsection{The total face color polynomial} \label{section:totalfacecolorpoly} Another way to get a trivalent abstract graph invariant is to take the Poincar\'{e} polynomial of the filtered color homology. 

\begin{definition} Let $G(V,E)$ be a connected graph and let $\Gamma$ be any signed ribbon diagram of it. The Poincar\'{e} polynomials of the filtered $n$-color homologies generate the {\em $2$-variable total face color polynomial} and is characterized by
$$T(\Gamma,n,t) := \sum_{|\alpha|=i} t^i \dim \widehat{\mathcal{CH}}_n(\Gamma_\alpha) $$
when evaluated at $n\in\BN$. The {\em total face color polynomial of $\Gamma$} is $T(\Gamma,n):=T(\Gamma,n,1)$.  \label{definition:totalfacecolorpolynomial}
\end{definition}

\begin{remark}
If $G$ is also trivalent, then the total face color polynomial is an invariant of $G$, not just the ribbon graph chosen to define it. This is once again due to  \Cref{thm:map-R-theorem}.  In this case, denote the total face color polynomial by $T(G,n)$. Note that the $2$-variable total face color polynomial  does depend on the ribbon structure even when the graph is trivalent (cf. $K4s$ and $K4t$ in the table below). \label{rem:totalcolorpoly-graphinvariant}
\end{remark}

\begin{example}
When an abstract graph is not trivalent, the total face color polynomials can be quite different for different ribbon graphs of it. The two tori made from a bouquet of three loops in the figure below have a total of $12$ and $36$  $4$-face colorings for all of their respective state graphs:
\end{example}

\begin{figure}[H]
\includegraphics[scale=.8]{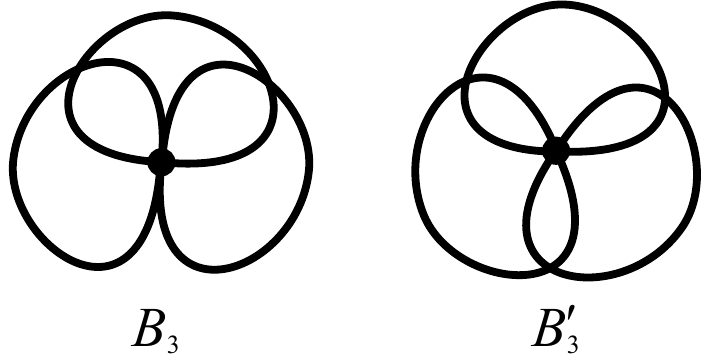}
$$T(B_3,n)=n(n-1) \mbox{ \ \ \ } T(B'_3,n) =n(n-1)^2$$
\end{figure}

If $G$ is trivalent and has $Aut(G)=1$, then $T(G,n)$ counts the total number of $n$-face colorings on {\em all} ribbon graphs of $G$, hence its name. Also, when $G$ is a trivalent planar graph with plane graph $\Gamma$, then $P(G,n)=T(\Gamma,n,-1)=T(G,n)$. These equalities are due to \Cref{proposition:even-degree-non-zero-n-face-colorings}.  Each polynomial ($T(\Gamma,n,t)$, $T(\Gamma,n)$, $T(G,n)$) will be referred to  as the ``total face color polynomial'' when the context is clear.

While the Penrose polynomial is sometimes identically the zero function (cf. \Cref{example:K33-computation-of-Penrose-poly}), the total face color polynomial need not be. For example, the total face color polynomial of the $K_{3,3}$ for $n=4$ is $$T(\Gamma,4,t)=24t^2+48t^3+24t^4+24t^5+48t^6+24t^7,$$   
where $T(\Gamma,4,t)$ is computed for (the blowup of) the ribbon graph $\Gamma$ shown in \Cref{Fig:ribbondiagram}. Thus, $T(K_{3,3},4) = 192$ while $P(\Gamma,n)=0$ for all $n$.

In fact, if $T(G,n)$ is not the zero function for all bridgeless trivalent graphs $G$, then the cycle double cover conjecture is true. This makes the total face color polynomial a powerful detector of closed $2$-cell embeddings/strong embeddings:

\begin{theorem}[Cycle double cover conjecture equivalence]\label{theorem:cycledoublecoverconjectureequivalence}
Let $G$ be a bridgeless connected trivalent graph. Then the total face color polynomial is nonzero if and only if $G$ has a cycle double cover.
\end{theorem}

\begin{proof}
A cycle double cover is equivalent to a strong embedding of $G$ into a (possibly non-orientable) surface. A $2$-cell embedding is {\em strong} if the closure of each face is homeomorphic to a closed disk, i.e., each edge is incident to two distinct faces. Hence a ribbon graph with an $n$-face coloring for some $n>1$ determines a strong embedding of $G$ as well as a cycle double cover.  On the other hand, if $G$ strongly embeds into a surface, let $\Gamma$ be the associated ribbon graph to that embedding. If $f$ is the number of faces of that embedding, then $\dim \widehat{CH}^0_f(\Gamma;\BC) >0$. \end{proof}

If the cycle double cover conjecture is true, then  the total face color polynomial is nonzero for all bridgeless trivalent graphs. Even if the cycle double cover conjecture is true, there may be a non-trivalent ribbon graph whose total face color polynomial is identically zero. This is because a state graph with a coloring may not exist in the hypercube of states for that particular ribbon graph while there is a different ribbon graph of the abstract graph where there is a coloring. A stronger form of the cycle double cover conjecture is:

\begin{conjecture} Let $G$ be a bridgeless connected graph and $\Gamma$ a signed ribbon diagram of $G$. Then the total face color polynomial $T(\Gamma, n,t)$ of $\Gamma$ is nonzero.
\end{conjecture}

This conjecture, if true, also gives a new way to express the four color theorem: If  $\Gamma$ is a bridgeless connected plane graph, then $\Gamma$ is $4$-face colorable if and only if $(n-4)$ is not a factor of $T(\Gamma,n,t)$.  The total face color polynomial links the four color theorem directly to the cycle double cover conjecture.

\subsection{A table of examples of \Cref{Theorem:mainthmPenrose}}
\label{section:table-of-examples}

We have written Mathematica code, available upon request, that calculates the filtered $n$-color homologies for $n=2,3,4,5,6$ of a graph. Below is sampling of ribbon graphs and their filtered $n$-color homologies, where the $n$-color homologies are expressed as the $2$-variable total face color polynomial. For example, the Klein bottle made out of a bouquet graph with a positive and  negative edge, $B_2k$, is $T(B_2k,4,t)=12 t^1$ for $n=4$.  This is the same as $\widehat{CH}^i_4(B_2k;\BC)\cong\BC^{12}$ for $i=1$ and zero elsewhere.\\ 

\noindent\begin{tabular}{| c | l | l |}
\hline
Ribbon & Total face color polynomial & Penrose polynomial\\
Graph&  for $n=4,5$ or $n$ (in bold) & (depends on ribbon graph)\\
\hline \hline
\parbox{.8in}{\centering \vskip.1cm $B_1s$ Sphere\\[.1cm] \includegraphics[scale=0.60]{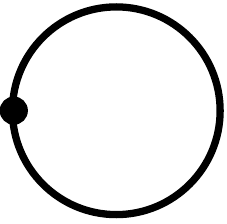}\\[.1cm]} & $T(B_1s,{\bf n},t) = n(n-1) t^0$ & $P(B_1s,n) =n(n-1)$ \\ \hline
\parbox{.8in}{\centering \vskip.1cm $B_1p$ $\BR P^2$ \\[.1cm] \includegraphics[scale=0.60]{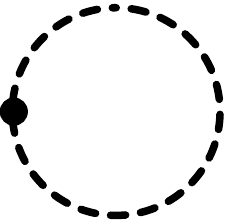}\\[.1cm]} & $T(B_1p,{\bf n},t) = n(n-1) t^1$ & $P(B_1p,n) =-n(n-1)$ \\ \hline
\parbox{.8in}{\centering \vskip.1cm $B_2t$ Torus\\[.1cm] \includegraphics[scale=0.60]{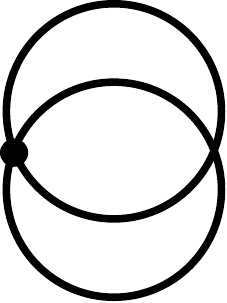}\\[.1cm]} & $T(B_2t,{\bf n},t) = n(n-1) t^2$ & $P(B_2t,n) =n(n-1)$ \\ \hline
\parbox{.8in}{\centering \vskip.1cm $B_2k$ Klein\\[.1cm] \includegraphics[scale=0.60]{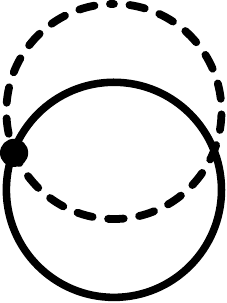}\\[.1cm]} & $T(B_2k,{\bf n},t) = n(n-1) t^1$ & $P(B_2k,n) =-n(n-1)$ \\ \hline
\parbox{.8in}{\centering \vskip.1cm $B_2p$ $\BR P^2$ \\[.1cm] \includegraphics[scale=0.60]{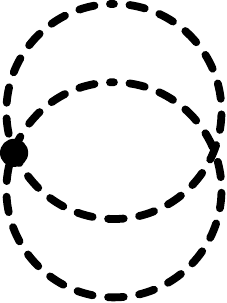}\\[.1cm]} & $T(B_2p,{\bf n},t) = n(n-1) t^0$ & $P(B_2p,n) =n(n-1)$ \\ \hline
\parbox{.8in}{\centering \vskip.1cm $\theta$ Sphere\\[.1cm] \includegraphics[scale=0.70]{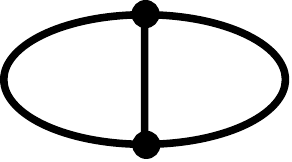}\\[.1cm]} & $T(\theta,{\bf n},t) = [n(n-1)(n-2)] t^0$ & $P(\theta,n) = n(n-1)(n-2)$ \\ \hline
\parbox{.8in}{\centering \vskip.1cm $3$-Prism \\[.1cm] \includegraphics[scale=0.55]{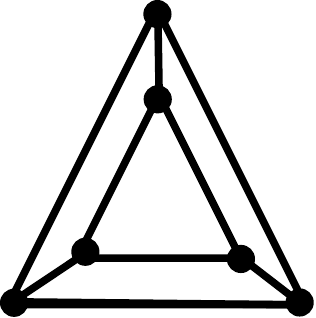}\\[.1cm]} & \begin{minipage}{2.7in}\noindent \begin{eqnarray*}{T(P3,{\bf n},t)} &\!\!\!\!\!\! =\!\!\!\!\!\! & [n(n\!\! -\!\! 1)(n\!\! -\!\! 2)(n\!\! -\!\! 3)^2]t^0 \\ && {+[n(n\!\! -\!\! 1)(n\!\! -\!\! 2)(2n\!\! -\!\! 5)]t^6} \end{eqnarray*}\end{minipage} & $P(P3,n) = n(n-1)(n-2)^3$\\ \hline
\parbox{.8in}{\centering \vskip.1cm $4$-Prism \\[.1cm] \includegraphics[scale=0.55]{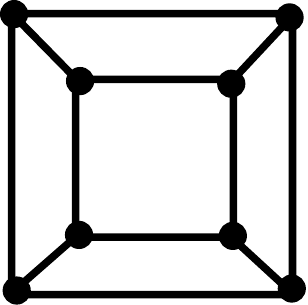}\\[.1cm]} & \begin{minipage}{2.5in}\noindent \begin{eqnarray*} T(P4,{\bf 4},t) &\!\! =\!\!& 96t^0+144t^4+288t^6\\ && +72t^8+144t^{10}+24t^{12}\end{eqnarray*} \end{minipage}  & \begin{minipage}{2.3in}\begin{eqnarray*} P(P4,n) &\! \! \!\!\! =\! \!\!\!\! & -208 n \! +\! 420 n^2 \! -\! 284 \\ & & n^3 \! +\! 83 n^4 \! -\! 12 n^5 \! +\! n^6\end{eqnarray*}\end{minipage} \\ \hline
\end{tabular}

\noindent \begin{tabular}{| c | l | l |}
\hline
 & Total face color polynomial & Penrose polynomial \\
Graph&  for $n=4,5$ or $n$ (in bold) & (depends on ribbon graph)\\
\hline \hline
\parbox{.8in}{\centering \vskip.1cm K3 Sphere \\[.1cm] \includegraphics[scale=0.65]{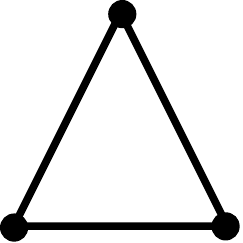}\\[.1cm]} & \begin{minipage}{2.1in}\noindent \begin{eqnarray*}{T(K3,{\bf n},t)} &=& {[n(n-1)] t^0} \\ && +3[n(n-1)]t^2\end{eqnarray*}\end{minipage} & $P(K3,n) = 4(n(n-1))$ \\ \hline
\parbox{.8in}{\centering \vskip.1cm K4s Sphere \\[.1cm] \includegraphics[scale=0.5]{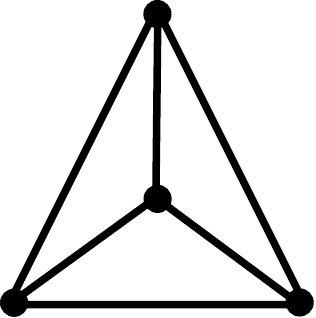}\\[.1cm]} & \begin{minipage}{2.55in}\noindent \begin{eqnarray*} {T(K4s,{\bf n},t)} &\!\!\!\!\! \!\! =\!\!\!\!\!\!\! & {[n(n\!\! -\!\! 1)(n\!\! -\!\! 2)(n\!\! -\!\! 3)]t^0}\\ &&+[n(n\!\! -\!\! 1)(n\!\! -\!\! 2)]t^6\end{eqnarray*}\end{minipage} & $P(K4s,n) = n(n-1)(n-2)^2$ \\ \hline
\parbox{.8in}{\centering \vskip.1cm K4t Torus \\[.1cm] \includegraphics[scale=0.7]{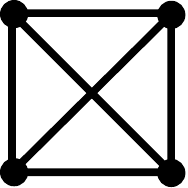}\\[.1cm]} & \begin{minipage}{2.5in}\noindent \begin{eqnarray*} {T(K4t,{\bf n},t)} &\!\!\!\!\! =\!\!\!\!\! & [n(n\!\! -\!\! 1)(n\!\! -\!\! 2)]t^2 \\ && {[n(n\!\! -\!\! 1)(n\!\! -\!\! 2)(n\!\! -\!\! 3)]t^4} \end{eqnarray*}\end{minipage} & $P(K4t,n) = n(n-1)(n-2)^2$ \\ \hline
\parbox{.8in}{\centering \vskip.1cm $K5$ $g\!=\!2$ \\[.1cm] \includegraphics[scale=0.5]{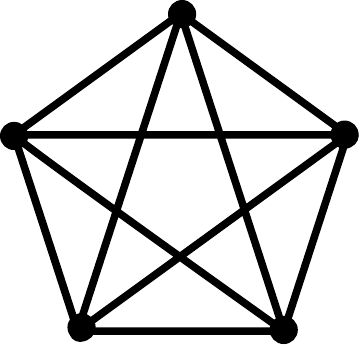}\\[.1cm]} & \begin{minipage}{2.4in}\noindent \begin{eqnarray*} T(K5,{\bf 4},t) &=& 36t^0+120t^3+180t^4\\ && +120t^5+1560t^6\end{eqnarray*} \end{minipage}  & \begin{minipage}{2in}\begin{eqnarray*} P(K5,n) &\!\!\! = \!\!\! & 176 n \! -\! 396 n^2 \! +\! 316 n^3\\ &&  - 115 n^4 \! +\! 20 n^5 \! -\! n^6\end{eqnarray*}\end{minipage} \\ \hline
\parbox{.8in}{\centering \vskip.1cm K6 $g\!=\!4$ \\[.1cm] \includegraphics[scale=0.5]{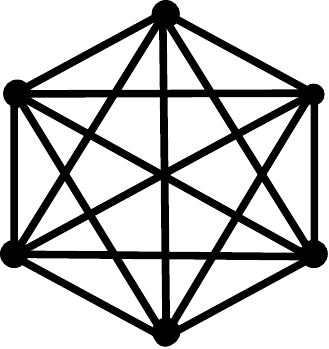}\\[.1cm]} & \begin{minipage}{2.65in}\noindent \begin{eqnarray*} T(K6,{\bf 4},t) &=& 72t^3+384t^6+3024t^7\\ && +5256t^9+864t^{10}\end{eqnarray*} \end{minipage} & \begin{minipage}{2.6in}\begin{eqnarray*} P(K6,n) &\!\!\!\!\! = \!\!\!\!\! & -368n \!\! +\!\! 1472 n^2\!\! -\!\! 2488 n^3 \\ && \!\! +\! 2328 n^4\!\! -\!\! 1327 n^5 \!\!+\!\! 478 n^6\\ && \!\!-\! 109 n^7 \!\!+\!\! 15 n^8\!\! - \!\! n^9
\end{eqnarray*}\end{minipage} \\ 
\hline
\parbox{.8in}{\centering \vskip.1cm $K_{3,3}$ Torus \\[.1cm] \includegraphics[scale=0.45]{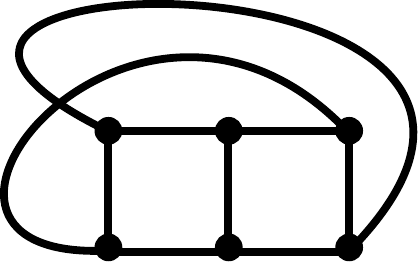}\\[.1cm]} & \begin{minipage}{2.6in}\noindent \begin{eqnarray*} T(K_{3,3},{\bf 4},t) &\!\!\!\!\!\! = \!\!\!\!\!\! & 24t^2+48t^3+24t^4 \\ && \!\! +\! 24t^5+48t^6+24t^7\end{eqnarray*} \end{minipage}   & $P(K_{3,3}, n) = 0$ \\ 
\hline
\parbox{.8in}{\centering \vskip.1cm Petersen \\[.1cm] \includegraphics[scale=0.5]{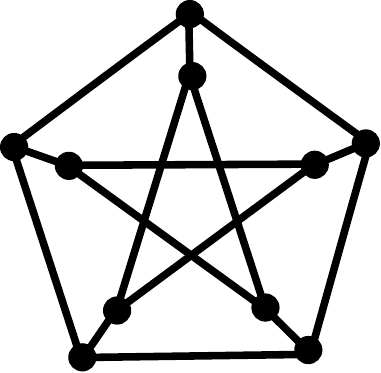}\\[.1cm]} &\begin{minipage}{2.6in}\noindent \begin{eqnarray*} T(Pet,{\bf 4},t)&\!\!\!\!\!\! = \!\!\!\!\!\! & 0 \\T(Pet,{\bf 5},t) &\!\!\!\!\!\! = \!\!\!\!\!\! & 1200t^4\!\! + \!\! 600t^6\!\! + \!\! 600t^7\\ && \!\! +\! 600t^8\!\! +\!\! 2400t^9\!\! +\!\! 600t^{10}\end{eqnarray*} \end{minipage}   & $P(Pet, n) = 0$ \\ 
\hline
\end{tabular}

\bigskip

\Cref{Theorem:mainthmPenrose} can be checked against these examples by inserting $t=-1$  into the total face color polynomial, which gives the Euler characteristic of the filtered $n$-color homology. Then the total face color polynomial evaluated at $n$ will equal the Penrose polynomial evaluated at the same $n$. This calculation can be checked visually for the first half of the table and computed for the second half. For example, $T(K6,4,-1)=-7104$, which equals $P(K6,4)$. 

When the degrees of $t$ in the $2$-variable total face color polynomial are known to be all even (e.g. $\Gamma$ is a plane graph), then the total face color polynomial can be determine from the degree of the Penrose polynomial and a finite number of filtered $n$-color homology calculations.  For example, suppose that the Penrose polynomial is degree $k$. Then inserting $t=-1$ into the total face color polynomial cannot ``cancel'' the degree $k$ term (all such terms have positive coefficients). Therefore the $n$-valued coefficients of each of the $t^{2i}$ terms in the total face color polynomial must be polynomials in variable $n$ of degree $k$ or less.  Since $n=0$ and $n=1$ are zeros of each of these polynomials (c.f. \Cref{thm:penrose-polynomial-values}),  only $k-1$  filtered $n$-color homologies need  to be computed to determine the total face color polynomial. For example, the total face color polynomial for the $P4$ graph is:

\begin{eqnarray*}
T(P4,n,t) &=& n(n-1) (n-2) \left[(-32 + 29 n - 9 n^2 + n^3) t^0 + 
    3 (n - 2) t^4 \right. \\
    & \mbox{} & \left. + 12 ( n - 3 ) t^6 + 3 ( n - 3 ) t^8 + 
    6 ( n - 3 ) t^{10} + (n - 3) t^{12}\right].
\end{eqnarray*}

Using the same logic, there is a bound on the number of $n$-color homologies that need to be computed to determine the total face color polynomial:

\begin{theorem}
Let $G(V,E)$ be a connected graph. Let $\Gamma$ be a signed ribbon diagram of $G$ with $f$ faces and $e=|E|$ edges. Then at most $e+f-1$ filtered $n$-color homologies ($n>1$) need to be computed to determine the $2$-variable total face color polynomial completely. \label{thm:bound-filtered-n-color-homology-calc}
\end{theorem}

\begin{proof}
In the hypercube of states for $\Gamma^\flat_E$, the all-zero state has $f$ circles, which corresponds to contributing a degree $f$ polynomial in variable $n$ for $t^0$ term in the total face color polynomial.  The worst this may change by is a sequence of $\widehat{\Delta}$ maps from the all-zero state to the all-one state. In this case, the all-one state will have $f+e$ circles. Thus, the $t^{e}$ term will have at most a degree $f+e$ polynomial in variable $n$. To determine these polynomials, $f+e+1$ data points are needed. Since $n(n-1)$ is always a factor of each polynomial, only $f+e-1$ data points are required.
\end{proof}

A sharper bound for \Cref{thm:bound-filtered-n-color-homology-calc} is the number of circles in the state with the maximum number of circles, minus one. Also, there are clear relationships between different filtered $n$-color homologies for different $n$, so it may be possible to drastically reduce the number of homologies needed to determine the total face color homology. More generally, research is needed to see if the total face color polynomial is related to any of the other graph polynomial invariants.

\section{ \Cref{Theorem:mainthm-four-color}: TQFT approaches to the four color theorem and other famous conjectures} \label{section:A-TQFT-approach-to-the-4CT} In this section we begin with a discussion of the four-color theorem, that is, every bridgeless planar graph has a $4$-face coloring. Because the current proofs \cite{AppelHaken, AppelHaken2, RSST} rely on computer generated checking of hundreds of cases, there is still a ``black box'' aspect to the result. Hence it is still an interesting endeavor to develop a theory that can elucidate and expose the underlying structure of planar graphs. 

Recently, there have been attempts to prove this theorem and give meaning to that underlying structure using gauge theory. In this section, we compare and contrast our homology to the remarkable theory of instanton homology of webs defined by Kronheimer and Mrowka (cf. \cite{KM3, KM2, KM1}).  A web $K$ is a trivalent graph embedded in $\BR^3$. The first obvious difference between our homology theories and their instanton homology is that a web naturally includes the topology of how it sits in ambient space, while our homologies are ultimately about the surface associated to a ribbon graph, or if trivalent, the abstract graph itself (cf. \Cref{rem:totalcolorpoly-graphinvariant}). However, the theories are strikingly similar in regards to results about the spectral sequences they both generate. It is too soon to say if there is a correlation between the our homology and instanton homology. We note that our homology can be defined directly from a TQFT (see \Cref{section:UnorientedTQFTTheory}), and that the homology appears to be far more computable and richer (in that it is defined for all $n\in \BN$) as exhibited by the table in the last section. This makes it a worthy contender for understanding the structure of graphs and graph coloring.

\subsection{A comparison of our homology to instanton homology} \label{section:comparisionofhomologies} In \cite{KM3}, Kronheimer and Mrowka  constructed an instanton homology $J^\sharp(K)$ for a web $K$ from the Morse theory of the Chern-Simons functional on the space of connections $\mathcal{B}$ associated to $K$.  In \cite{KM1}, they introduced a system of local coefficients $\mathcal{L}$ on $\mathcal{B}$  and used it to to define  a homology $J^\sharp(K;  \mathcal{L})$ as a module over the ring $R=\mathbb{F}[\BZ^3]$. They then derived the following inequality

\begin{equation}
\dim_{\mathbb{F}} J^\sharp(K) \geq \mbox{rank}_R J^\sharp(K;\mathcal{L}), \label{eqn:JsharpleqJsharpL}
\end{equation}
and  three key facts about their homology theories:\\

\hspace{.2in}\parbox{6in}{ \begin{enumerate}
\item[\bf Fact 1:] The vector space $J^\sharp(K)$ is nonzero if and only if $K$ does not have an embedded bridge (Theorem 1.1 of \cite{KM3}).\\
\item[\bf Fact 2:] If $K$ lies in the plane, then the rank of $J^\sharp(K;\mathcal{L})$ as an $R$-module is equal to the number of $3$-edge colorings of $K$ (Theorem 1.2 of \cite{KM1}).\\
\item[\bf Fact 3:] The inequality follows from the fact that $J^\sharp(K)$ is related to $J^\sharp(K;\mathcal{L})$ through a spectral sequence where $J^\sharp(K)$ generates the $E_1$-page (Theorem 6.1 of \cite{KM1}).\\
\end{enumerate}}

Mrowka and Kronheimer point out that, while these facts alone do not prove the four color theorem (it could be that $J^\sharp(K;\mathcal{L})$ is still zero and not violate the inequality),  these homologies do offer a useful perspective on a structural proof of the four color theorem.  

We  prove or note the same three facts about our spectral sequence and produce a similar inequality. In this way, our bigraded $n$-color homology is like $J^\sharp(K)$ and our filtered $n$-color homology is like $J^\sharp(K;\mathcal{L})$.

First, we define a Poincar\'{e} polynomial for the bigraded $n$-color homology: 

\begin{definition}
Let $G(V,E)$ be a connected graph and let $\Gamma$ be any signed ribbon diagram of it.  The Poincar\'{e} polynomials of the bigraded $n$-color homologies generate $T^\sharp(\Gamma,n,t,q)$, which is characterized by
$$T^\sharp(\Gamma,n,t,q):= \sum_{i=0}^{|E|} t^i \left(\sum_j q^j \dim CH_n^{i,j}(\Gamma;\BC)\right)$$
when evaluated at $n\in \BN$. 
\end{definition}

The  inequality below follows from our version of Fact 3: By standard results in spectral sequence theory (cf. Example 1.F in \cite{McCleary}), \Cref{thm:spectral-sequence-for-filtered-homology} implies:

\begin{proposition}\label{prop:inequality-of-Ts}
Let $G(V,E)$ be a connected graph and let $\Gamma$ be any signed ribbon diagram of it. Then
$$T^\sharp(\Gamma,n,1,1) \geq T(\Gamma,n)$$
for all $n\in \BN$.
\end{proposition}

This proposition is the first part of Statement (1) of \Cref{Theorem:mainthm-four-color}.  A notable difference between our inequality and  \Cref{eqn:JsharpleqJsharpL} is that ours holds for all $n\in \BN$ while instanton homology is only similar to the $n=3$ version of \Cref{prop:inequality-of-Ts}. To see this, a plane graph $\Gamma$ satisfies $T(\Gamma,n) = T(\Gamma,n,1) = T(\Gamma,n,-1)=P(\Gamma,n)$ by Statement (1) of  \Cref{MainTheorem:Colorings-of-State-Graphs}.  When $n=3$,  Statement (6) of \Cref{thm:penrose-polynomial-values} implies that $T(\Gamma,3)$ is equal to the number of $3$-edge colorings of $\Gamma$ when $\Gamma$ is trivalent, which is  analogous to Fact 2 above. 

Next,  we prove the rest of Statement (1) of \Cref{Theorem:mainthm-four-color}.  It is the main theorem of this subsection and our version of Fact 1: 

\begin{theorem} \label{theorem:CH-not-zero}
Let $G$ be a connected graph and $\Gamma$ be a ribbon graph of it with $f$ faces. Let $n>1$. Then, for the bigraded $n$-color homology,  
$$CH^{0,k}_n(\Gamma;\BC)=\BC,$$ 
where $k=(1-n/2)f$ and $n$ is even.   Furthermore, if the ribbon graph is oriented, then the conclusion also holds when $n$ is odd and  $k=(1/2-n/2)f$. In particular, $T^\sharp(\Gamma,n,1,1) >0$ for all oriented ribbon graphs regardless if $n$ is even or odd.
\end{theorem}

There are examples of non-orientable ribbon graphs where $CH^{0,*}_n(\Gamma;\BC)=0$ for $n$ odd (cf. $B_1p$ in the table above).  However, we conjecture that for all connected ribbon graphs $\Gamma$, $CH^{i,j}_n(\Gamma; \BC)$ is nonzero for some homological grading $i$ and quantum grading $j$.

This theorem exposes another fundamental difference between our homology and instanton homology: our bigraded $n$-color homology for an oriented ribbon graph is nonzero even when $\Gamma$ contains a bridge. Instanton homology is zero if $K$ lies in a plane and has a bridge. 

\begin{proof}
Let $\Gamma$ be represented by a signed ribbon diagram. Recall that the all-zero state $\Gamma_{\vec{0}}$ of $\Gamma^\flat_E$ is a set of $f$ circles corresponding to each of the faces  of $\Gamma$ and that $\vec{\alpha}_i$ denotes the element of $\{0,1\}^{|E|}$ with a one in the $i$th position and zeros elsewhere.  Therefore, $V_{\vec{0}} = V^{\ot f}$ where $V=\BC[x]/(x^n)$.  Furthermore, recall that each map $\partial_{\vec{0}\vec{\alpha}_i}:V_{\vec{0}} \ra V_{\vec{\alpha}_i}$ that goes from a $0$-smoothing at an edge $e_i$ to its $1$-smoothing is either an  $m$ map, a $\Delta$ map, or an $\eta$ map (cf. \Cref{subsec:face-colorings-of-ribbon-graphs}).

Consider the basis element $\psi(\Gamma) = x^{n-1}\ot x^{n-1}\ot \cdots \ot x^{n-1}$ that corresponds to associating the element $x^{n-1}$ to each of the $f$ circles in $V_{\vec{0}}$. This element constitutes a basis for $C^{0,k}(\Gamma)$ where $k=(1-n/2)f$ if $n$ is even and $k=(1/2-n/2)f$ if $n$ is odd. 

For an edge $e_i\in E$, there is an edge in the hypercube $\Gamma_{\vec{0}} \ra \Gamma_{\vec{\alpha}_i}$.  If there are two different faces adjacent to this edge, then $\partial_{\vec{0}\vec{\alpha}_i}$ corresponds to the map $m:V\ot V \ra V$. In the case of $\psi(\Gamma)$, $m(x^{n-1}\ot x^{n-1}) = x^{2n-2}$, which is zero in $V=\BC[x]/(x^n)$ when $n>1$. Thus, $\partial_{\vec{0}\vec{\alpha}_i}(\psi(\Gamma))=0$ for all $n>1$. If there is only one face (circle) adjacent to this edge $e_i$, then orient the circle.  If the orientation of the circle points in opposite directions along the edge, then $\partial_{\vec{0}\vec{\alpha}_i}$ corresponds to the map $\eta:V\ra V$. Since $\eta(x^{n-1}) = \sqrt{n} (x^{n-1+m})$ where $m=n/2$ if $n$ is even and $m=(n-1)/2$ if $n$ is odd, when $n>1$, $\eta(x^{n-1})=0$ and therefore $\partial_{\vec{0}\vec{\alpha}_i}(\psi(\Gamma))=0$ for all $n>1$.  If the orientation of the circle points in same direction along the edge, then $\partial_{\vec{0}\vec{\alpha}_i}$ corresponds to $\Delta:V\ra V\ot V$. Since $\Delta(x^{n-1})=0$ when $n$ is even and $n>1$, $\partial_{\vec{0}\vec{\alpha}_i}(\psi(\Gamma))=0$ for all even $n>1$. Hence, when $n$ is even, $\psi(\Gamma)$ is in the kernel of $\partial: C^{0,k}(\Gamma) \ra C^{1,k}(\Gamma)$.

If $n$ is odd, then $\Delta(x^{n-1})=x^{n-1}\ot x^{n-1}$, which is not in the kernel of $\partial: C^{0,k}(\Gamma) \ra C^{1,k}(\Gamma)$. However, when $\Gamma$ is oriented, then there exists a signed ribbon diagram representative of $\Gamma$ where all edges are positive. For this diagram,  the orientation of all circles of $\Gamma_{\vec{0}}$ can be chosen so that the circle(s) along each edge are oriented in the opposite direction to each other along the edge.  Hence, there can be no $\Delta$ maps from $V_{\vec{0}} \ra V_{\vec{0}\vec{\alpha}_i}$, only $m$ maps and $\eta$ maps. But in this situation, $\del(\psi(\Gamma))=0$.
\end{proof}

The homology class,  $\psi(\Gamma) = x^{n-1}\ot x^{n-1}\ot \cdots \ot x^{n-1}$, should remind topologists of Olga Plamenevskaya's invariant $\psi(L)$ for a transverse link $L$ (cf. \cite{Olga}).  Indeed, for oriented ribbon graphs, and in particular, plane graphs, $\psi(\Gamma)$ can play a role in the four color theorem.  In the spectral sequence, $\psi(\Gamma) \in E_1$, but one can ask if it lives to the $E_\infty$-page, and if not, on what page  it dies.  All of these questions lead to different invariants of the ribbon graph $\Gamma$.  

If $\Gamma$ is a bridgeless plane graph, $\psi(\Gamma)$ lives to the $E_\infty$-page as a nonzero linear combination of $n$-face colorings of $\Gamma$ in all examples we have calculated (cf. \Cref{prop:dim-of-CH0-equals-n-face-colorings}).  This leads us to make  the following conjecture, which, if true, gives a non-computer based approach to the proof of the four color theorem:

\begin{conjecture}
Let $\Gamma$ be a bridgeless plane graph of a connected graph $G$. Then the invariant, $\psi(\Gamma)\in CH^{0,k}_n(\Gamma;\BC)$, lives to the $E_\infty$-page. \label{conjecture:Psi-lives-to-E-infinity-page}
\end{conjecture}

In other words, the nonzero class, $\psi(\Gamma)$, is the prime candidate in the bigraded $n$-color homology for investigating whether or when $\widehat{CH}_n^0(\Gamma;\BC)$ is nonzero. The value of this nonzero class and the spectral sequence cannot be overstated. For example, try to specify a class in $E_1$ that lives to the $E_\infty$-page knowing only that $\Gamma$ is $n$-face colorable but without knowing any specific $n$-face coloring of $\Gamma$. A first step toward this conjecture would be to prove that $\psi(\Gamma)$ always lives to the $E_\infty$-page when $\Gamma$ is known to be $n$-face colorable.

We conjecture more---that for any oriented ribbon graph $\Gamma$,  $\psi(\Gamma)$ lives to the $E_\infty$-page if and only if $\Gamma$ is $n$-face colorable. Given an oriented ribbon graph $\Gamma$, write down $\psi(\Gamma)$ and explicitly calculate the cocycles determined by $d_1[\psi(\Gamma)]_1=0, d_2[\psi(\Gamma)]_2=0$, etc.  If $\psi(\Gamma)$ lives to the $E_\infty$-page, these calculations will {\em produce} a nonzero linear combination of colorings in $\widehat{CH}_n^0(\Gamma;\BC)$ from $\psi(\Gamma)$. In other words, the algebra of the spectral sequence builds, page by page, a set of colorings of $\Gamma$. In this way our spectral sequence offers a theoretical way to explore a constructible proof of the four color theorem.

\subsection{Obstructions in the spectral sequence to coloring a graph} The idea of constructibility can be flipped around: Instead of constructing $n$-face colorings from $\psi(\Gamma)$, we can view the spectral sequence as a source of {\em obstructions} to building face colorings. Like \Cref{Theorem:mainthmPenrose}, the entire homology of the ribbon graph $\Gamma$ (and not just the degree zero homology) is important for studying these possible obstructions. The following theorem, which is Statement (2) of \Cref{Theorem:mainthm-four-color}, highlights how the full homology can play a role in finding obstructions to plane graph colorings:

\begin{theorem}
Let $\Gamma$ be a plane  graph of a connected planar trivalent graph $G$ and $n=2^k$ for some $k\in \NN$. If there is a state graph $\Gamma_\alpha$ with $\dim \widehat{\mathcal{CH}}_n(\Gamma_\alpha) >0$, then $\Gamma$ is $n$-face colorable. In particular, if any ribbon graph of a connected planar trivalent graph $G$ is $n$-face colorable, then all plane graphs of $G$ are also $n$-face colorable.\label{prop:state-graph-colorable-implies-plane-graph-colorable}
\end{theorem}

An immediate corollary of this proposition and \Cref{Theorem:mainthmPenrose} is Proposition 7 of \cite{Aigner} (cf. \Cref{thm:4-color-number-positive}). Proposition 7 is an important theorem about the Penrose polynomial because it links the four color theorem to the evaluation of the Penrose polynomial at $n=4$.  In comparing the proof of Proposition 7 and the proof below, one sees the power of the TQFT approach: The homology explicitly links  $P(G,n)>0$ to colorings on all ribbon graphs of $G$. 

\begin{proof}
If $|\alpha|=0$, i.e., $\Gamma_\alpha$ is the all-zero state, then the theorem follows from \Cref{prop:dim-of-CH0-equals-n-face-colorings}. Suppose $|\alpha|>0$. Since $\widehat{\mathcal{CH}}_n(\Gamma_\alpha) \not= 0$, there exists an $n$-face coloring of the associated closed surface $\overline{\Gamma}_\alpha$ to $\Gamma_\alpha$. Choose one such $n$-face coloring. There is a one-to-one correspondence with each color $\{c_0,c_1,\ldots, c_{n-1}\}$ and elements of the $n$-dimensional Klein group $K_n= \ZZ_2 \times \cdots \times \ZZ_2$.  Use this relationship to ``color''  the faces of $\overline{\Gamma}_\alpha$ by elements of $K_n$ so that at least one face is labeled $0$.  Pick a face labeled by $0$. Using the additive structure of $K_n$, use an analog of Tait's construction in \cite{Tait} to label each edge of $G$ with a nonzero element of $K_n$. This is a nowhere zero $K_n$-flow of  $G$, which can then be used (again by Tait's construction) to $n$-face color the faces of $\Gamma$.
\end{proof}

In particular, \Cref{prop:state-graph-colorable-implies-plane-graph-colorable} shows that the entire hypercube of states plays a role in finding obstructions to $4$-face coloring a plane graph: A supposed counterexample to the four color theorem would necessitate the entire filtered $4$-color homology of a plane graph to vanish. An example of the entire hypercube of states behaving in this way is when the graph has a bridge. Certainly, based upon \Cref{prop:state-graph-colorable-implies-plane-graph-colorable} above, the $E_\infty$-page of a graph with a bridge must vanish. We have noticed, after computing many examples and not discovering any counterexamples, that it seems to always vanish by the $E_2$-page:

\begin{conjecture}
Let $\Gamma$ be a ribbon graph of a connected graph $G$.  If $G$ has a bridge, then the $E_2$-page of the spectral sequence of $n$-color homologies vanishes for all $n\in \BN$.\label{conj:G-has-a-bridge-implies-E2-vanishes}
\end{conjecture}

\begin{remark} There are ribbon graphs that vanish at the $E_2$-page that do not have bridges; this is not an ``if and only if'' statement like in instanton homology (cf. Fact 1 above).  For example, the theta graph $\theta$ vanishes at the $E_2$-page for bigraded $2$-color homology.
\end{remark}

More generally, classes in the bigraded $n$-color homology act as obstructions to constructing colorings. By \Cref{proposition:even-degree-non-zero-n-face-colorings}, the odd degrees of filtered $n$-colored homology vanish for plane graphs. Therefore, any class that is in an odd degree in the bigraded $n$-color homology has to eventually act as an obstruction to prevent some class in an even degree to live to the $E_\infty$-page.  

Thus, embedded in the spectral sequence of the TQFT is deep structural information of plane graphs relevant to proving the four color theorem. Future research will be on extracting that information. For example, Kempe chains that can be color-switched to get valid $4$-face colorings are likely to show up on the penultimate page to the $E_\infty$-page. One should expect a relationship between the non-colorable class with a Kempe chain and the switched valid $4$-face coloring class on that page. Also, valid  (likewise invalid) $4$-face colorings on one state graph are related to valid (invalid) $4$-face colorings on nearby state graphs in the hypercube of states. We have already seen this type of phenomenon in the color hypercube of a coloring in \Cref{subsection:Poincare-Lemma} and the proof of \Cref{prop:direct-sum-equals-all-harmonics}. (Remember, the Poincar\'{e} Lemma of  \Cref{subsection:Poincare-Lemma} is a major key to unlocking the main theorems of this paper.) Those relationships should exist in the pages of the spectral sequence leading up to the $E_\infty$-page as well.

As a simple first step towards understanding this deeper structure, we can create a new invariant of ribbon graphs:

\begin{definition} For any ribbon graph $\Gamma$, define the {\em spectral invariant of $\Gamma$}, denoted $\mathcal{S}_n(\Gamma)$, to be the smallest number $r\in \BZ$ where $E_r=\widehat{CH}_n^*(\Gamma;\BC)$. 
\end{definition}

For example,  if \Cref{conj:G-has-a-bridge-implies-E2-vanishes} is true, then $\mathcal{S}_n(\Gamma)=2$  for all $n$ for all ribbon graphs whose underlying graph has a bridge.

\subsection{Tutte's $4$-flow conjecture, cycle double cover conjecture, and $T(G,n)$}  In this subsection we prove Statement (3) of \Cref{Theorem:mainthm-four-color}, which generalizes Statements (2)--(5) of \Cref{MainTheorem:n-color-polynomial} to the total face color polynomial (see also \Cref{thm:penrose-polynomial-values}). We then go on to describe a conjecture, which if true, can be used to prove the cycle double cover conjecture. Like our discussion in the previous subsection, a proof of the conjecture would likely involve the entire spectral sequence of the blowup of the ribbon graph and the relationships {\em between} states in the hypercube.

First, we prove the main part of Statement (3) of \Cref{Theorem:mainthm-four-color}:

\begin{theorem}
Let $G(V,E)$ be a connected trivalent graph and $n=2^k$ for some $k\in\BN$.  Then
$$\frac{1}{n}\cdot T(G,n-1) \leq \# \{\mbox{nowhere zero $K_n$-flows of $G$}\},$$
and, for $n=4$,  $T(G,3) = \#\{\mbox{$3$-edge colorings of $G$}\}$.\label{theorem:flow-theorem}
\end{theorem}

\begin{proof} Let $\Gamma$ be a ribbon diagram of $G$ and consider its blowup $\Gamma_E^\flat$. Inspect the proof of Statement (4) and proof of the lower bound in Statement (5) in \Cref{subsection:proof-of-theorem-A} again. These arguments only require the existence of valid colorings on states of the hypercube of $\Gamma^\flat_E$, not whether those states were even or odd with respect to $\Gamma$. Hence, $$T(G,n-1)\leq \#\{\mbox{nowhere zero $K_n$-flows of $\Gamma^\flat_E$}\}.$$

When $n>4$, there are many more nowhere zero $K_n$-flows on the blowup then on the original graph. However, we can use $T(G,n-1)$ to count only flows on the original graph. Recall the algorithm for producing a $K_n$-flow on $G$ in the proof of Statement (4) of \Cref{MainTheorem:n-color-polynomial} in  \Cref{subsection:proof-of-theorem-A}.  Given a valid coloring of the circles in a state by nonzero elements of $K_n$ (after choosing a correspondence as in the proof of Statement (5) of \Cref{subsection:proof-of-theorem-A}), the element of $K_n$ associated to each perfect matching, which in the case of the blowup is each edge of $G$, is the sum of the two elements of the circles adjacent to that edge. Look at a vertex of the original graph. Near a vertex, a valid coloring of a state is a coloring of the faces of the state graph. Since there are always three distinct faces for a valid coloring at a vertex, we can label those faces by three nonzero, distinct elements $a,b,c\in K_n$.  The $K_n$-flow on the edges are then $a+b$, $a+c$, and $b+c$ as shown in \Cref{fig:K-n-flow-of-edges}.

\begin{figure}[H]
\includegraphics[scale=.7]{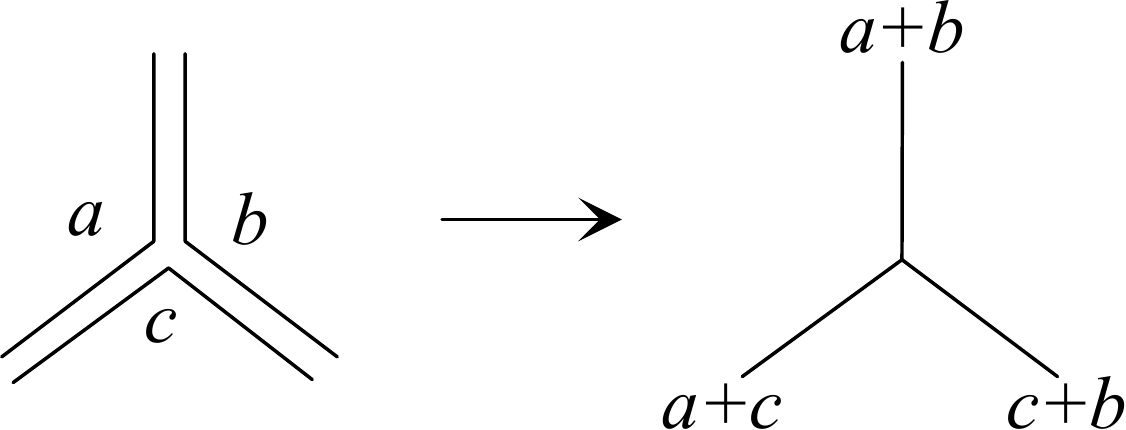}
\caption{How to take a valid $(n-1)$-coloring to a nowhere zero $K_n$-flow.}
\label{fig:K-n-flow-of-edges}
\end{figure}

We claim this defines a  nowhere zero $K_n$-flow.  The fact that it is a flow follows since $(a+b)+(b+c)+(a+c)=2(a+b+c)$, which is equal to zero in $K_n$.  It is a nowhere zero flow by the algebra of $K_n$: if $a\not=0$ and $b\not=0$, and $a\not=b$, then $a+b\not=0$. Finally, the algorithm defines a nowhere zero flow (possibly) up to globally adding an $x\in K_n$ to each element associated to each circle of a valid coloring of the state.  Suppose $a',b',c'\in K_n$ is another set of colors at a vertex  coming from a valid coloring of the state such that $a+b=a'+b'$, $a+c=a'+c'$, and $b+c=b'+c'$.  Then $a+a' = b+b'$, and if $a'=a+x$ and $b'=b+y$, then $x=y$. Since this equality holds over any edge, given a valid coloring by nonzero elements of $K_n$, then globally adding an element $x\in K_n$ to each element will generate the same nowhere zero $K_n$-flow, and importantly, this is the only way to get the same flow. Hence, $T(G,n-1)$ is at most $n$ times the number of nowhere zero $K_n$-flows that it specifically counts, proving the inequality.  

Note that it is often not exactly $n$ times more. When adding a global element $x\in K_n$ to each color of each face (circles of the state), the new valid coloring will sometimes include faces that will be colored $0\in K_n$, which is a valid $n$-coloring but not a valid $(n-1)$-coloring of the state according to our definitions. The colorings that have a $0\in K_n$ for a face do not get counted in $T(G,n-1)$. Therefore, the inequality can be and often is strict. 

When $n=4$, however, we do get a maximal bound. In this situation, the three faces at a vertex are labeled with the three nonzero elements $a,b,c\in K_4$, and {\em any} global nonzero $x\in K_4$ will necessarily be $a$, $b$, or $c$. This will cause the new coloring to always have a face colored by zero, which is not counted by $T(G,3)$. Hence, there is no over counting in this situation and $$T(G,3) \leq \#\{\mbox{nowhere zero $K_4$-flow of  $G$}\}.$$

We show this last inequality is an equality.  The upper bound proof of Statement (2) of \Cref{MainTheorem:n-color-polynomial} in \Cref{subsection:proof-of-theorem-A} showed that, given a $3$-edge coloring (a nowhere zero $K_4$-flow) of $\Gamma_E^\flat$, there must exist a state in the hypercube of $\Gamma_E^\flat$ that has a valid $3$-coloring of the circles by nonzero elements of $K_4$, i.e., a valid $3$-coloring of the faces of the state graph. Again, for $n=4$ only, the state and the coloring are uniquely determined by the nowhere zero $K_4$-flow on $\Gamma_E^\flat$, see \Cref{fig:k4-flow-to-coloring-of-state}.  Nowhere zero $K_4$-flows on the blowup correspond one-to-one with nowhere zero $K_4$-flows on $G$ using Tait's argument. Furthermore, whether the state is even or odd  with respect to $\Gamma$ did not matter with regards to the argument; only that the state exists and that there is a valid $3$-face coloring of its circles. Hence, $T(G,3)$ is also an upper bound for the number of nowhere zero $K_4$-flows of $G$. In general, for $G$ trivalent, $$\{\mbox{nowhere zero $K_n$-flows of $G$}\} \subset \{\mbox{$(n-1)$-edge colorings of $G$}\}$$ and is an equality when $n=4$. Thus, $T(G,3) = \#\{\mbox{$3$-edge colorings of $G$}\}$.
\end{proof}

We suspect that following corollary is a known result, but we were unable to find it after a detailed search of the literature. 

\begin{corollary}\label{cor:3-edge-color-equals-sum-of-3-face-color}
Let $G$ be a connected trivalent graph with trivial automorphism group. Then the number of $3$-edge colorings of $G$ is equal to the sum of the counts of the $3$-face colorings on all ribbon graphs of $G$.
\end{corollary}

In the corollary, existence is not clear: given a $3$-face coloring of a ribbon graph, there exists a $3$-edge coloring of $G$, and vice versa. However, given a $3$-edge coloring, there may exist multiple ribbon graphs with $3$-face colorings on each that generate that $3$-edge coloring. (See the proof of Statement (2) of \Cref{Theorem:mainthm-four-color}, for example.) The corollary follows from the theorem above and \Cref{thm:map-R-theorem}: for a connected trivalent graph with trivial automorphism group, there are exactly $2^{|E|}$ inequivalent ribbon graphs. These ribbon graphs are the state graphs of the blowup of any choice of a ribbon graph $\Gamma$ of $G$. In general, when the automorphism group is not trivial, it is possible for a $3$-face coloring of a ribbon graph to generate multiple distinct $3$-edge colorings of $G$, i.e., there may be multiple state graphs that are equivalent as ribbon graphs.

\Cref{theorem:flow-theorem} when $n=4$ together with  \Cref{theorem:fulltheoremA} show that $T(G,n)$ is the correct generalization of the Penrose polynomial for nonplanar graphs. It also makes sense in terms of other known theorems and conjectures. For example, $T(K_{3,3},3) = 12$ and $T(Petersen, 3)=0$, which match well with Tutte's $4$-flow conjecture. 

Many of the famous flow conjectures can be rephrased in terms of $T(G,n)$. The most important of these in the context of \Cref{theorem:flow-theorem} is Tutte's $4$-flow conjecture. A proof for the conjecture was announced in 1998 by Robertson, Sanders, Seymour, and Thomas; this proof has not yet been completely verified. Like the proof of the four color theorem, it relies heavily on computer calculations. A theory-only proof of it would still be valuable.

The theorem does not directly relate to Tutte's $4$-flow conjecture to $T(\Gamma,n)$ because blowing up a non-trivalent graph can introduce a Petersen minor. However, the proof of \Cref{theorem:flow-theorem} can still be used to deduce: If $G$ is a connected graph with a ribbon graph $\Gamma$ such that $T(\Gamma,n-1)>0$ for $n=2^k$, then there exists a nowhere zero $n$-flow on $G$. Thus, a graph has a nowhere zero $4$-flow if and only if $T(\Gamma,3)>0$ for some ribbon graph of it, and

\begin{scholium}[Tutte's $4$-flow conjecture equivalence] A bridgeless graph $G$ that does not have the Petersen graph as a minor has a nowhere zero $4$-flow if and only if there exists a ribbon graph $\Gamma$ of $G$ such that   $T(\Gamma,3)>0$.  \label{prop:Tutte-4-flow-equivalence} 
\end{scholium}

In \Cref{theorem:flow-theorem}, $\frac{1}{n}\cdot T(G,n-1)$ is a lower bound on the number of nowhere zero $K_n$-flows. The proof of \Cref{theorem:flow-theorem} suggested that $\frac{1}{n}\cdot T(G,n)$ is possibly an upper bound: in the proof, globally adding an element $x\in K_n$ to an $(n-1)$-face coloring always produced a face coloring, but possibly with a face colored $0$, which is still an $n$-face coloring. The proof or a counterexample of either statement below would be interesting:

\begin{conjecture} \label{conjecture:total-color-poly-is-an-upper-bound}
Let $G(V,E)$ be a connected graph and $n=2^k$ for some $k\in \BN$. Then
\begin{enumerate}
\item $\#\{\mbox{nowhere zero $K_n$-flows of $G$}\} \leq \frac{1}{n} T(G,n)$, or
\item if there exists a nowhere zero $K_n$-flow, then $T(G,n)>0$.
\end{enumerate}
\end{conjecture}

If either statement is true, then the cycle double cover conjecture is true by \Cref{theorem:cycledoublecoverconjectureequivalence}.   By Seymour's $6$-flow theorem \cite{Seymour},  every bridgeless graph has a nowhere zero $6$-flow, which means it also has a nowhere zero $8$-flow. Then $T(G,7) > 0$ if Statement (2) is true and, by \Cref{theorem:cycledoublecoverconjectureequivalence}, $G$ has a cycle double cover. 

Right now there is no way to prove this conjecture without first assuming that $G$ has a cycle double cover. With more research on $T(G,n)$, however, one might be able to relate $T(G,n)$ or the homologies themselves to other known invariants of $G$ that would allow one to prove the conjecture without that assumption.

\section{\Cref{MainTheorem:TQFT}: Unoriented TQFT theory}\label{section:UnorientedTQFTTheory}

In this section we prove \Cref{MainTheorem:TQFT}.  Thus we establish that the homology theories of this paper are induced from a topological quantum field theory. 

\subsection{Facts about $\mathcal{UC}ob^n_{/l}$ and geometric complexes}\label{subsection:facts-about-UCOB}
It is well known that oriented $(1+1)$-dimensional Topological Quantum Field Theories valued in the category of vector spaces are classified by Frobenius algebras \cite{Kock}. Tureav and Turner \cite{TT} showed how to classify unoriented $(1+1)$-dimensional TQFTs using extended Frobenius algebras. We introduce the notion of a hyperextended version:

\begin{definition}[Compare to \cite{TT}] A {\em hyperextended Frobenius algebra} is a Frobenius algebra $(V,m,\iota, \epsilon)$ over $R$ together with an involution of Frobenius algebras $\phi:V\ra V$ and an element $\theta \in V$ satisfying the following two axioms:

\begin{enumerate}
\item $\phi(\theta v) = \theta v$, for all $v\in V$, 
\item there exists an $\ell\in \BN$  and a nonzero $a\in R$ such that $m(\phi \ot Id)(\Delta(1)) = a \theta^\ell$.
\end{enumerate}
\label{def:XFA}
\end{definition}

Here the {\em counital comultiplication},  $\Delta:V\ra V\ot V$, is defined by the unique element $\Delta(v) = \sum_i v'_i\ot v''_i$ that satisfies $v\cdot w = \sum_i v'_i \epsilon(v''_i\cdot w)$ for all $w\in V$. This counital comultiplication works automatically for our theory when $n$ is odd. We will need a shifted comultiplication, i.e., the one defined in earlier sections, when $n$ is even (see \Cref{def:HENFA}).  A characterization of this comultiplication in terms of the TQFT will emerge from the work of this section (cf. \Cref{{subsection:ribbon-graphs-to-geo-complex}}).

\begin{remark}
This definition is equivalent to Definition 2.5 in \cite{TT} when $\ell =2$ and $a=1$, in which case it is called an {\em extended Frobenius algebra}. We distinguish $\ell \not=2$ from $\ell=2$ with the use of the term ``hyper,'' but maybe in the future both  will simply be called ``extended.''
\end{remark}

Turaev and Turner showed that the isomorphism classes of unoriented $(1+1)$-dimensional TQFTs over $R$ are in bijective correspondence with isomorphism classes of extended Frobenius algebras over $R$. We show that their result can profitably be interpreted using hyperextended Frobenius algebras. In particular, their definition of $\theta$ is the same definition we use.

Given the work in previous sections of this paper, we will take our ground ring $R=\BC$ throughout this section. However, there are times (like $n=4$) when one may want to consider $R=\BZ$, for example. Furthermore, we take as our Frobenius algebra $(V_t, m, \iota, \epsilon)$ where  $V_t:=\BC[x]/(x^n-t)$ for some choice of $t\in [0,1]$, and $\epsilon:V_t \ra \BC$ with $\epsilon(x^{n-1}) = 1$ and zero otherwise. This is the same algebra described in the introduction and throughout the paper. In particular, $V_0 = V$ and $V_1 = \widehat{V}$ as defined in the previous sections.

The Frobenius algebra $V_t$ can be hyperextended by making the following choices:

\begin{enumerate}
\item set $\phi:V_t\ra V_t$ to be the identity map, and
\item set $\theta=x$, $a=n$ and $\ell = n-1$.
\end{enumerate}
By this point in the paper it should be no surprise that these are the choices that need to be made for when $n$ is odd: If $n$ has a square root in $R$, then by defining $\eta(v) = \sqrt{n}\theta^{(n-1)/2}\cdot v$ we can solve the equation $m\circ\Delta(1)=\eta\circ\eta(1)$, which is necessary for our homology theories. In effect, when $n$ is odd, a square root $\sqrt{a\theta^\ell}\in V_t$ can be found.

When $n$ is even, the square root does not exist using the counital comultiplication. Instead, we need a non-counital comultiplication that can still be used to solve the equation $m\circ\Delta(1)=\eta\circ\eta(1)$. Non-conunital comultplications have been studied in the literature (cf. \cite{GLSU}) under the name {\em nearly Frobenius algebras} (or sometimes {\em non-conuital} or {\em open} Frobenius algebras).  Regardless of whether $n$ even or odd, in this paper it is necessary to always start with a hyperextended Frobenius algebra. Therefore, we prove some basic theorems that relate  $\mathcal{UC}ob_{/l}^n$ to  this algebra  that can be of used for both parities. 

The basic types of cobordisms in $\mathcal{UC}ob^n$ discussed in \cite{TT} are: cylinder (identity), pair of pants (multiplication and comultiplication),  capping disk (birth and death). These are the usual  types of oriented surface cobordisms from the standpoint of Morse theory. Tureav and Turner identified one more for unoriented surface cobordism: a cylinder connect summed with an $\BR P^2$ (cf. Figure 1 in \cite{TT}).  We will take these types as the basic building blocks and mod out by the local relations $NC$, $S$, and $T$ (\Cref{eq:neckcutting,eq:S-relation,eq:T-relation}). Immediately, one gets the following consequences:

\begin{proposition} \label{theorem:UCob-relations} In $\mathcal{UC}ob^n_{/l}$, the following relations hold for all $n$:
\begin{enumerate}
\item  \begin{minipage}[c]{1.5in} \includegraphics[scale=0.70]{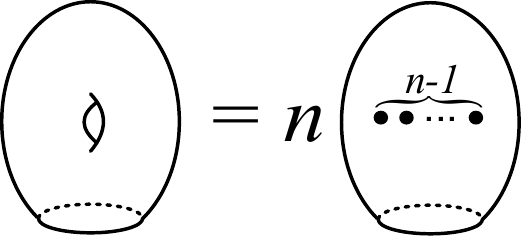}\end{minipage} and\\ \\

\item  \begin{minipage}[c]{2.16in} \includegraphics[scale=0.70]{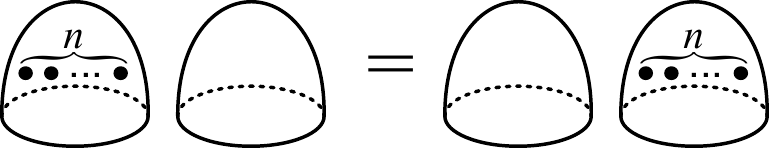}\end{minipage}.
\end{enumerate}
\end{proposition}

In the figures above, each $\bullet$ represents one connect sum with $\BR P^2$.

\begin{proof} 
The first statement follows from applying the neck cutting relation $NC$ to the handle. The second statement follows from applying the $NC$ relation twice to a cylinder with one connect sum $\BR P^2$ attached,  first with the connect sum before the cut and the second after. Equate the two results and cancel like terms (see Lemma 2.1 in \cite{Naot} for a similar proof in knot theory).
\end{proof}

Statement (1) implies that all oriented genus $g$ surfaces (closed or with boundary) are equivalent to a connect sum of some number of projective planes with the same boundary. Hence, from now on we work connect sums of $\BR P^2$.

The second statement implies that we can freely move $\sharp n\BR P^2$ between components of a surface.  Therefore, for any component with $n$ or more connect sums, $n$ copies of $\BR P^2$ can be transferred at a time to a copy of $\sharp (n-1) \BR P^2$, until there are less than $n$ connect sums on that component. This is due to the $T$ relation: First multiply the component by $1$, replace the $1$ with $\sharp (n-1)\BR P^2$, then transfer $n$ copies of $\BR P^2$ to it:

\begin{center}
\includegraphics[scale=0.70]{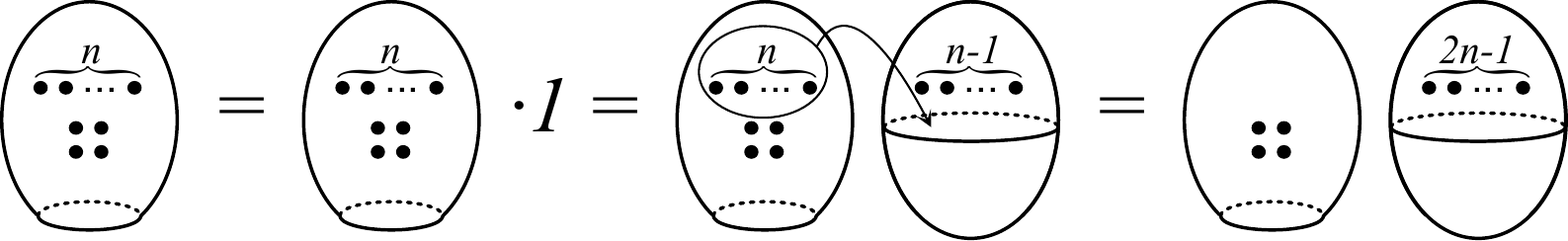}
\end{center}

This can be repeated with the remaining connect sums, each time picking up a copy of $\sharp(2n-1)\BR P^2$. Set $U$ to be the surface $\sharp(2n-1)\BR P^2$. Define $U^k$ to be the disjoint union of $k$ copies of $U$. Set $U^0$ to be the empty set, which is equivalent to $\sharp(n-1)\BR P^2$ after multiplying $\emptyset$ by $1$. Thus, we get:

\begin{theorem}\label{theorem:UCob-presentation}
Over the extended ring $\BC[U]$, where $U= \sharp(2n-1)\BR P^2$, the morphism groups of $\mathcal{UC}ob^n_{/l}$ are generated freely by unions of disks with up to $n-1$ connect sums of $\BR P^2$s, i.e.,
\begin{center}
\includegraphics[scale=0.70]{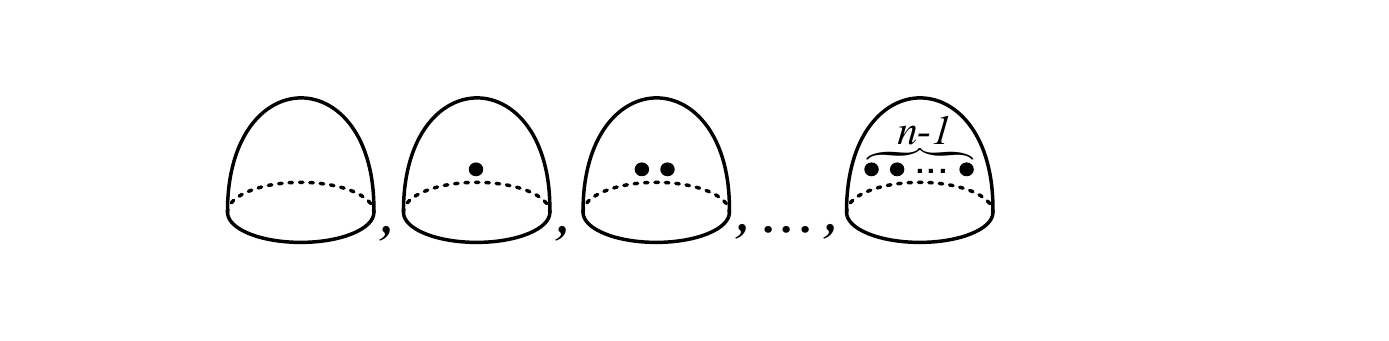}
\end{center}
together with the empty cobordism. Furthermore, $\sharp \ell \BR P^2 = 0$ except for ${\ell=n-1+kn}$ for  $k \in \BZ$, $k\geq 0$, i.e., except when $\sharp \ell \BR P^2 = U^k$.
\end{theorem}

\begin{proof}
It was already shown that any oriented connected surface can be rewritten as a connect sum of $\BR P^2$s, possibly with boundary, by \Cref{theorem:UCob-relations}. Suppose that $W$ is non-orientable and connected. If $W$ has more than one boundary component, the number of boundary components may be reduced by one using the $NC$ relation at the cost of increasing the number of boundary components. Hence, one can assume that $W$ is non-orientable, connected, and with at most one boundary component.  Thus, $W$ may be presented as the connect sum of a disk or a sphere with $\ell$ projective planes, i.e., $W=D^2\sharp k\BR P^2$ or $W=S^2\sharp k \BR P^2$ for some $k \geq 0$. In either situation, using the $T$ relation and \Cref{theorem:UCob-relations}, $W$ can be reduced to a union of one or two disks with up to $n-1$ connect sums of $\BR P^2$ on one of the disks, multiplied by some power of $U$.

Finally, we prove the last statement. By the $S$ relation, $S^2=0$, and an exercise using \Cref{theorem:UCob-relations} shows that $\sharp n\BR P^2 =0$. Set $\BR P^2 =y_1$, $\sharp 2 \BR P^2=y_2$, \ldots, $\sharp (n-2) \BR P^2 = y_{n-2}$. For a moment, set $U=t$ for $t\in \BC$. Applying the neck cutting $NC$ relation to $\sharp \ell\BR P^2$ for $0 \leq \ell \leq n-1$ produces $n$  equations, which can be further reduce to equations  in only variables $y_1, y_2, \ldots, y_{n-2}$ and $t$. Since there are no constant terms (they are homogenous) or $t$-only terms in these equations, $y_1=y_2=\dots = y_{n-2}=0$ is always a solution. If $t=0$, one equation is eliminated---the one found by applying $NC$ to $\sharp (n-1)\BR P^2$, which leaves $n-1$ equations and $n-2$ variables. If $t\not=0$, then there are $n$ equations and $n-1$ variables. Either way, there are $n-2$ ($t=0$) or $n-1$ ($t\not=0$) linearly independent equations, which implies that  $y_1 = y_2 =\dots =y_{n-2}=0$ is the only solution. Note that $t$ can be taken to be zero or nonzero, and in general $U$ can be left unspecified.  Since every  $\sharp \ell \BR P^2$ can be written as  $S^2$, $\BR P^2$, \dots, $\sharp (n-2) \BR P^2$ multiplied by some power of $U$, the last statement follows.
\end{proof}

Compare this theorem to Proposition 2.3 in \cite{Naot}.  The $U$ relation (\Cref{eq:U-relation}) sets $U$ to a number $t$ for some $t\in \BC$. Not specifying the $U$ relation means that $U$ can be viewed as an operator that ``adds'' $n$ connect sums of $\BR P^2$ to any component of a cobordism.

 We can use the presentation in \Cref{theorem:UCob-presentation} of  $\mathcal{UC}ob^n_{/l}$ to write down a way to represent $x^i \in V_t$ in $\mbox{Kom}_{/h}(\mbox{Mat}(\mathcal{UC}ob^n_{/l}))$ and the basic cobordism types, complete with gradings.  
 
First, since we are working in $\mbox{Mat}(\mathcal{UC}ob^n_{/l})$, we give a direction to each of the cobordism types, reading them left-to-right on the page, or from top-to-bottom (the $NC$, $S$, and $T$ and the theorems above are directionless). Thus, \includegraphics[scale=0.15]{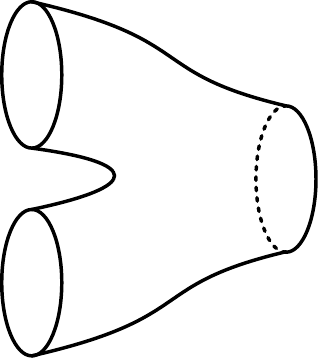} corresponds to multiplication $m$, \includegraphics[scale=0.15]{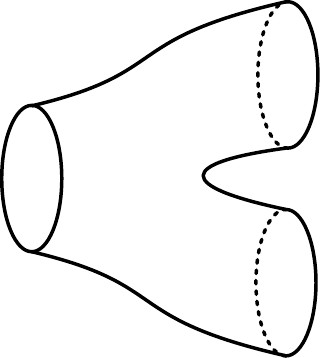} corresponds to comultiplication $\Delta$, \includegraphics[scale=0.35]{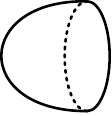} corresponds to a birth/cap $\iota:\BC \ra V_t$, and \includegraphics[scale=0.35]{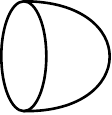} corresponds to a death/cup $\epsilon:V_t\ra \BC$. To determine their gradings, note that the degree of $\sharp(n-1)\BR P^2$ must be zero since it is equivalent to the empty set in $\BC[U]$. In addition to this requirement, choose the degree of $(cylinder)\sharp\BR P^2$  to be $-1$ so that it behaves like multiplication by $x$ in $V_t$. (The degrees of basis elements in $V_t=\langle 1, x, x^2, \ldots x^{n-1}\rangle$ are given in \Cref{subsection:finite-dim-graded-vs}.)  The degrees of the remaining maps are determine by this requirement and choice.   

\begin{center}
\includegraphics[scale=.4]{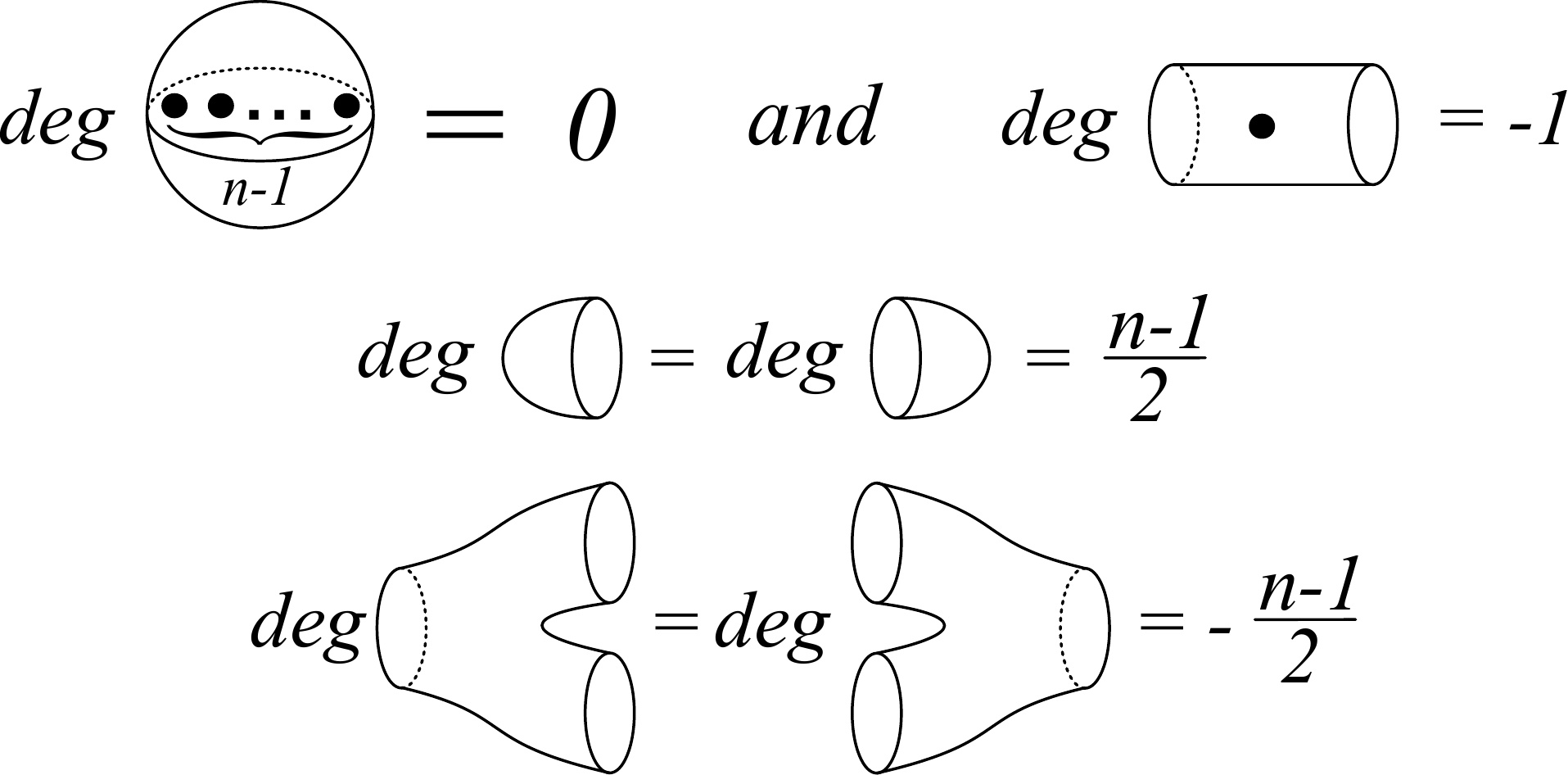}
\end{center}

\begin{remark} When $n$ is odd, the degrees of the cap and cup are $m$ and the multiplications are $-m$.  When $n$ is even, these degrees are half-integers. This still leads to the integer graded homologies of this paper.  If desired, the half-integer degrees for the cup/cap/multiplications can be eliminated entirely  by setting the degree of the cylinder connect sum $\BR P^2$ to $-2$ instead. This will change all of the $q$-gradings in the bigraded $n$-color homology by a factor of two, and will change the $n$-color polynomials by $q\mapsto q^2$ in \Cref{eq:immersed_circle,eq:EMformula,eq:disjoint_graphs_identity} and other places. Continue to call these new  polynomials and homologies by the same names after the change in variables and gradings.
\end{remark}

An exercise using the degrees above shows that the degree of $U$ is $-n$. 

\begin{equation}
\includegraphics[scale=.4]{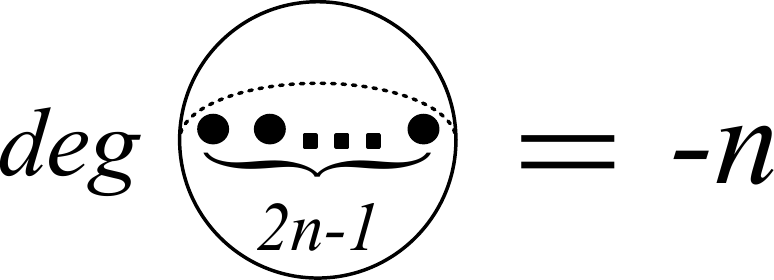}
\label{eq:RP2Relations}
\end{equation} 

\noindent This matches the $n$ in the bidegree $(1,n)$ of $\tilde{\del}$ in \Cref{theorem:tilde-differential}. See the last cobordism in \Cref{example:delta-cobordism} below for how $U$ plays a role in the geometric complex. 

\subsection{Building the geometric chain complex} \label{subsection:geometricchaincomplex} Mimicking Proposition 3.1 in \cite{Naot}, a circle object in $\mbox{Mat}(\mathcal{UC}ob^n_{/l})$ can be replaced with a column of {\em graded empty set objects} with appropriate $q$-gradings (denoted by $\emptyset\{\ell\}$ as defined in \Cref{subsection:finite-dim-graded-vs}). The $NC$ relation in \Cref{eq:neckcutting} provides the model for these graded empty set objects:  

\begin{center}
\includegraphics[scale=.70]{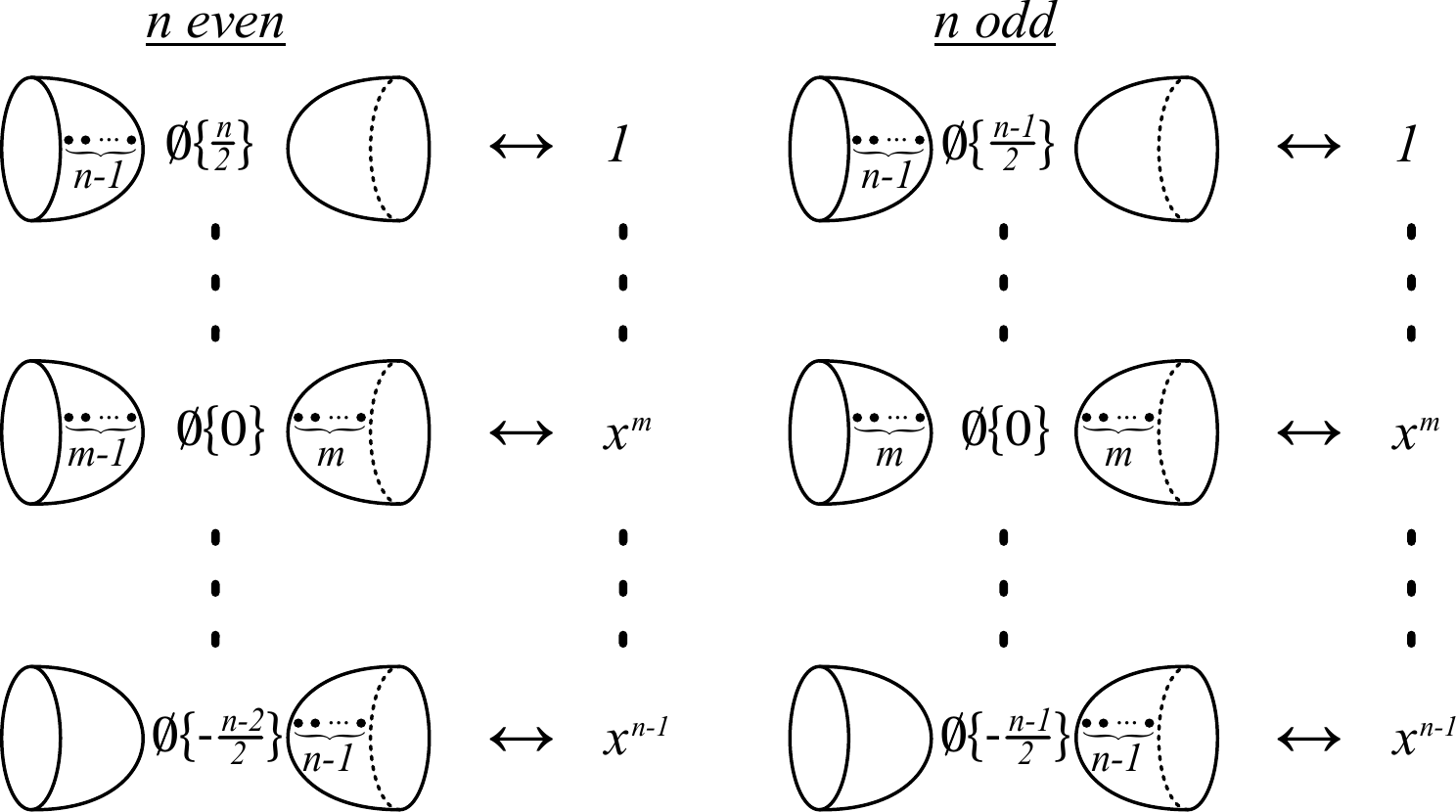}
\end{center}

Compare these empty set objects to the definitions in  \Cref{subsection:finite-dim-graded-vs}. For example, the $n=2$ and $n=3$ follow the examples in \Cref{sec:example-of-2-color-homology,sec:example-of-3-color-homology}:

\begin{center}
\includegraphics[scale=0.70]{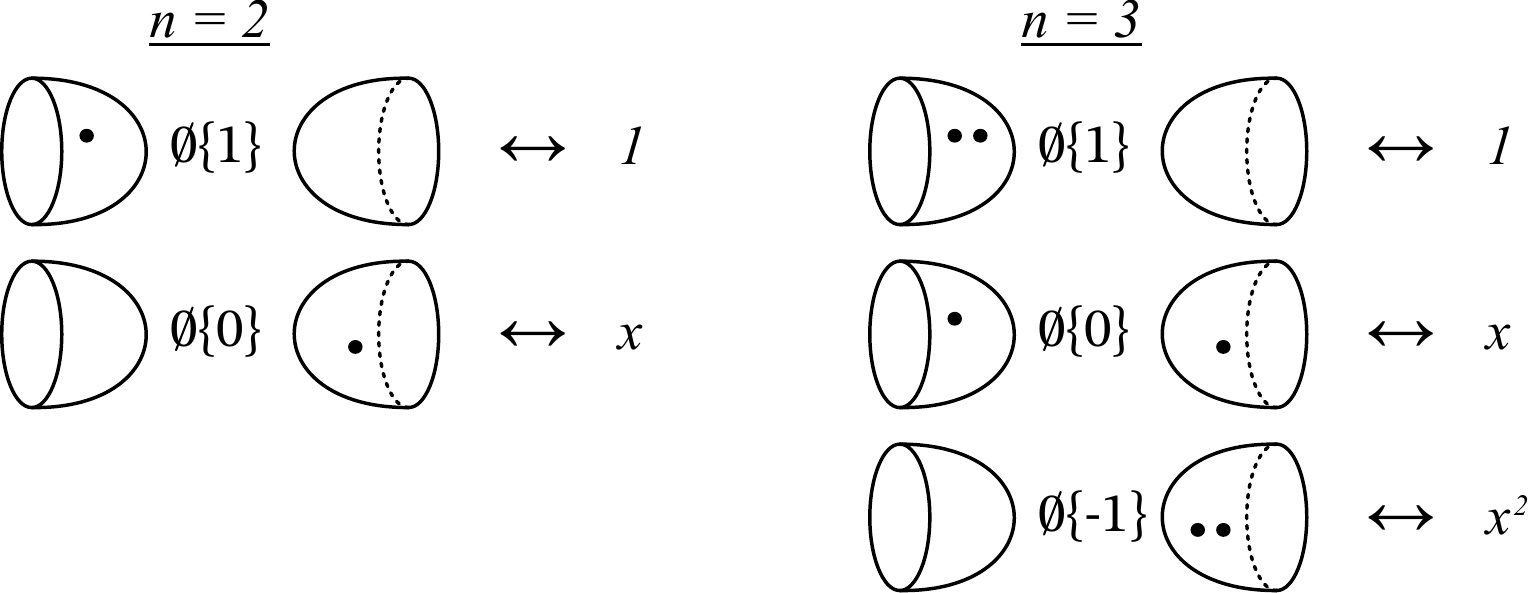}
\end{center}


\subsection{The map from ribbon graphs to geometric complexes} \label{subsection:ribbon-graphs-to-geo-complex} We now have enough background to start to describe the process of taking $\Gamma \ra [[\Gamma]]$ for a ribbon graph $\Gamma$. We work the details out for the blowup $\Gamma_E^\flat$ of $\Gamma$ with its canonical perfect matching, but everything discussed here also works for any perfect matching graph $\Gamma_M$.  Also, for the sake of clarity and ease, we assume that our ribbon diagrams are oriented. By modifying definitions in this section according to the formulas described in \Cref{subsection:non-orientable-surfaces}, the theory also works for signed ribbon diagrams, i.e., any ribbon diagram.

Let $G(V,E)$ be an abstract graph and let $\Gamma$ be a ribbon graph given by an oriented ribbon diagram.  Let $\Gamma^\flat_E$ be the blowup perfect matching diagram of $\Gamma$. For each state $\Gamma_\alpha$ in the hypercube of states of $\Gamma^\flat_E$, and for each circle in that state, associate one of the graded empty set objects above. An association of $1,\dots, x^{n-1}$ to each circle in a state is called an {\em enhanced state} in knot theory. Here we do the same, but in $\mbox{Kom}_{/h}(\mbox{Mat}(\mathcal{UC}ob^n_{/l}))$, we use the graded empty set objects defined above instead. We will refer to both as enhanced states.  If there are $k$ circles in a state, then there will be $n^k$  enhanced states, in fact, we get a natural identification between enhanced states and the basis of $V_t^{\ot k}$, where $V_t= \langle 1, x, x^2, \dots, x^{n-1}\rangle$. 

Maps between states require understanding the meaning of going from a $0$-smoothing to a $1$-smoothing at a smoothing site.  In knot theory, the local piece of the cobordism from a $0$-smoothing to a $1$-smoothing looks like a saddle surface: 

\begin{center}
\includegraphics[scale=.8]{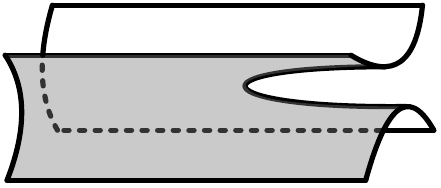}
\end{center}

\noindent There is also a local picture for going from $\IIDiag$ to $\XDiag$ in our theory, which can be described using time slices of $\BR^4$, see \Cref{fig:saddle}. 

\begin{figure}[H]
\includegraphics[scale=1]{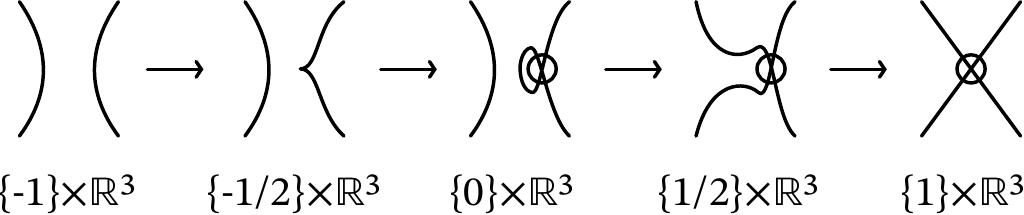}
\caption{A picture of a local cobordism from a $0$-smoothing to a $1$-smoothing. From $t=-1$ to $t=0$, part of the righthand  strand is flipped over in $\BR^3$. The circle at the crossing is to emphasize the intrinsic nature of this cobordism---there are no crossings. From $t=0$ to $t=1/2$, a saddle is inserted between the two strands as in the previous picture.}
\label{fig:saddle}
\end{figure}
Hence, cobordisms from circle(s) of a $0$-smoothing to circle(s) of a $1$-smoothing will be cylinders or pants, possibly with connect sums of $\BR P^2$, possibly multiplied by a constant in $\BC$. For edges in the hypercube of states, these cobordisms can be determined using the degrees of the graded empty set objects and the degrees of the cap/cup and multiplications above:
  
 \begin{enumerate}
 
 \item Cobordisms from two circles to one:
 $$[[\begin{minipage}[c]{.49in} \includegraphics[scale=0.50]{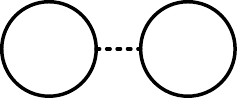}\end{minipage}]] =  \begin{minipage}[c]{.45in} \includegraphics[scale=0.30]{mPants.pdf}\end{minipage} \mbox{if $n$ is odd or even,}$$
 
  \item Cobordisms from one circle to two:
 $$[[\ \begin{minipage}[c]{.44in} \includegraphics[scale=0.50]{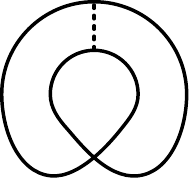}\end{minipage}]] =  \begin{minipage}[c]{.45in} \includegraphics[scale=0.30]{DeltaPants.pdf}\end{minipage} \mbox{if $n$ is odd \ \ \  or \ \ \ }
 [[\ \begin{minipage}[c]{.44in} \includegraphics[scale=0.50]{TwistedCircle.pdf}\end{minipage}]] =  \begin{minipage}[c]{.45in} \includegraphics[scale=0.30]{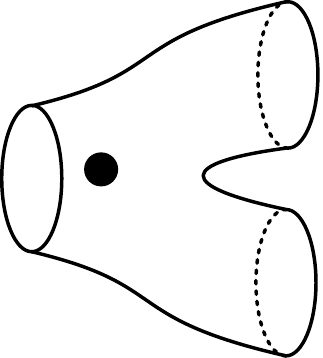}\end{minipage} \mbox{if $n$ is even,}$$
 
  \item Cobordisms from one circle to one circle:
 $$[[ \  \begin{minipage}[c]{.25in} \includegraphics[scale=0.50]{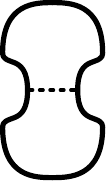}\end{minipage}]] =  \begin{minipage}[c]{.95in} \includegraphics[scale=0.60]{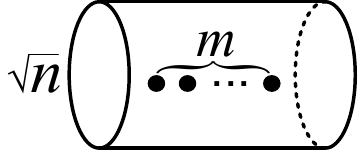}\end{minipage} \mbox{if $n$ is odd or even,}$$ 
\end{enumerate}

The dotted edge in the diagrams above represent the local site where a $0$-smoothing $\IIDiag$ is locally cobordant to a $1$-smoothing $\XDiag$ using the picture in \Cref{fig:saddle}.  Hence the diagrams represent edges in the hypercube of states.

The cobordisms above are determined by the degrees of the maps and graded empty set objects. The next two examples describe the key elements in how these maps are determined and why they are the same for each $n$, except Cobordism (2) where the cobordism depends upon the parity of $n$ only. The general cases are left as an exercise for the reader.

\begin{example}[The $\eta$ cobordism for $n=3$]  The cobordism pictured below takes a circle labeled by $1$ as a graded empty set object and follows it with Cobordism (3) above. By \Cref{theorem:UCob-presentation}, the only graded empty set object that {\em does not} vanish when attached to the end of Cobordism (3) is the one that corresponds to a circle labeled by $x$. (The middle surface must be an $\BR P^2 \sharp \BR P^2$ to be a nonzero cobordism.)
\begin{center}
\includegraphics[scale=.76]{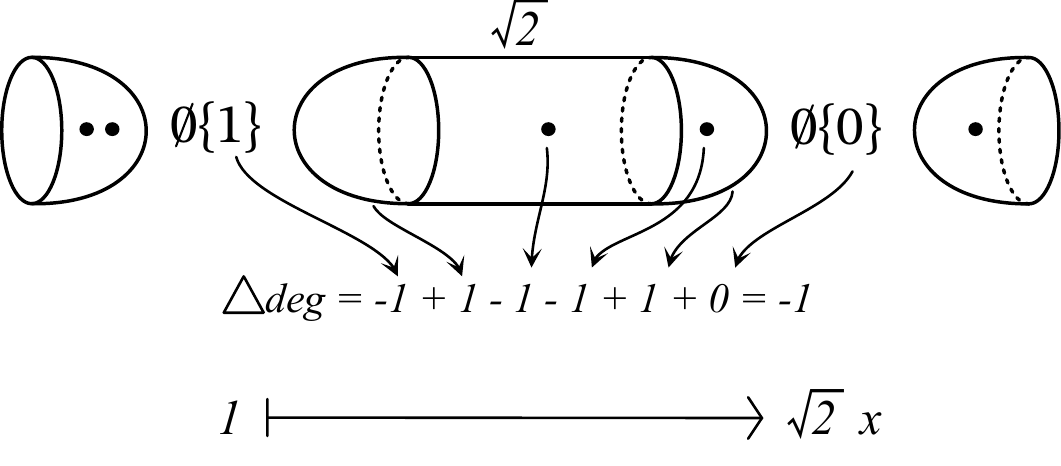}
\end{center}
\end{example}
The change in degree from $1$ to $x$ is $-1$, and after a global shift up by $m=1$, the total degree of the cobordism in the geometric chain complex is $0$ (cf. the definition of $V_\alpha$ in \Cref{subsection:differential} and \Cref{prop:preserve-quantum-degree}). Hence, this map corresponds to the $\eta$-map differential $d_\eta$ with bidegree $(1,0)$.

The most interesting cobordism above is Cobordism (2) when $n$ is even. This is because the pants includes a connect sum of one copy of $\BR P^2$. The following example shows that this $\BR P^2$ is necessary in order to preserve the comultiplication grading.

\begin{example}[The $\Delta$ cobordism for $n=4$] \label{example:delta-cobordism}   In the cobordism pictured below, the degrees of the caps and cups are $\frac{3}{2}$ and the degree of Cobordism (2) is $-1-\frac32$. The extra $\BR P^2$ is needed in Cobordism (2) to ensure that the change in degree of the cobordism is equal to $-2$. After a global shift up of $m=2$ in the geometric chain complex, the total degree of the cobordism in the geometric chain complex is zero. Hence, this cobordism corresponds to one of the three terms of the $\Delta$ map in bigraded $4$-color homology.

\begin{center}
\includegraphics[scale=.75]{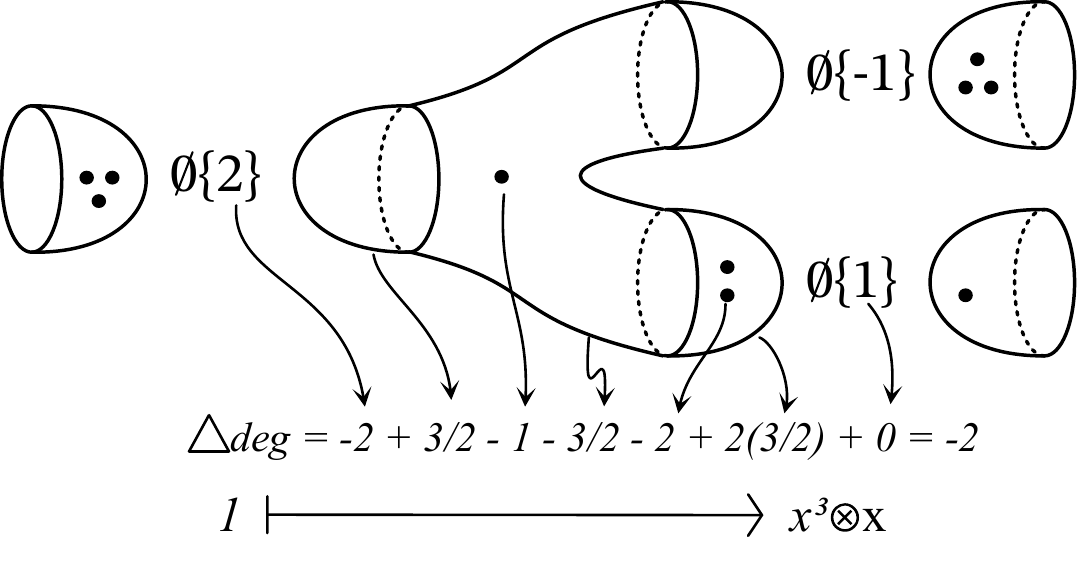}
\end{center}

The other two terms of the $\Delta$ map are found by gluing the graded empty set objects corresponding to $x^2\ot x^2$  and $x\ot x^3$ to the end of Cobordism (2). Hence,  $\Delta(1)=x^3\ot x +x^2\ot x^2 + x\ot x^3$ defined in this paper for $n=4$ corresponds to three  cobordisms, formally summed in  $\mbox{Kom}_{/h}(\mbox{Mat}(\mathcal{UC}ob^n_{/l}))$. See   \Cref{sec:example-of-4-color-homology}.

There is one more cobordism that is nontrivial in the $n=4$ case for $\Delta$. This is when the cobordism has seven connect sums of $\BR P^2$ when going from a graded empty object corresponding to $1$ to two graded empty objects labeled by $1$.  In this case,  $4$ connect sums of $\BR P^2$ can be transferred to $\sharp(3)\BR P^2$ to create a $U$:

\begin{figure}[H]
\includegraphics[scale=.45]{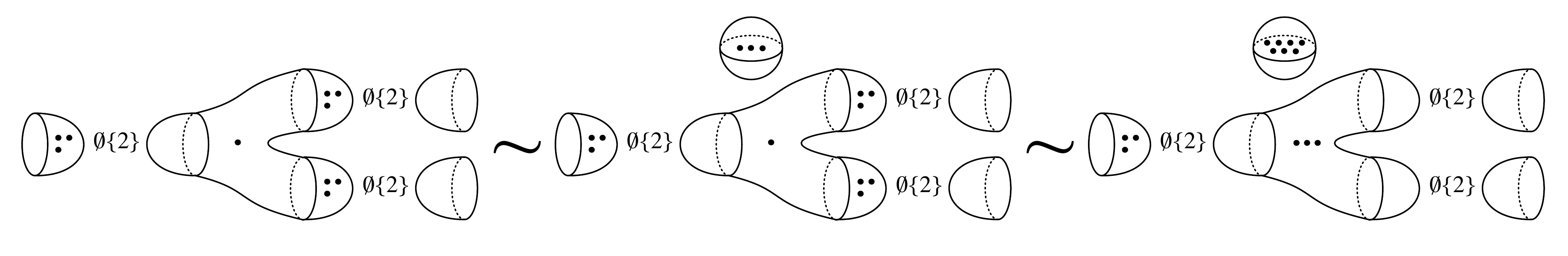}
\caption{An example that is equivalent to a corbordism with a $U$.}
\label{fig:cobordism-with-U}
\end{figure}

Hence, the fourth cobordism corresponds to $1\ra (1\ot 1)U$ where $\deg U= -4$ by \Cref{eq:RP2Relations}.  Setting $U$ to the value $t$ (the $U$ relation) and adding this cobordism formally to the other cobordisms corresponds to the map $\Delta_t(1)=x^3\ot x +x^2\ot x^2 + x\ot x^3 +t(1\ot 1)$. When $t=0$, $\Delta_0 = \Delta$ in the definition of bigraded $4$-color homology of \Cref{subsection:differential}. When $t=1$, $\Delta_1 = \widehat{\Delta}$ in the definition of filtered $4$-color homology of \Cref{subsection:colorbasis}.
\end{example}

Embedded in the discussion of \Cref{example:delta-cobordism} is the necessary insight in how to obtain $[[\Gamma]]$ from $\Gamma$ in general. First, generate columns of graded empty set objects using the circles in each state $\Gamma_\alpha$ of the hypercube of states to get the enhanced states. Shift the gradings of each enhanced state by  $m|\alpha |$ as in the definition of $j$ in  \Cref{subsection:differential}. Next, for an edge $\Gamma_\alpha \ra \Gamma_{\alpha'}$, use the identity cobordism between all circles in the enhanced state that are not part of the change from a $0$-smoothing to a $1$-smoothing, and insert Cobordisms (1)-(3) depending on how the number of circles  change between $\Gamma_\alpha$ and $\Gamma_{\alpha'}$.  This insertion is done between {\em all} enhanced states of the initial state to {\em all} enhanced states of the target state. Most of these cobordisms are equivalent to zero by \Cref{theorem:UCob-presentation}; only cobordisms like the four in \Cref{example:delta-cobordism} are nonzero, for example. The geometric chain complex formed from this process is then $[[\Gamma]]$.

\subsection{Proof of \Cref{MainTheorem:TQFT}} 
\label{subsection:functorFandproof}

In this subsection, we describe the TQFT of this paper, that is, a functor from $\mathcal{UC}ob^n_{/l}$ to the category $\BC\mbox{Mod}$ of graded $\BC$-modules given by taking disjoint unions of circles to tensor products. We define the functor on generators of $\mathcal{UC}ob^n_{/l}$.  The objects and morphisms are now familiar: the object is a single circle $\bigcirc$ and the morphisms are the cap/birth \includegraphics[scale=.40]{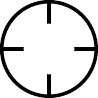}, cup/death  \includegraphics[scale=.30]{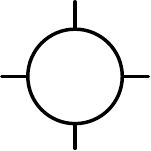}, projective cylinder  \includegraphics[scale=0.4]{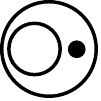}, i.e., a cylinder with a connect sum of an $\BR P^2$, pair of pants \includegraphics[scale=0.30]{TwistedCircle.pdf} (with a connect sum of $\BR P^2$ when $n$ is even), upside down pair of pants  \includegraphics[scale=0.40]{2Circles.pdf},  twisted cylinder \includegraphics[scale=0.30]{EtaCircle.pdf},  and $U$ where $U=\sharp(2n-1)\BR P^2$.

\begin{definition}\label{def:functor-math}
Let  $t\in\BC$ and $n\in \BN$ with $n>1$. Let $V$ be the graded $\BC$-module freely generated on $n$ elements $\{v_0,v_1,\ldots,v_{n-1}\}$ with $\deg v_i$ corresponding to the degree of $x^i$  for $n$ even or odd described in \Cref{subsection:geometricchaincomplex} (see also in \Cref{table:quantum-gradings}). Then $\mathcal{F}$ is the {\em TQFT} defined by $\mathcal{F}(\bigcirc)=V$ and by $\mathcal{F}(\raisebox{-0.23\height}{\includegraphics[scale=.40]{iota.pdf}})=\iota:\BC\ra V$,  $\mathcal{F}(\raisebox{-0.23\height}{\includegraphics[scale=.30]{epsilon.pdf}})=\epsilon:V\ra\BC$, $\mathcal{F}(\raisebox{-0.23\height}{\includegraphics[scale=.40]{projectivecylinder.pdf}})=\rho_t:V\ra V$, $\mathcal{F}(\raisebox{-0.23\height}{\includegraphics[scale=.30]{TwistedCircle.pdf}})=\Delta_t:V\ot V\ra V$, $\mathcal{F}(\raisebox{-0.23\height}{\includegraphics[scale=.30]{2Circles.pdf}})=m_t:V\ra V\ot V$, and $\mathcal{F}(\raisebox{-0.23\height}{\includegraphics[scale=.30]{EtaCircle.pdf}})=\eta_t:V \ra V$, and $\mathcal{F}(U)=t$.  The maps $\iota, \epsilon, \rho_t, m_t,\Delta_t, \eta_t$ are defined below.
\end{definition}

The product maps $m_t$ and $\Delta_t$ are defined similarly to Equation \ref{eq:wide-hat-differential-m}. The  difference is there is a factor of $t$ on all terms that do not correspond to the product maps $m$ and $\Delta$ in \Cref{subsection:differential}, i.e., when $t=0$  the $m_t$ and $\Delta_t$ maps become $m$ and $\Delta$ maps. With $m_t$ defined, $\rho_t(v) = m_t(v_1\ot v)$ and $\eta_t(v) = \sqrt{n}\rho^m(v)$ where $m=(n-1)/2$ when $n$ is odd and $m=n/2$ when $n$ is even.

One example should suffice to show how these maps are defined in general:

\begin{example} Some examples of product maps  and the $\eta_t$ map in $V$ when $n=3$:

\noindent\begin{minipage}{.7in}
\begin{eqnarray}\nonumber
m_t(v_0\ot v_i) &=& v_i  \\ \nonumber
m_t(v_1 \ot v_1) &=& v_2\\ \nonumber
m_t(v_1 \ot v_2) & =& t v_0\\ \nonumber
\end{eqnarray}
\end{minipage}\hspace{.5in}
\hspace{.5in}\begin{minipage}{2in}
\begin{eqnarray}\nonumber
\Delta_t(v_0) &=&v_0\ot v_2 + v_1\ot v_1+ v_2 \ot v_0 \\ \nonumber
\Delta_t(v_1) &=& v_1\ot v_2 +v_2\ot v_1 +t v_0 \ot v_0\\ \nonumber
\Delta_t(v_2) &=& v_2\ot v_2 + t v_0 \ot v_1 + t v_1 \ot v_0\\ \nonumber
\end{eqnarray}
\end{minipage}\hspace{1in}\begin{minipage}{2in}\noindent
\begin{eqnarray}\nonumber
\eta_t(v_0)&=& \sqrt{3} v_1 \\ \nonumber
\eta_t(v_1) &=& \sqrt{3} v_2 \\ \nonumber
\eta_t(v_2) &=&t\sqrt{3} v_0\\ \nonumber
\end{eqnarray}
\end{minipage}
\noindent Compare these maps with the maps in \Cref{sec:example-of-3-color-homology} when $t=0$.
\end{example}

This only leaves the definitions of the cap and cup maps, which are defined as follows: $\iota:\BC \ra V$ is given by $\iota(1) = v_0$ and $\epsilon:V\ra \BC$ is given by $\epsilon(v_{n-1})=1$ and is zero on the other generators. Both maps are extended linearly.

 \Cref{MainTheorem:TQFT} follows directly from the next theorem given the definition of $[[\Gamma]]$ above.

\begin{theorem} The TQFT $\mathcal{F}$ descends to a functor $\mathcal{UC}ob^n_{/l}\ra \BC\mbox{Mod}$, and therefore extends to a functor $$\mathcal{F}:\mbox{Kom}_{/h}(\mbox{Mat}(\mathcal{UC}ob^n_{/l})) \ra\mbox{Kom}(\BC\mbox{Mod}).$$
\end{theorem}

\begin{proof} The functor is well-defined: it respects the relations between the set of generators of $\mathcal{UC}ob^n$, i.e., the relations defining a Frobenius algebra.  It also satisfies the definition of a hyperextended Frobenius algebra over $\BC$ by choosing $\phi:V\ra V$ to be the identity map, $\theta=v_1$, and $\ell=n-1$ and $a=n$.  Note that $\eta_t^2(v_0)=n\rho_t^{2m}(v_0)$, and when $n$ is odd,  
\begin{equation}\label{eq:extendedFrobeniusrelationodd}
m_t(\Delta_t(v_0)) = m_t(\sum_{i=0}^{n-1} v_i \ot v_{n-i-1}) = n v_{n-1} = \eta_t^2(v_0).
\end{equation}
When $n$ is even,
\begin{equation} \label{eq:extendedFrobeniusrelationeven}
m_t(\Delta_t(v_0)) = m_t(\sum_{i=1}^{n-1} v_i \ot v_{n-i} + t(v_0\ot v_0)) = n t v_o=\eta_t^2(v_0).
\end{equation}
See \Cref{fig:P_3_A-squared-ex} for an example in the hypercube of states for why  \Cref{eq:extendedFrobeniusrelationodd} and \Cref{eq:extendedFrobeniusrelationeven} are important and necessary for defining the homology theories in this paper.

The functor $\mathcal{F}$ satisfies the $S$, $T$, and $NC$ relations. The $S$ relation is the simplest: $S^2=0$ is equivalent to a cap followed by a cup, which is equivalent to showing $\epsilon(\iota(1))=0$. This holds for all $n>1$.

The $T$ relation requires some care.  A torus is a cap followed by a pair of pants followed by an upside down pair of pants, then finally by a cup. When $n$ is odd, $T^2=n$ follows straightforwardly from \Cref{eq:extendedFrobeniusrelationodd}:
\begin{equation}\label{eq:calc-that-T2-equals-one}
 \mathcal{F}(T^2)=(\epsilon\circ m_t \circ \Delta_t \circ \iota)(1) =\epsilon\left(m_t(\Delta_t (v_0))\right) = \epsilon(n v_{n-1})  = n.
\end{equation}
But this calculation does not  work using \Cref{eq:extendedFrobeniusrelationeven} when $n$ is even. This is because  $\Delta_t$ defined above for building the geometric chain complex (and homology theories) is different than the comultiplication defined by the counit $\epsilon$ of the Frobenius algebra.  When $n$ is even, the map $\Delta_t$ is a pair of pants connect summed with a copy of $\BR P^2$. The $n$ even case is an example of a {\em shifted comultiplication}. We discuss these comultiplications in the next subsection. The calculation to show $T^2=n$ in the even case uses only a pair of pants, i.e., the comultiplication defined by the counit $\epsilon$.  

Let $\Delta_\epsilon:V\ra V\ot V$ be the counital comultiplication defined by the Frobenius algebra. Then $\Delta_\epsilon(v_0) = \sum_{i=0}^{n-1} v_i \ot v_{n-i-1}$, and in general, $\Delta_t(v)=\Delta_\epsilon(\rho_t(v))$ when $n$ is even. (Both maps can be computed directly by applying the neck cutting relation ($NC$) twice to a pair of pants together with applications of $\rho_t$.) Replacing $\Delta_t$ with $\Delta_\epsilon$ in \Cref{eq:calc-that-T2-equals-one} shows that $\mathcal{F}(T^2) = n$ when $n$ is even as desired. Hence, when $n$ is even, the isomorphism class of the category $\mathcal{UC}ob^n_{/l}$ corresponds to a hyperextended Frobenius algebra with its counital comultiplication $\Delta_\epsilon$, but the map used to define the homology theories is the non-counital shifted comultiplication $\Delta_t$.

Next we prove the neck cutting relation.  The righthand side of the neck cutting relation \Cref{eq:neckcutting} is the following expression in $\BC\mbox{Mod}$:
$$  (\rho_t^0 \circ \iota\circ \epsilon\circ \rho_t^{n-1})(v)+(\rho_t^1 \circ \iota \circ \epsilon\circ \rho_t^{n-2})(v)+ \dots+ (\rho_t^{n-1} \circ \iota\circ \epsilon\circ \rho_t^{0})(v)$$
This expression is the identity on $\{v_0,\dots,v_{n-1}\}$.  For example, only the second  term is nonzero when applied to $v_1$,
\begin{eqnarray*}
(\rho_t^1 \circ \iota \circ \epsilon\circ \rho_t^{n-2})(v_1) & = & \rho_t^1 \circ \iota\circ \epsilon(v_{n-1}) = \rho_t(v_0) = v_1,
\end{eqnarray*}
the remaining terms result in computing $\epsilon(v_i)$ for $i\not=n-1$.
Thus, $\mathcal{F}$ of the righthand side of \Cref{eq:neckcutting} is the identity map, which equals $\mathcal{F}$ of the  cylinder. 

Finally, the extension to complexes is done by taking formal direct sums into honest direct sums. Note that at this stage, circle objects have been replaced with a column of graded empty set objects (corresponding to the generators of $V$) in the geometric chain complex. Verifying that $\mathcal{F}$ respects degrees is left as an exercise to the reader.
\end{proof}

\subsection{Shifted comultiplications in hyperextended Frobenius algebras} 
\label{subsection:ExtendedNearlyFrobeniusAlgebras}

Looking back over the paper, one sees that the choice of  comultiplication $\Delta_t$ is crucial. It must simultaneously
\begin{enumerate}
\item preserve the quantum grading when $t=0$ and preserve the filtration when $t=1$,
\item be injective when $t=1$ so that $\Delta_1(c_i)$ is proportional to $c_i\ot c_i$ on the color basis, and
\item be used to define a ``square root'' of $m_t\circ\Delta_t(1)$ so that $m_t\circ \Delta_t = \eta_t \circ\eta_t$ can be solved for all $t$.
\end{enumerate}
Without these properties, the homology theories cannot capture face coloring information about the ribbon graph. The counital comultiplication given by $\epsilon(x^{n-1})=1$ in the hyperextended Frobenius algebra satisfies all three conditions when $n$ is odd and does not when $n$ is even. When $n$ is even, we used a comultiplication that does satisfy the three conditions. This comultiplication is, in the following sense, a ``shifted'' version of the counital comultiplication:

\begin{definition} Let $(V, m,\iota, \epsilon, \phi,  \theta, a, l)$ be a hyperextended Frobenius algebra over $R$ and let $k\in \BZ$, $k\geq 0$.  A {\em shifted comultiplication}, $\Delta:V\ra V\ot V$, is defined by the unique element $\Delta(v) = \sum_i v'_i\ot v''_i$ that satisfies $\theta^k\cdot v\cdot w = \sum_i v'_i \epsilon(v''_i\cdot w)$ for all $w\in V$.
\label{def:HENFA}
\end{definition}

In our homology theories, when $n$ is odd, the shifted comultiplication we used is $k=0$. When $n$ is even, the reader can check that $k=1$ was used.  Our homology theories show that that there are nontrivial results for shifted comultiplications when $k>0$.  Indeed, there should be other homologies for other choices of $k$ or by choosing a different counit (e.g. $\epsilon(x)=1$ instead of $\epsilon(x^{n-1})=1$).  We leave all of these considerations for possible future research (cf. \cite{K1} or \cite{BN3}).

\section{Final remarks}

In this paper we introduced invariants of graphs and showed that many known results in graph theory follow as corollaries from these stronger invariants, see \Cref{MainTheorem:n-color-polynomial} for example. In some cases, the invariants lead to simple-to-state results of graph theory itself (cf. \Cref{cor:3-edge-color-equals-sum-of-3-face-color}) or could be used to generalize and give new meaning to 50-year-old polynomials (\Cref{Theorem:mainthmPenrose}). We also linked some of the big conjectures in graph theory, i.e., four color theorem (\Cref{Theorem:mainthm-four-color}), cycle double cover conjecture (\Cref{theorem:cycledoublecoverconjectureequivalence}), nowhere zero flow and edge coloring conjectures (\Cref{prop:Tutte-4-flow-equivalence}), directly to showing that certain filtered $n$-color homologies are nonzero. 

Most of the graph theory results we derived used the graded Euler characteristic or a version of the Poincar\'{e} polynomial of dimensions derived from the harmonic colorings on individual state graphs.  Look at \Cref{fig:InvariantTable} again.  The $n$-color polynomial and total face color polynomial are the invariants of this paper that most closely align to combinatorial arguments, so it is maybe not surprising that these invariants could be used to easily reproduced known and nontrivial graph theory results.  Moving forward, however, we think the real value of this paper is in the introduction of the underlying homology theories themselves: we hope that the easy-reproduction of known graph theory results serve as motivation to the reader to study the homology theories and spectral sequences. It in the spectral sequence where we think many of the solutions to the highly nontrivial problems live. 

Homology theories come with well-known tools that are helpful for decomposing the spaces studied into smaller, more manageable pieces. For example, what does excision look like for our $n$-color homologies? What are graphs with ``boundaries?'' Are there Mayer-Vietoris long exact sequences, or at least, short exact sequences that lead to interesting long exact sequences? (We know some of these exist for our homologies.) Future research will explore the homology of these ``tangles''  and how they can be glued together. 

Finally, we point out that the homology theories presented in this paper are only the second set of new homology theories predicted in \cite{BaldCohomology}. In that paper, the first author also suggested a third set of homology theories based upon the vertex polynomial.  These homology theories have been worked out and written up in \cite{BM-Vertex}. They also compute useful information about a graph. Or, if you think like a topologist, results in graph theory establish the non-vanishing of invariants of $2$-dimensional CW complexes of smooth surfaces.

For example, the {\em filtered vertex $n$-color homology}, denoted  $\widehat{VCH}_n^{*}(\Gamma)$ in \cite{BM-Vertex}, can be used to count the number of {\em perfect matchings} of the graph $G$. More specifically, if $\Gamma$ is a plane ribbon graph for an abstract planar graph $G$, then
$$\dim \widehat{VCH}_2^0(\Gamma) = 2 \cdot \# \{\mbox{perfect matchings of $G$}\}.$$
We expect that this is just the beginning of the discovery of homology theories that categorify polynomials described in \cite{BaldCohomology} that report interesting information about graphs and CW complexes of surfaces.


\end{document}